\pdfoutput=1
\documentclass{article} 
\usepackage{times}
\usepackage{geometry}

\usepackage[utf8]{inputenc} 
\usepackage[T1]{fontenc}    
\usepackage[colorlinks,citecolor=blue!70!black,linkcolor=blue!70!black,urlcolor=blue!70!black]{hyperref}       

\usepackage{authblk}

\usepackage{graphicx}
\usepackage{subfigure}
\usepackage{multirow}
\usepackage{floatrow}
\newfloatcommand{capbtabbox}{table}[][\FBwidth]

\PassOptionsToPackage{numbers}{natbib}
\usepackage{natbib}

\usepackage{xspace}
\usepackage{graphicx}

\usepackage{enumerate}
\usepackage{amsthm,amsmath,amssymb}
\usepackage{dsfont}
\usepackage{mathtools}

\usepackage[utf8]{inputenc} 
\usepackage[T1]{fontenc}    
\usepackage{hyperref}       
\usepackage{url}            
\usepackage{booktabs}       
\usepackage{amsfonts}       
\usepackage{nicefrac}       
\usepackage{microtype}      
\usepackage{xcolor}         
\usepackage{standard_definitions}

\newcommand\blfootnote[1]{%
  \begingroup
  \renewcommand\thefootnote{}\footnote{#1}%
  \addtocounter{footnote}{-1}%
  \endgroup
}

\usepackage{algorithm,algorithmicx}
\usepackage[noend]{algpseudocode}
\usepackage{setspace}
\usepackage{xspace}
\usepackage{enumitem}
\usepackage{multirow}

\usepackage{setspace}
\usepackage{amssymb}
\usepackage{graphicx}
\usepackage[normalem]{ulem}	

\newcommand{\rnk}{\mathrm{rank}}
\newcommand{\op}{\mathrm{op}}
\newcommand{\cov}{\mathrm{Cov}}
\newcommand{\inv}{\mathrm{inv}}
\newcommand{\size}{\mathrm{size}}

\newtheorem{assumption}{Assumption}

\usepackage{thmtools,thm-restate}

\usepackage{tikz}
\usetikzlibrary{trees}

\usepackage{enumitem}

\usepackage{comment}
\usepackage{color-edits}

\addauthor{ye}{ForestGreen}
\addauthor{ak}{orange}

\usepackage{xcolor,colortbl}

\title{Sparsity in Partially Controllable Linear Systems}
\begin{document}

\author{
Yonathan Efroni \quad Sham Kakade \quad Akshay Krishnamurthy \quad Cyril Zhang}

\affil{Microsoft Research\\New York, NY}
\maketitle
\blfootnote{yefroni@microsoft.com, sham.kakade@microsoft.com, akshay@cs.umass.edu, cyrilzhang@microsoft.com}
\vspace{-1cm}

\maketitle

\begin{abstract}
A fundamental concept in control theory is that of controllability, where any system state can be reached through an appropriate choice of control inputs. Indeed, a large body of classical and modern approaches are designed for controllable linear dynamical systems. However, in practice, we often encounter systems in which a large set of state variables evolve exogenously and independently of the control inputs; such systems are only \emph{partially controllable}. The focus of this work is on a large class of partially controllable linear dynamical systems, specified by an underlying sparsity pattern. Our main results establish structural conditions and finite-sample guarantees for learning to control such systems. In particular, our structural results characterize those state variables which are irrelevant for optimal control, an analysis which departs from classical control techniques. Our algorithmic results adapt techniques from high-dimensional statistics---specifically soft-thresholding and semiparametric least-squares---to exploit the underlying sparsity pattern in order to obtain finite-sample guarantees that significantly improve over those based on certainty-equivalence. We also corroborate these theoretical improvements over certainty-equivalent control
through a simulation study.
\end{abstract}

\section{Introduction}
A recurring theme in modern sequential decision making and control
applications is the presence of high-dimensional signals containing
much irrelevant information. Operating on raw signals provides
flexibility to learn much higher-quality policies than what may be
expressed using hand-engineered inputs or features, but it poses new
challenges for reinforcement learning (RL) and control. In the context of
controls, high-dimensionality inevitably leads to many state variables
that do not affect and cannot be affected by the controller inputs. Hence, these state variables are irrelevant for optimal control. 
In this work, we consider the question of how to efficiently learn to control partially controllable systems, while ignoring these irrelevant variables.


\begin{example}[Turbine Orientation~\citep{stanfel2020distributed}]\label{ex: stanfel}
Consider the problem of learning to orient
turbines in a wind farm in response to sensor measurements of wind
speed and direction. To learn a high-quality controller that can
anticipate local wind patterns, it is desirable to collect
measurements from a broad region. However geographical features such
as mountains and valleys may render some of these measurements
irrelevant for the control task, although this may not be known to the
system designer in advance. As such, we would like our controller to
efficiently learn to ignore these irrelevant sensors while relying on
the relevant ones for decision making.
\end{example}

Systems like this contain two challenging elements for learning to
control. First, a large part of the system state --- namely the wind
speed and direction at all locations --- is completely
\emph{uncontrollable}, as the wind turbines negligibly affect weather
patterns. Rather, the controller must react to these state variables
even though they cannot be controlled. Second, some of the
uncontrollable variables may be completely \emph{irrelevant}, meaning
they have no bearing on the optimal control decisions. To complicate
matters, which variables are controllable, uncontrollable, and
irrelevant must be learned, ideally in a sample-efficient manner.

In the broader literature, there are two well-studied approaches for addressing high dimensionality. One approach is through feature engineering or the use of kernel machines, while the other exploits sparsity to recover certain low-dimensional structural information. Both approaches have been utilized in the context of decision making, the former via dimension-free linear control~\citep{perdomo2021towards} and the Kernelized Nonlinear Regulator~\citep{deisenroth2011pilco,mania2020active,kakade2020information}, and the latter both in RL~\citep{agarwal2020flambe,hao2021online} and some works on continuous control~\citep{fattahi2018sample,wang2020episodic,sun2020finite}. This work contributes to the latter line of work on structure recovery in continuous control.

Our focus is on establishing non-asymptotic guarantees for learning to
control in high-dimensional partially controllable systems like the wind farm example described above. We focus our attention on the problem of learning the linear quadratic regulator (LQR) in which the majority of the state variables are irrelevant. 

\paragraph{Technical Overview.}
Deferring further details and technical motivation to subsequent
sections, we present a brief overview of the setup and
results. Consider a dynamical system of the form $x_{t+1} =
A x_t + B u_t + \xi_t$
where $x_t \in \RR^d$ is
the system state, $u_t\in \RR^{d_u}$ is the controller input, and
$\xi_t$ is a (stochastic) disturbance. The system is said to be
\emph{controllable} if, in expectation, any system state can be reached through an
 appropriate choice of a deterministic control sequence (Formally, this
 condition is equivalent to the controllability matrix being full rank. See \pref{sec: LQ with non essential noise}).
When such a condition does not hold, we call the system \emph{partially controllable}. 
For such systems, it is well known that there exists
an invertible transformation of the state variables, such that the
system can be rewritten  with dynamics of the form \citep{klamka1963controllability,sontag2013mathematical}: 
\begin{align}
A = \sbr{\begin{matrix}
A_1 & A^{\mathrm{PC}}_{12} \\
0 & A^{\mathrm{PC}}_2\\
\end{matrix}}, \qquad B = \sbr{\begin{matrix} B_1\\ 0 \end{matrix}}. \label{eq:uncontrollable_system}
\end{align}
Here the first block of coordinates corresponds to the controllable
subsystem. On the other hand, the second block of
\emph{uncontrollable} coordinates cannot be affected by the control
inputs (due to that $B_2=0$, although it can affect the controllable
subsystem (if $A^{\mathrm{PC}}_{12}\neq0$)~\citep{klamka1963controllability, zhou1996robust,
  sontag2013mathematical}.  

In this work, to capture the presence of \emph{irrelevant} state variables
that do not affect the controllable subsystem, we consider a dynamical
system that is more structured
than~\eqref{eq:uncontrollable_system}. In our setting, which we call the
partially controllable linear-quadratic (PC-LQ) control problem, the system admits the block structure: 
\begin{align}
A = \sbr{\begin{matrix}
A_1 & A_{12} & 0\\
0 & A_2 & 0\\
0 & A_{32} & A_3
\end{matrix}}, \qquad B = \sbr{\begin{matrix} B_1\\ 0 \\0 \end{matrix}}. \label{eq:pc_lqr}
\end{align}
To capture the irrelevance of state variables, our main learnability results will assume that
the underlying dynamics of the system are determined by an  $(A,B)$ in this
form, up to a permutation of the coordinates (see below for more discussion about this assumption). As we
shall see, the first two blocks make up the relevant part of the
system, while the third block of coordinates are irrelevant (in the sense
that if we condition on knowing the values of the coordinates in
blocks $1$ and $2$, then the
state variables in block $3$ provide no further information with
regards to predicting the controllable coordinates in block $1$,
which, as we shall see, is what is required for optimal control).
We are particularly interested in the high-dimensional regime where $A_1\in \mathbb{R}^{s_c\times s_c},A_2\in \mathbb{R}^{s_e\times s_e}$ and $s_c+s_e :=
s \ll d$.


\begin{table*}
\begin{center}
\begin{tabular}{|c| c | c |  }\hline
 Covariance Matrix & Estimation Algorithm &  Sample Complexity \\ \hline \hline
   Positive Definite &  Least-Squares & $\tilde{O}\rbr{\frac{poly(d,d_u)}{\epsilon}}$  \\ 
  \hline
\rowcolor{lightgray}  Diagonal &  Second-Moment Product &  $\tilde{O}\rbr{\frac{s^2+d_u s}{\epsilon}}$ \\  
 \hline
\rowcolor{lightgray}  Positive Definite  &  Semiparametric Least-Squares &  $\tilde{O}\rbr{\frac{s^2+d_u s}{\epsilon} + \frac{\sqrt{\rbr{s^2+d_u s}}d}{\epsilon^{0.5}}}$  \\
  \hline
\end{tabular}
\end{center}
\caption{
Sample complexity results for learning a near-optimal controller in the PC-LQ setting.
Our results, highlighted in gray, compare favorably with the classical least-squares/certainty-equivalent control when the relevant subsystem has dimensionality $s \ll d$. We assume the third, irrelevant, block of~\eqref{eq:pc_lqr} is stable in $L_\infty$ norm (\pref{assum: L1 stability on $A_3$}). In $\tilde{O}(\cdot)$ we only keep polynomial dependence in $\epsilon,d,s$, and $d_u$. See \pref{app: summary} for a thorough summary.  \vspace{-0.25cm}}  
\label{tab:tabbounds}
\end{table*}

\paragraph{Our Contributions.}
Our first theorem is a structural result characterizing which
state variables are irrelevant for optimal control. The result
pertains to all problems equivalent to PC-LQ control, and is proven via an
invariance argument.
When specialized to PC-LQ control, the theorem verifies that the third
block of state variables can be ignored by the optimal controller
(while it is clear that the optimal value function \emph{depends} on block
three). This structural result and our assumption that the relevant subsystem
(blocks one and two) comprises few state variables, shows that the
optimal policy is ``sparse'': it is determined by $\textrm{poly}(s)$ parameters, although neither the system dynamics $A$ nor the optimal value function are sparse matrices.

Relying on the characterization of the relevant state variables for optimal control we turn to the main contribution of our work. We derive two algorithms that
incorporate ideas from high-dimensional statistics to efficiently
estimate only the relevant parts of the system dynamics. In
\pref{tab:tabbounds}, we summarize the main results of the paper
and compare with guarantees for certainty-equivalent control.  We
study two settings that differ only in their assumptions on the
distribution of the starting state $x_0$. In the first setting
(labeled ``diagonal'' in \pref{tab:tabbounds}), we assume that
$x_0$ is sampled such that $\EE[x_0] = 0$ and $\EE[x_0x_0^\top]$ is a
diagonal matrix. In this case, we show that our algorithm learns a
near-optimal control with a \emph{nearly-dimension-free} rate: the
sample complexity scales polynomially with the sparsity $s$ and action
dimension $d_u$, but only logarithmically with the ambient dimension
$d$.

The second setting generalizes the diagonal case to only require that
$x_0$ has strictly positive definite (PD) covariance. Here our algorithm
incurs a lower order polynomial dependence on the ambient dimension~$d$. In particular, for $d^2 \leq (s^2+d_u s)/\epsilon$ this lower
order term is dominated by the leading term, which yields the same
sample complexity as in the diagonal case. In both settings, our bounds
compare quite favorably to certainty-equivalent control, which incurs
a $\textrm{poly}(d)/\epsilon$ leading order dependence. For the second setting, our algorithmic approach relies on a reduction to a semi-parametric least squares estimation~\citep{chernozhukov2016locally,chernozhukov2018double,foster2019orthogonal}. We provide a new result (see \pref{prop: semiparameteric least model sample complexity}), which might be of independent interest, for the semi-parametric least squares estimation algorithm for the linear case. 


\section{Preliminaries and Notation}

\paragraph{Linear-Quadratic Control.} 
A linear-quadratic (LQ) control problem  is specified by a tuple of matrices $L =
(A,B,Q,R)$. The state $x\in \mathbb{R}^d$ evolves according to $x_{t+1} = A x_t + B
u_t+\xi_t$ where $u\in \mathbb{R}^{d_u}$ is the input to the system and $\xi_t$ is i.i.d. noise. The cost is
given by $J(x_1,\cbr{u_t}_{t\geq 1}) = \EE\sbr{\sum_{t\geq1} x_{t}^{\top} Q x_t
+u_t R u_t}$, with $Q \succcurlyeq 0, R \succ 0$; the task is to find the policy that minimizes
$J\rbr{\cbr{u_t}_{t\geq 1}}$. It is well-known that the
optimal controller, the linear quadratic regulator (LQR), of such a system is linear in the state vector,
$u_t = K_\star x_t$, and the optimal value is $J_\star(x_1) = x_1^{\top}
P_\star x_1$, where $P_\star$ is the solution of the Riccati equation
and $K_\star = (R + B^TP_\star B)^{-1}B^{\top} P_\star A$. 
In this work, we
assume that $R=I_{d_u}$, and write $L = (A,B,Q)$ for short. This can be obtained by rotating
$u\rightarrow R^{-1/2}u$, which is valid since $R \succ 0$. We also
assume the system is stabilizable, which means that there
exists a matrix $K\in \mathbb{R}^{d_u \times d}$ such that
$\rho(A+BK)<1$, where $\rho(X) = \max\cbr{|\lambda_i(X)|}_i$ is the spectral radius of $X$ and $\lambda_i(X)$ refers to the eigenvalues. 
Furthermore, we denote $A_{\max} = \max_{i,j\in [d]}|A(i,j)|$ and
$B_{\max} = \max_{i\in[d],k\in [d_u]} |B(i,k)|$.

\paragraph{Notation.} 
We denote by $J_\star(x)$ the optimal value of the LQ problem $L$ from a
state $x$, and $K_\star(L)$ as the optimal policy of
$L$. We let $[n] = \cbr{1,..,n}$. Given two ordered lists $\Ical_1$ and
$\Ical_2$ we let $\Ical_2 / \Ical_1 =\cbr{x\in \Ical_2 | x\notin
  \Ical_1} $ denote their difference. Furthermore, given a vector
$x\in \mathbb{R}^d$ and a list $\Ical$  with entries in $[d]$ we let $x(\Ical)$ denote
the vector in $\mathbb{R}^{|\Ical|}$ which contains the coordinates of
$\Ical$, i.e., $x(\Ical) =\sbr{ \begin{matrix} x(\Ical(1)) & \cdots &
    x(\Ical(|\Ical|)) \end{matrix}}$. We denote $I_d$ as the identity
matrix of dimension $d$. The spectral/$L_2$ norm of a matrix is
denoted by $\norm{A}_{\op}$ and the Frobenius norm by $\norm{A}_F.$
We use $O(X)$ to refer to a quantity that depends on $X$ up to
constants, and denote $a\vee b = \max(a,b).$ Lastly, for a square matrix $A\in \mathbb{R}^{d\times d}$ we denote $\size(A)=d$.

\section{The Partially Controllable Linear-Quadratic Control Problem}
\label{sec: LQ with non essential noise}

In this section we formally define the LQ problem we analyze and later derive sample complexity results. We focus on an LQ problem
that consists of a partially controllable system and define an explicit notion of irrelevant state variables. Specifically, we establish that these state variables are irrelevant for optimally control this system, and, for that reason, we say the optimal controller of such a system is sparse.


A linear system is said to be \emph{partially controllable} if the
controllability matrix ${\Gcal = \sbr{\begin{matrix} B & AB &\cdots
    &A^d B \end{matrix}}}$ is not of a full rank, that is
$\rnk(\Gcal)=s_c<d$~(e.g., \citet{sontag2013mathematical}). For an LQ problem in such a system,
there exists a linear transformation $T$ that transforms the system and cost function to obtain an equivalent LQ control problem $\tilde{L} =(\tilde{A},\tilde{B},\tilde{Q})$ with the block structure of~\eqref{eq:uncontrollable_system}.
This representation reveals that the second block of coordinates $A^{\mathrm{PC}}_{2}$
cannot be affected by the controller inputs. As such, one might hope
that $A^{\mathrm{PC}}_{12}$ and $A^{\mathrm{PC}}_2$ are not required for optimal
control. Unfortunately, this is not the case, as we show in the next
simple example. Even when $\rnk(\Gcal)=1$ and $Q=I_d$, the optimal
policy may depend on the full dynamics of the uncontrollable subsystem
(see \pref{app: counterexample for general uncontrollable system} for detailed analysis).
\begin{example}[Necessity of uncontrollable dynamics for optimal control] \label{ex: non controllable part is necessary}
Let $\rho\in\mathbb{R}^{d-1},\norm{\rho}_\infty<1$,
\begin{align*}
    A_{\rho}
    =
    \sbr{
    \begin{matrix}
     1 & 1  & 1       &   \cdots     &  1   \\
     0 &\rho(1)  & 0         &  \cdots  &    0\\
      &     &  \vdots &       &      \vdots    \\
     0 &  0    &\cdots        & 0  &        \rho(d-1)
    \end{matrix}
    }, \quad 
    B
    =  
    \sbr{
    \begin{matrix}
     1 \\
     0 \\
     \vdots \\
     0
    \end{matrix}
    }, Q=I_d, 
\end{align*}
Let $L_{\rho} = (A_{\rho},B,I_d)$ be a stabilizable LQ problem. Then,  $K^\star(L_{\rho})$ is a function of $\rho$.
\end{example}

The example highlights that, without further structure, the optimal
policy may depend on $\Omega(d)$ parameters of the transition dynamics
$A$ even though only a small portion of the system is
controllable. Intuitively, this occurs because the uncontrollable
system interacts with the controllable one through matrix $A^{\mathrm{PC}}_{12}$
in~\eqref{eq:uncontrollable_system}, so the
optimal controller must plan for and react to the uncontrollable state. 

On the other hand, there are many systems in which some uncontrollable
state variables \emph{do not} affect the controllable ones whatsoever. 
The following model captures this scenario; we refer to this model as a Partially Controllable Linear Quadratic (PC-LQ) control problem.\footnote{Note that the results in this section apply to any system that is rotationally equivalent to~\eqref{eq: main body central lqr model }.}
\begin{align}
    \mathrm{(PC\text{-}LQ):}\quad A =
    \sbr{
    \begin{matrix}
     A_1 & A_{12} & 0\\
     0 & A_2 & 0 \\
     0 & A_{32} & A_3
    \end{matrix}},\
    B = 
    \sbr{\begin{matrix}
     B_1\\
     0\\
     0
    \end{matrix}
    },\ Q= I_d,\label{eq: main body central lqr model }
\end{align}
where $A_1\in \mathbb{R}^{s_c\times s_c}, A_2\in
\mathbb{R}^{s_{e}\times s_{e}}, A_3^{d-s\times d-s}, B_1\in
\mathbb{R}^{s_c\times d_u}$ and $s=s_e+s_c$. The linear system in a PC-LQ problem\footnote{For brevity, we will henceforth use ``a PC-LQ'' to stand for ``a PC-LQ control problem''.} can be decomposed into three components:
a controllable system, an uncontrollable relevant system, and an
uncontrollable irrelevant system, where the latter has no interaction
with the controllable system.  These are the first, second, and third
blocks on the diagonal, respectively. Furthermore, $A_{12}$ is a
coupling that allows the uncontrollable relevant dynamics to affect
the controllable ones, and $A_{32}$ is a coupling that allows the
uncontrollable relevant system to affect the irrelevant one. 
Observe that any LQ control problem can be written in the form of~\eqref{eq: main body central lqr model }, for some $s_c$ and $s_e$, where, for a general stable system, with no uncontrollable irrelevant dynamics,  $s_c+s_e=d$.


If the PC-LQ has $s < d$, then there are variables that are essential
for modeling the dynamics that are superfluous for optimal
control. Indeed, as we show in the next result, the optimal policy of
any PC-LQ problem does not depend on the entire transition dynamics,
specifically, the optimal controller is insensitive to the dynamics of
the uncontrollable irrelevant subsystem (blocks $A_3$ and
$A_{32}$). On the other hand, this subsystem can exhibit a very complex
temporal structure, so it is important for dynamics modeling/certainty
equivalence. Thus, even though the dynamics matrix $A$ is not a low-dimensional object,
when $s \ll d$, it is thus apt to say that the optimal policy of a
PC-LQ is low-dimensional. The following result explores two invariance properties of the optimal controller in a PC-LQ problem under cost and dynamics transformation (see~\pref{app: invariance of optimal policy of a PC-LQR} for the proof).

\begin{restatable}[Invariance of Optimal Policy for PC-LQ]{theorem}{propositionInvarianceOfOptimalPolicyWithDifferentDynamics}\label{thm: invariance of optimal policy}
Consider the following PC-LQ problems:
\begin{enumerate}
    \item Let $L_1=(A,B,I_d),L_2 = (A,B, I_{1+})$ be PC-LQ problems in stabilizable systems with similar dynamics. Let $I_{1+}$ be a diagonal matrix such that $(i)$ if $i\in [d]$ is a coordinate of the first block then $I_{1+}(i,i)=1$, and, $(ii)$ for any other $i\in [d]$, $I_{1+}(i,i)\in \cbr{0,1}$.
   \item Let $L_1=(A,B,I_d),L_2=(\bar{A},B,I_d)$ be PC-LQ problems in stabilizable systems such that
        \begin{align*}
            A =
            \sbr{
            \begin{matrix}
             A_1 & A_{12} & 0\\
             0 & A_2 & 0 \\
             0 & A_{32} & A_3
            \end{matrix}},\ \ 
            \bar{A} =
            \sbr{
            \begin{matrix}
             A_1 & A_{12} & 0\\
             0 & A_2 & 0 \\
             0 & \bar{A}_{32} & \bar{A}_3
            \end{matrix}},\ \ 
            B = 
            \sbr{\begin{matrix}
             B_1\\
             0\\
             0
            \end{matrix}
            }.
        \end{align*}
\end{enumerate}
Then, for both $(1)$ and $(2)$, the optimal policy of $L_1$ and $L_2$ is equal, i.e., $K^*(L_1)=K^*(L_2)$.
\end{restatable}

Of course, since $Q=I_d$, the optimal value functions for $L_1$ and
$L_2$ will -- in general -- be quite different. Since the uncontrollable blocks $A_3$ and $A_{32}$ of a PC-LQ are irrelevant to optimally control it, we refer to both of the block as the \emph{irrelevant blocks} from this point onward. This highlights the fact that the LQR of a PC-LQ is sparse: it does not depends on the parameters of the irrelevant blocks.

\subsection{Characterization via controllability and the relevant disturbances matrices}
A natural question is to understand when a system is equivalent to a PC-LQ with an
irrelevant subsystem. The next result provides a
characterization of PC-LQ in terms of the controllability matrix and
a new object that we call the \emph{relevant disturbances matrix}.
Recall that any LQ problem with controllability index $s_c$ can be
rotated into the form~\eqref{eq:uncontrollable_system}.  For brevity, denote $X_{12} = A_{12}^{\mathrm{PC}}$ and $X_{2} = A_{2}^{\mathrm{PC}}$. Let the relevant disturbances matrix using this representation be
\begin{align}
    \Rcal\Dcal= \sbr{\begin{matrix} X^{\top}_{12} & X^T_2 X_{12}^{\top} & \cdots & (X^T_2)^{d-s_c} X_{12}^{\top} \end{matrix}}. \label{eq: definition relevant disturbances matrix}
\end{align}
Then, we have the following structural characterization of a PC-LQ through the controllability and relevant disturbances Krylov matrices (see \pref{app: structural properties of PC-LQ} for the proof). 



\begin{restatable}[Controllability characterization of PC-LQ]{proposition}{propositionExistenceOfNonEssentialDisturbances}\label{prop: necessary and sufficient for existence of non essential disturbances}
If $L$ has controllability index $s_c$ and $\textrm{rank}(\Rcal\Dcal) = s_e$ then $L = (A,B,I_d)$ is rotationally equivalent to~\eqref{eq: main body central lqr model }.
\end{restatable}




\subsection{Characterization via minimal invariant subspaces}\label{sec: invariant subspaces}


We next characterize a PC-LQ via the notion of minimal invariant subspaces. This characterization is more useful for our subsequent algorithmic development. Minimal invariant subspaces (w.r.t., an initial subspace) are formalized in the next definition.
\begin{definition}[Minimal invariant subspace w.r.t. another subspace, e.g.,~\cite{basile1992controlled}]\label{def: minimal invariant subspace}
Let $K$ be a subspace and $A\in \mathbb{R}^{n\times n}$.  Subspace $V$ is an \emph{invariant subspace} of $A$ w.r.t. $K$ if
    $(i)$, $K\subseteq V$, and
    $(ii)$ $AV \subset V$. $V$ is the \emph{minimal invariant subspace} of $A$ w.r.t. $K$ if $(i)$ and $(ii)$ hold and  $V$ is the subspace with the smallest dimension that satisfies both $(i)$ and $(ii)$. 
\end{definition}
That is, the minimal invariant subspace of $A$ w.r.t. $K$ is the
smallest subspace that contains $K$ and is closed/invariant under the
action of $A$, meaning that $Av \subset V$ for any $v \in V$. In
\pref{app: invariant subspace and minimal invarinat subspace} we show that the minimal invariant subspace is always
unique, and, thus, it is always well defined. 



The next result shows that the first and second blocks of a
partially controllable system can be expressed in terms of two minimal invariant
subspaces. This yields a simple algebraic characterization of the
relevant components of the system, which we will use to develop
algorithms (see \pref{app: structural properties of PC-LQ} for the proof).
\begin{restatable}[PC-LQ and Minimal Invariant Subspaces]{proposition}{PropFromLQRUDtoMIS}\label{prop: lqr rud and minimal invariant subspaces}
An LQ problem is equivalent to PC-LQ~\eqref{eq: main body central lqr model } if and only if there exist projection matrices with $\rnk(P_B)\leq \rnk(P_c)\leq \rnk(P_r)$ where
\begin{enumerate}
    \item $P_c$ is an invariant subspace of $A$ w.r.t. $P_B$ and $\rnk\rbr{P_c} = s_c$,
    \item $P_{r}$ is an invariant subspace of $(I-P_c)A^{\top}$ w.r.t.~$P_c$ and $\rnk\rbr{P_{r}} = s_c+s_e=s$,
\end{enumerate}
such that $A,B$ can be written as
\begin{align*}
    A = P_c A P_c + P_{r}A (P_{r}-P_{c}) + (I-P_{r})A(I-P_c),\quad  B = P_B B,
\end{align*}
Furthermore, the subspaces $P_c$ and $P_r$ are the minimal invariant subspaces if and only if the controllability matrix is of rank $s_c$ and the relevant disturbances matrix is of rank $s_e$.
\end{restatable}

With the above notation, the subspace $P_c$ represents the first block
of~\eqref{eq: main body central lqr model }, and $P_r$ represents
the first two blocks which are generally required for optimally control a PC-LQ. The matrix $(I-P_r)A(I-P_c)$
represents the irrelevant blocks of a PC-LQ which we can
safely ignore by \pref{thm: invariance of optimal policy}.

\section{Learning Sparse LQRs in Partially Controllable Systems}\label{sec: learning sparse lqr}

\begin{algorithm}[t]
\caption{Learning Optimal Policy of PC-LQ}
\label{alg: control pc lqr samples}
\begin{algorithmic}[1]
\State {\bf Require:} $\epsilon,\delta>0$, $\mathrm{STh}_{\epsilon}(x) = \ind\cbr{|x|>\epsilon}(x-\mathrm{sign}(x)\epsilon) $ 
\State Get $\widehat{A}$ and $\widehat{B}$, an $(\epsilon,\delta)$ element-wise estimates of $A$ and $B$, respectively    
\State Soft threshold  the empirical estimates element-wise, $\bar{B}= \mathrm{Th}_{\epsilon}(\widehat{B}), \bar{A}= \mathrm{Th}_{\epsilon}(\widehat{A})$
\State {\bf Return: } Optimal policy of $\bar{L}=(\bar{A},\bar{B},I)$ 
\end{algorithmic}
\end{algorithm}

We now turn to our main question and focus on the learnability of optimal policy in PC-LQ. We assume that the model is transformed to be in the
form of~\eqref{eq: main body central lqr model }, so it is
\emph{axis-aligned} up to permutations, i.e., the irrelevant state variables are not a-priori known to the algorithm designer. 
We further assume  $\size(A_1) + \size(A_2)=s_e+s_c = s \ll d$. Of course, as we
have discussed, the dynamics matrix $A$ itself \emph{is not sparse}, but the optimal policy of such system, the LQR, is sparse. \pref{thm: invariance of optimal policy} establishes the LQR depends only on $O(poly(s))$ parameters. Thus, we hope for sample complexity guarantees that
scale primarily with the intrinsic dimension $s$, rather than the
ambient dimension $d$.

\begin{remark}[Axis-aligned assumption]
The axis-aligned assumption is a natural extension of the sparsity assumption made in sparse regression literature (e.g.,~\cite{wainwright2019high}, Chapter 7). In control problems, this assumption may be satisfied when the state variables $x$ arise from physical measurements. In this case, axis-alignment corresponds to negligible coupling between different state variables that represent measurements in different locations (as elaborated in~\pref{ex: stanfel}).  Furthermore, all the results generalize naturally when the rotation for which the LQ problem can be written as~\eqref{eq: main body central lqr model } is known. We comment that asymptotic dimension-free bounds for system identifications without the axis-aligned assumptions are impossible, due to the need to learn the rotation matrix. We leave it as an interesting future question to study whether asymptotic dimension-free bounds are possible for general PC-LQ problems.  
\end{remark}

By \pref{prop: lqr rud and minimal invariant subspaces} the
optimal controller is insensitive to errors in $(I-P_r)A(I-P_c)$,
corresponding to block 3 of the dynamics matrix. However, to take
advantage of this, we must first identify the zero pattern of the
matrix $A$. More formally, we seek estimates $(\bar{A},\bar{B})$ of the
dynamics satisfying the following \emph{no false positive} property:
\begin{align}
    \forall i,j\in [d],\ k\in [d_u]: \ A(i,j)=0 \Rightarrow \bar{A}(i,j)=0,\ \mathrm{and}\ B(i,k)=0 \Rightarrow \bar{B}(i,k)=0. \label{eq: main paper goal for estimate}
\end{align}
Indeed, in the presence of such a condition, we can ensure that there
is no interaction between the relevant and irrelevant parts of the
system in the \emph{estimated model}, so that $(\bar{A},\bar{B})$ is a PC-LQ
with a similar block structure to the true dynamics.


A natural way to obtain estimates of $(A,B)$ that satisfy
~\eqref{eq: main paper goal for estimate} is to perform
\emph{soft-thresholding} on an entrywise accurate initial
estimate. Note that the soft-thresholding operation does not introduce
much additional error.  Since many options are available for obtaining
the initial estimate, we formalize this via an oracle that we call the
\emph{entrywise estimate}. In \pref{sec: estimation oracle
  sample complexity}, we instantiate this oracle with two different
procedures and analyze their sample complexity.


\begin{definition}[Entrywise estimator]
\label{def: dynamical system estimation oracle}
We say that $\widehat{X}$ is an $(\epsilon,\delta)$ entrywise estimator of a matrix $X\in \mathbb{R}^{d_1\times d_2}$ if with probability at least $1-\delta$ we have $\max_{i,j} |\widehat{X}(i,j) - X(i,j)| \leq \epsilon$.
\end{definition}

Given access to such an oracle, \pref{alg: control pc lqr
  samples} learns an optimal policy in a PC-LQ problem. First, it
estimates $(A,B)$ via the entrywise estimator, to obtain
$(\widehat{A},\widehat{B})$. Second, it applies a soft-thresholding to
these estimates to get $\bar{A},\bar{B}$. Finally, it returns the
optimal policy of the LQ problem $\bar{L} =
(\bar{A},\bar{B},I)$. 

For the analysis, we require a technical assumption on the $L_{\infty}$ stability
of the irrelevant subsystem~$A_3$.


\begin{assumption}[$L_\infty$-stability of irrelevant dynamics]\label{assum: L1 stability on $A_3$}
$A_3$ is $L_\infty$ stable: $\max_{i} \sum_{j}|A_3(i,j)| = \norm{A_3}_\infty<1$.
\end{assumption}

In addition, our guarantee scales with the operator norm of the
optimal value function for the \emph{relevant} subsystem only. Formally, let
$L_{1:2} = (A_{1:2},B_{1:2},I_{1:2})$ be an $s$-dimensional LQ problem
defined by the first two blocks of~\eqref{eq: main body central lqr
  model } and let $P_{\star,1:2}$ be the solution to the Ricatti
equation for this system. The guarantee is given as follows (see \pref{app: from PC-LQ to MIS} for the proof).


\begin{restatable}[Learning the PC-LQR]{theorem}{LQRUDviaMIS}\label{thm: PC-LQ main result}
Fix $\epsilon,\delta>0$. Assume access to an entrywise estimator of $(A,B)$ with parameters $(\sqrt{\epsilon/\rbr{2s(s+d_u)}},\delta)$, and that~\pref{assum: L1 stability on $A_3$} holds.  Then, if $\epsilon<1/ \norm{P_{\star,1:2}}^{10}_{\op}$, with probability greater than $1-\delta$ \pref{alg: control pc lqr samples} outputs a policy $\bar{K}$ such that $$J_\star(\bar{K}) \leq J_\star + O(\norm{P_{\star,1:2}}_{\op}^8 \epsilon) .$$
\end{restatable}
To prove this result we utilize the machinery of \pref{thm: invariance of optimal policy}, \pref{prop: lqr rud and minimal invariant subspaces}, the perturbation result of~\cite{simchowitz2020naive}, and the no-false positive property of the estimated model. 

\section{Sample Complexity for Entrywise Estimation}\label{sec: estimation oracle sample complexity}

We now instantiate two entrywise estimators and establish their sample
complexity guarantees in two settings. First, when the initial state
$x_0$ has a diagonal covariance matrix, we show that a simple second-moment estimator suffices. In the more general setting where the
initial state $x_0$ has PD covariance, we develop an
estimator based on semiparametric least-squares. The first estimator
has better sample complexity guarantees, while the second estimator
is more general.


\subsection{Diagonal covariance matrix}\label{sec: diagonal matrix assumption}
When the initial state $x_0$ has a diagonal covariance matrix, we analyze a simple second-moment estimator. Specifically we estimate the model with
\begin{align}
\widehat{A} = \frac{1}{N\sigma_0^2}\sum_{n} x_{1,n} x_{0,n}^{\top}, \quad \textrm{and} \quad \widehat{B} =\frac{1}{N}\sum_{n} x_{1,n} u_{0,n}^{\top}, \label{eq:second_moment_estimator}
\end{align}
given $N$ partial trajectories $\{(x_{0,i}, u_{0,i}, x_{1,i})\}_{i=1}^N$ where $u_{0,i} \sim \mathcal{N}(0, I_{d_u})$. 
For this estimator we prove the following (see \pref{app: disc estimation via thresholding} for a proof):

\begin{restatable}[Entrywise estimation with diagonal covariance] {proposition}{theoremDiscViaThreshold}\label{prop: elementwise convergence of second-moment based estimation}
Assume that $x_0\sim \Ncal(0,\sigma_0 I_d)$ and that~\pref{assum: L1 stability on $A_3$} holds. Denote $\sigma_{\mathrm{eff}} = 1 + A_{\max}\sqrt{s}+ \rbr{1+B_{\max}\sqrt{d_u}}\rbr{(\sigma/\sigma_0) \vee \sigma}.$ Then, given  $N = O\rbr{\frac{\log\rbr{\frac{d}{\delta}}\sigma^2_{\mathrm{eff}}}{\epsilon^2}}$ samples~\eqref{eq:second_moment_estimator} is an entrywise estimator of $(A,B)$ with parameters $(\epsilon,\delta)$.
\end{restatable}

Combining with \pref{thm: PC-LQ main result}, we obtain the first shaded row of \pref{tab:tabbounds}.

\subsection{Positive definite covariance matrix} \label{sec: PD covariance assumption}
For the second setting, we only assume that the covariance of $x_0$ is
PD. This, more general setting, is of importance since  the stationary measure of a policy may be quite complex, and, in particular, it may induce correlations between the irrelevant and relevant blocks (see~\pref{app: structure of covariance matrix} for further discussion on the need to handle general covariance matrices).  In this case, the least-squares estimator of $A$
yields a guarantee in the Frobenius norm, which can be translated into
an entrywise estimate. However, the sample complexity of this approach
scales as $\textrm{poly}(d)/\epsilon^2$, which is too large for our
purposes.  Instead of using classical least-squares, our approach is
based on a reduction to \emph{semiparametric least-squares}~ \citep{chernozhukov2016locally,chernozhukov2018plug,chernozhukov2018double,foster2019orthogonal}, which, as
we will see, results in a sample complexity of $1/\epsilon^2 +
d/\epsilon$ for entrywise estimation. Observe that here the ambient
dimension only appears in the lower order term.

The main idea is as follows: Suppose we wish to learn the $(i,j)$-th
entry of $A$ and assume we have $(x_1,x_0)$ sample pairs from the
model $x_1 = A x_0 + \xi$ where $\xi$ is a zero-mean $\sigma$
sub-gaussian vector. Then, for any $i \in [d]$,
\begin{align}
    x_{1}(i) = A(i,j) x_0(j) + \inner{A(i, [d]/ j)}{x_0([d]/j)} +\xi_i. \label{eq: element wise estimate to semi parametric}
\end{align}
If the first and second terms on the RHS were uncorrelated, then a
linear regression of $x_1(i)$ onto $x_0(j)$ would yield an unbiased
estimate of $A(i,j)$. Unfortunately, these two terms are correlated
under our assumptions, so least-squares may be biased. To remedy this,
we attempt to decorrelate the two terms using a two-stage regression
procedure. The first stage involves high dimensional regression
problems, but these errors ultimately only appear in the lower order terms.



Since our results for this problem may be of independent interest, we
next study a generalization of the model in~\eqref{eq: element wise
  estimate to semi parametric} and explain the estimator in detail. As
a corollary, we obtain a sample complexity guarantee for the entrywise
estimator for the PC-LQ.

\begin{algorithm}[t]
\caption{Semiparametric Least Squares}\label{alg: semi parametric regression main paper}
\begin{algorithmic}[1]
\State {\bf Require:} Dataset $\Dcal = \cbr{(x_{1,n} , x_{0,n})}_{n=1}^{2N}$ row and column indices $i,j\in [d]$
\State Reduction to semiparametric LS: $\Dcal_{SP} = \cbr{(y_n, z_{1,n} , z_{2,n})}_{n=1}^{N}$ where
\begin{align*}
    y_n = x_{1,n}(i), \ z_{1,n} = x_{0,n}(j),\ z_{2,n} = x_{0,n}([d] / j).
\end{align*}
\State Estimate cross correlation $\widehat{L} = \rbr{\sum_{n=1}^{N} z_{1,n}z_{2,n}^{\top}} \rbr{\sum_{n=1}^{N} z_{2,n}z_{2,n}^{\top}}^\dagger $
\State Estimate conditional output
$
    \widehat{c} = \rbr{\sum_{n=1}^{N} z_{2,n}z_{2,n}^{\top}}^\dagger \rbr{\sum_{n=1}^{N} y_n z_{2,n}} 
$
\State Set 
$
\widehat{A}(i,j) =\rbr{\sum_{n=N+1}^{2N} (z_{1,n} -  \widehat{L}z_{2,n})(z_{1,n} -  \widehat{L}z_{2,n})^{\top}}^\dagger \rbr{\sum_{n=N+1}^{2N} \rbr{y_n - \inner{\widehat{c}}{z_{2,n}}} (z_{1,n} -  \widehat{L}z_{2,n})}
$
\State {\bf Output:} $\widehat{A}(i,j)$
\end{algorithmic}
\end{algorithm}

\paragraph{Semiparametric least-squares.} 
As a generalization of~\eqref{eq: element wise estimate to semi parametric}, assume that $x\sim \Ncal(0, \Sigma)$ where $\lambda_{\min}(\Sigma)>0$ and $x\in \mathbb{R}^d$ and let
\begin{align}
    y = \inner{w_\star}{x_1} + \inner{e_\star}{x_2} +\xi \label{eq: semi parametric LS main paper}
\end{align}
where $w_\star,x_1\in \mathbb{R}^{d_w}, e_\star,x_2\in \mathbb{R}^{d_e}$, $x = [x_1, x_2]^{\top}$ and $\xi$ is $\sigma$ sub-Gaussian. 
By observing tuples sampled from this model $\cbr{y_n,x_{1,n}, x_{2,n}}_{n=1}^N$ we wish to estimate only $w_\star$. 
To do so, we first estimate $L_\star\in \mathbb{R}^{d_w \times d_w}$ and $c_\star\in \mathbb{R}^{d_e}$, that relate $x_2$ to the conditional expectation $\EE[x_1| x_2]$ and $\EE[y| x_2]$, with $N$ samples via standard least-squares. Due to the model Gaussian assumption, it holds that
\begin{align*}
    & \EE[x_1| x_2]  = L_\star x_2,
    &\EE[y| x_2] = c^T_\star x_2.
\end{align*}
When access to exact estimates of these quantities is given, we show in \pref{app: semi parametric ls linear model}, that the model~\eqref{eq: semi
  parametric LS main paper} can be `orthogonalized' and written as
\begin{align*}
    y = \inner{w_\star}{x_1- L_\star x_2} + \inner{c_\star}{x_2},
\end{align*}
where $\EE[(x_1- L_\star x_2)x_2^{\top}]=0$, so that the two terms on the right hand side are uncorrelated, unlike in the original model. 
Thus, given estimates $\widehat{L}_N,\widehat{c}_N$, we regress $y-\inner{\widehat{c}_N}{x_2}$ onto $(x_1 - \widehat{L}_N x_2)$ to get an estimate of $w_\star$. 
See \pref{alg: semi parametric regression main paper} for a description of the algorithm. 
In the next result, we show that this estimator has leading order error scaling with $d_w$ and only a lower order error term scaling with $d_e$. Furthermore, we get a minimal dependence in $\lambda_{\min}(\Sigma)$, with similar scaling as in usual OLS analysis~\citep{hsu2012random} (see \pref{app: finite sample analysis, semi parametric LS} for proof). 

\begin{restatable}[Semiparametric Least-Squares]{proposition}{theoremDiscViaNuisanceLS}\label{prop: semiparameteric least model sample complexity}
Let $\delta\in (0,e^{-1})$.
Consider model~\eqref{eq: semi parametric LS main paper} and assume that $\Sigma$ is PD. Denote $\sigma_c^2=  \norm{w_\star}_{\Sigma/\Sigma_{2}}^2 + \sigma^2$.
Then, if  ${N\geq O\rbr{\rbr{\sigma^2_c/\lambda_{\min}(\Sigma)}\vee 1} d d_w \log\rbr{\frac{d}{\delta}}}$, 
with probability $1-\delta$, the semiparametric LS estimator $\hat{w}$ of $w_\star$ satisfies
$$\norm{w_\star - \widehat{w}}_2\leq O\rbr{\frac{1}{\sqrt{\lambda_{\min}(\Sigma)}}\rbr{ \sqrt{\frac{\sigma^2 d_w\log\rbr{\frac{1}{\delta}}}{N}} + \frac{ \rbr{\sigma^2_c \vee \sigma_c} d d_w \log\rbr{\frac{d_w}{\delta}}}{N}}}.$$
\end{restatable}

Returning to the PC-LQ setting, we obtain an entrywise estimator for
$A$ by applying the semiparametric LS approach on each pair $(i,j) \in
[d]^2$. To estimate $B$, since we can sample $u_0$ with a diagonal
covariance, we can apply the results for the diagonal covariance
case. We summarize the sample complexity for entrywise estimation in
the next corollary (see \pref{app: disc estimation via ls with nuisance parameter} for proof).


\begin{restatable}[Element-wise Estimate, PD Covariance]{corollary}{corrDiscViaNuisanceLS}\label{corr: corollary: semi parametric estimate for PD covariance}
Assume $x_0\sim \Ncal(0,\Sigma)$ and that ${\lambda_{\min}(\Sigma)>0}$.  Denote $\sigma_c^2=  A_{\max}^2 \lambda_{\max}(\Sigma)+\sigma^2$.
Then, if $N\geq O\rbr{\rbr{\sigma^2_c/ \lambda_{\min}(\Sigma) \vee 1}d\log\rbr{\frac{d}{\delta}}}$,
and  \\
${N = O\rbr{\frac{\sigma^2\log\rbr{\frac{d}{\delta}}}{\epsilon^2\lambda_{\min}(\Sigma)} + \frac{d\rbr{\sigma_c^2\vee \sigma_c}\log\rbr{\frac{d}{\delta}}}{\epsilon \sqrt{\lambda_{\min}(\Sigma)}}}}$,
then the semiparametric LS yields an entrywise estimate of $A$ with parameters $(\epsilon,\delta)$. 
\end{restatable}

Combining with \pref{thm: PC-LQ main result}, we obtain the second shaded row of \pref{tab:tabbounds}.

\section{Experiments}\label{sec: experiments}
We present a proof-of-concept empirical study, to demonstrate the end-to-end statistical advantages of leveraging sparsity in the LQR of a PC-LQ. We generate synthetic systems with marginally stable controllable blocks; the task is to learn a stabilizing controller $K$ (such that $\rho(A + BK) < 1$) from finite samples, in the presence of many irrelevant state coordinates (letting $d$ increase, while holding $s$ and $d_u$ constant). We compare \pref{alg: control pc lqr samples} with the certainty-equivalent controller obtained from the ordinary least-squares (OLS) estimator for the system's dynamics.

Synthetic PC-LQ problems were generated with i.i.d. standard Gaussian entries (for all $A_1, A_2, A_3, A_{12}, A_{32}, B_1$); the diagonal blocks were normalized by their top singular values so that $\rho(A_1) = 1$, and $\rho(A_2) = \rho(A_3) = 0.9$. We computed $\bar A$ from the minimum-norm $N$-sample OLS estimator, as well as the soft-thresholded semiparametric least-squares estimator from \pref{alg: control pc lqr samples} (with $\epsilon = 0.1$), and obtained certainty-equivalent controllers $\bar K$ by solving the Riccati equation with $\bar L := (\bar A, B)$. Over 100 trials in each setting, we recorded the fraction of times $\bar K$ stabilized the system ($\rho(A + B\bar K) < 1$, and $J(\bar K) \leq 1.1 \cdot J(K^*)$).

\pref{fig:synthetic-plot} summarizes our findings: keeping the relevant dimensions fixed ($s_c = s_e = 5, d_u = 1$) and allowing $d$ to grow, the sample complexity of stabilizing the system exhibits a far milder dependence on the ambient dimension $d$ when using our estimator. A complete description of the experimental protocol is given in \pref{app: appendix-experiments}.

\begin{figure}
    \centering
    \includegraphics[width=1.0\linewidth]{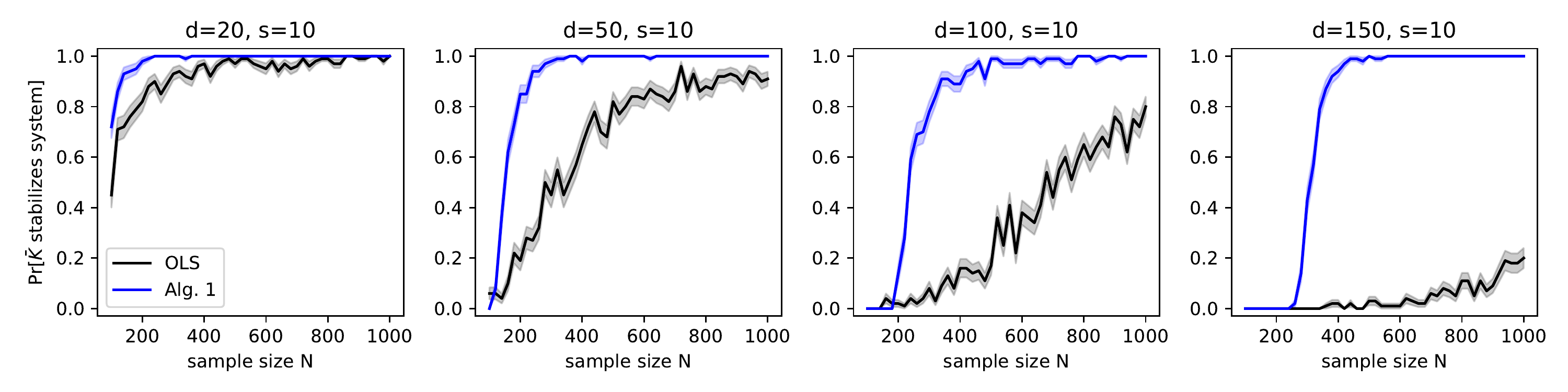}
    \vspace{-0.25cm}
    \caption{
     Empirical comparison of \pref{alg: control pc lqr samples} with OLS for stabilizing a marginally stable PC-LQ. As the number of irrelevant features increases, the sample complexity of the sparsity-leveraging estimator grows much more slowly. Success frequencies (with standard deviations from the normal approximation) are measured over 100 trials.\vspace{-0.35cm}}
    \label{fig:synthetic-plot}
\end{figure}

\section{Related Work}

\paragraph{Partial controllability in control theory.} The notion of controllability and partial controllability has been well studied from many different aspects in both  classical and modern control theory \quad \cite{kalman1963mathematical,lin1974structural,glover1976characterization,jurdjevic1978controllability,zhou1996robust,bashirov2007partial,sontag2013mathematical}, as well as, the relation between controllability and invariant subspaces~\cite{klamka1963controllability,basile1992controlled}.
In~\pref{sec: LQ with non essential noise}, we  characterize which parts of a PC-LQ  are not needed for optimal control. To the best of our knowledge, such characterization  does not exist in previous literature. One may interpret the results of~\pref{sec: LQ with non essential noise} as an extension of Kalman's canonical decomposition. That is, we further decompose the uncontrollable and observable system (see~\citet{kalman1963mathematical}, Page 165) into relevant and irrelevant parts for optimal control. 


\paragraph{Structural results in LQ.} Recently, there has been a surge of interest in the learnability of LQ~\cite{abbasi2011regret,dean2019sample,sarkar2019near,cohen2019learning,mania2019certainty,simchowitz2020naive,cassel2020logarithmic,tsiamis2021linear}. However, learning in the presence of structural properties of an LQ has been, to large extent, unexplored. Closely related to our work is the problem studied in~\citep{fattahi2018sample,fattahi2019learning}. There, the authors considered an LQ problem in which the dynamics itself has a sparse structure.  Specifically, the dynamics was assumed to have some sparse block structure such that all elements in each block are simultaneously zero or non-zero. We do not put any such restriction on a PC-LQ. Moreover, in our case, the transition matrix $A$ need not be a sparse matrix, and may have $\Omega(d^2)$ non-zero elements. The sparsity utilized in our work is \emph{sparsity of the optimal controller} and not of the dynamics itself. We also comment that in~\citep{fattahi2018sample,fattahi2019learning} additional assumptions were made,  which are not satisfied in our setting. First, the authors assume a mutual-incoherence condition on the covariance matrix. Additionally, it is assumed that $A(i,j),B(i,j)\geq \gamma>0$, i.e., that there is a minimal value for the entries of the dynamics. These assumptions are crucial for identification of the non-zero entries; assumptions we do not make in this work  (see~\pref{app: structure of covariance matrix} for further discussion on the structure of the covariance matrix in our setting). That is, we recover a near optimal policy without the need to recover the true block structure.  

Another related work is the work of~\cite{wang2020episodic}, where the authors assumed the dynamics is of low rank and fully controllable. We do not make such an assumption and allow for uncontrollable part to affect the controllable part. Lastly, in~\cite{sun2020finite}, the authors analyzed system identification via low-rank Hankel matrix estimation. Observe that Hankel based techniques only enable the recovery of the controllable parts of the system, as they are based on a function of $A^n B$. However, to optimally control a stable system, knowledge of the relevant uncontrollable process is also needed (see~\pref{ex: non controllable part is necessary}).


\section{Summary and Future Work}

In this work, we studied structural and learnability aspects of the PC-LQ. We characterized an invariance property of the LQR of a PC-LQ. This revealed that the optimal controller of such systems is, in fact, a low-dimensional object. Then, given an entrywise estimator,  we showed  that the sample complexity of learning an axis-aligned PC-LQ has only a mild dependence on the ambient dimension, scaling primarily with the dimensionality/sparsity of the optimal controller.

The results presented in this work opens several interesting future research avenues. First, we believe it would be interesting to study additional invariance properties of optimal policies of other control and RL problems. As stressed in this work, invariances of the optimal controller can yield statistical improvements for learning in such models. More broadly, is there a general way to characterize such invariances? Second, in this work, we assumed the PC-LQ model is sparse, or, axis-aligned. A natural question would be to study the learnability of such a model when the system is not axis-aligned, and understand the nature of possible sample complexity improvements in such systems? Lastly, extending our results to a single trajectory setting is of interest, and may require developing new tools for semiparametric least-squares analysis.

\section*{Acknowledgments}
YE is partially supported by the Viterbi scholarship, Technion.

\bibliographystyle{apalike}
\bibliography{citations}

\begin{thebibliography}{}

\bibitem[Abbasi-Yadkori and Szepesv{\'a}ri, 2011]{abbasi2011regret}
Abbasi-Yadkori, Y. and Szepesv{\'a}ri, C. (2011).
\newblock Regret bounds for the adaptive control of linear quadratic systems.
\newblock In {\em Proceedings of the 24th Annual Conference on Learning
  Theory}. JMLR Workshop and Conference Proceedings.

\bibitem[Agarwal et~al., 2020]{agarwal2020flambe}
Agarwal, A., Kakade, S., Krishnamurthy, A., and Sun, W. (2020).
\newblock Flambe: Structural complexity and representation learning of low rank
  mdps.
\newblock {\em Advances in Neural Information Processing Systems}.

\bibitem[Bashirov et~al., 2007]{bashirov2007partial}
Bashirov, A.~E., Mahmudov, N., {\c{S}}em{\i}, N., and Et{\i}kan, H. (2007).
\newblock Partial controllability concepts.
\newblock {\em International Journal of Control}.

\bibitem[Basile and Marro, 1992]{basile1992controlled}
Basile, G. and Marro, G. (1992).
\newblock {\em Controlled and conditioned invariants in linear system theory}.
\newblock Prentice Hall Englewood Cliffs, NJ.

\bibitem[Cassel et~al., 2020]{cassel2020logarithmic}
Cassel, A., Cohen, A., and Koren, T. (2020).
\newblock Logarithmic regret for learning linear quadratic regulators
  efficiently.
\newblock In {\em International Conference on Machine Learning}, pages
  1328--1337. PMLR.

\bibitem[Chernozhukov et~al., 2018a]{chernozhukov2018double}
Chernozhukov, V., Chetverikov, D., Demirer, M., Duflo, E., Hansen, C., Newey,
  W., and Robins, J. (2018a).
\newblock Double/debiased machine learning for treatment and structural
  parameters.

\bibitem[Chernozhukov et~al., 2016]{chernozhukov2016locally}
Chernozhukov, V., Escanciano, J.~C., Ichimura, H., Newey, W.~K., and Robins,
  J.~M. (2016).
\newblock Locally robust semiparametric estimation.
\newblock {\em arXiv preprint arXiv:1608.00033}.

\bibitem[Chernozhukov et~al., 2018b]{chernozhukov2018plug}
Chernozhukov, V., Nekipelov, D.~N., Semenova, V., and Syrgkanis, V. (2018b).
\newblock Plug-in regularized estimation of high-dimensional parameters in
  nonlinear semiparametric models.
\newblock Technical report, cemmap working paper.

\bibitem[Cohen et~al., 2019]{cohen2019learning}
Cohen, A., Koren, T., and Mansour, Y. (2019).
\newblock Learning linear-quadratic regulators efficiently with only $\sqrt{T}$
  regret.
\newblock In {\em International Conference on Machine Learning}. PMLR.

\bibitem[Dean et~al., 2019]{dean2019sample}
Dean, S., Mania, H., Matni, N., Recht, B., and Tu, S. (2019).
\newblock On the sample complexity of the linear quadratic regulator.
\newblock {\em Foundations of Computational Mathematics}.

\bibitem[Deisenroth and Rasmussen, 2011]{deisenroth2011pilco}
Deisenroth, M. and Rasmussen, C.~E. (2011).
\newblock Pilco: A model-based and data-efficient approach to policy search.
\newblock In {\em Proceedings of the 28th International Conference on machine
  learning (ICML-11)}. Citeseer.

\bibitem[Fattahi et~al., 2019]{fattahi2019learning}
Fattahi, S., Matni, N., and Sojoudi, S. (2019).
\newblock Learning sparse dynamical systems from a single sample trajectory.
\newblock In {\em 2019 IEEE 58th Conference on Decision and Control (CDC)}.
  IEEE.

\bibitem[Fattahi and Sojoudi, 2018]{fattahi2018sample}
Fattahi, S. and Sojoudi, S. (2018).
\newblock Sample complexity of sparse system identification problem.
\newblock {\em arXiv preprint arXiv:1803.07753}.

\bibitem[Foster and Syrgkanis, 2019]{foster2019orthogonal}
Foster, D.~J. and Syrgkanis, V. (2019).
\newblock Orthogonal statistical learning.
\newblock {\em arXiv preprint arXiv:1901.09036}.

\bibitem[Glover and Silverman, 1976]{glover1976characterization}
Glover, K. and Silverman, L. (1976).
\newblock Characterization of structural controllability.
\newblock {\em IEEE Transactions on Automatic control}.

\bibitem[Hao et~al., 2021]{hao2021online}
Hao, B., Lattimore, T., Szepesv{\'a}ri, C., and Wang, M. (2021).
\newblock Online sparse reinforcement learning.
\newblock In {\em International Conference on Artificial Intelligence and
  Statistics}. PMLR.

\bibitem[Hsu et~al., 2012a]{hsu2012tail}
Hsu, D., Kakade, S., Zhang, T., et~al. (2012a).
\newblock A tail inequality for quadratic forms of subgaussian random vectors.
\newblock {\em Electronic Communications in Probability}.

\bibitem[Hsu et~al., 2012b]{hsu2012random}
Hsu, D., Kakade, S.~M., and Zhang, T. (2012b).
\newblock Random design analysis of ridge regression.
\newblock In {\em Conference on learning theory}. JMLR Workshop and Conference
  Proceedings.

\bibitem[Jurdjevic and Quinn, 1978]{jurdjevic1978controllability}
Jurdjevic, V. and Quinn, J.~P. (1978).
\newblock Controllability and stability.
\newblock {\em Journal of differential equations}.

\bibitem[Kakade et~al., 2020]{kakade2020information}
Kakade, S., Krishnamurthy, A., Lowrey, K., Ohnishi, M., and Sun, W. (2020).
\newblock Information theoretic regret bounds for online nonlinear control.
\newblock {\em arXiv preprint arXiv:2006.12466}.

\bibitem[Kalman, 1963]{kalman1963mathematical}
Kalman, R.~E. (1963).
\newblock Mathematical description of linear dynamical systems.
\newblock {\em Journal of the Society for Industrial and Applied Mathematics,
  Series A: Control}.

\bibitem[Klamka, 1963]{klamka1963controllability}
Klamka, J. (1963).
\newblock Controllability of linear dynamical systems.
\newblock {\em Contrib. Theory Differ. Equ}, pages 189--213.

\bibitem[Lancaster and Rodman, 1995]{lancaster1995algebraic}
Lancaster, P. and Rodman, L. (1995).
\newblock {\em Algebraic riccati equations}.
\newblock Clarendon press.

\bibitem[Lin, 1974]{lin1974structural}
Lin, C.-T. (1974).
\newblock Structural controllability.
\newblock {\em IEEE Transactions on Automatic Control}.

\bibitem[Mania et~al., 2020]{mania2020active}
Mania, H., Jordan, M.~I., and Recht, B. (2020).
\newblock Active learning for nonlinear system identification with guarantees.
\newblock {\em arXiv preprint arXiv:2006.10277}.

\bibitem[Mania et~al., 2019]{mania2019certainty}
Mania, H., Tu, S., and Recht, B. (2019).
\newblock Certainty equivalence is efficient for linear quadratic control.
\newblock In {\em Proceedings of the 33rd International Conference on Neural
  Information Processing Systems}.

\bibitem[Perdomo et~al., 2021]{perdomo2021towards}
Perdomo, J.~C., Simchowitz, M., Agarwal, A., and Bartlett, P. (2021).
\newblock Towards a dimension-free understanding of adaptive linear control.
\newblock {\em arXiv preprint arXiv:2103.10620}.

\bibitem[Sarkar and Rakhlin, 2019]{sarkar2019near}
Sarkar, T. and Rakhlin, A. (2019).
\newblock Near optimal finite time identification of arbitrary linear dynamical
  systems.
\newblock In {\em International Conference on Machine Learning}. PMLR.

\bibitem[Simchowitz and Foster, 2020]{simchowitz2020naive}
Simchowitz, M. and Foster, D. (2020).
\newblock Naive exploration is optimal for online lqr.
\newblock In {\em International Conference on Machine Learning}. PMLR.

\bibitem[Smith, 1992]{smith1992some}
Smith, R.~L. (1992).
\newblock Some interlacing properties of the schur complement of a hermitian
  matrix.
\newblock {\em Linear algebra and its applications}.

\bibitem[Sontag, 2013]{sontag2013mathematical}
Sontag, E.~D. (2013).
\newblock {\em Mathematical control theory: deterministic finite dimensional
  systems}.
\newblock Springer Science \& Business Media.

\bibitem[Stanfel et~al., 2020]{stanfel2020distributed}
Stanfel, P., Johnson, K., Bay, C.~J., and King, J. (2020).
\newblock A distributed reinforcement learning yaw control approach for wind
  farm energy capture maximization.
\newblock In {\em 2020 American Control Conference (ACC)}. IEEE.

\bibitem[Sun et~al., 2020]{sun2020finite}
Sun, Y., Oymak, S., and Fazel, M. (2020).
\newblock Finite sample system identification: improved rates and the role of
  regularization.

\bibitem[Tropp, 2012]{tropp2012user}
Tropp, J.~A. (2012).
\newblock User-friendly tail bounds for sums of random matrices.
\newblock {\em Foundations of computational mathematics}.

\bibitem[Tsiamis and Pappas, 2021]{tsiamis2021linear}
Tsiamis, A. and Pappas, G.~J. (2021).
\newblock Linear systems can be hard to learn.
\newblock {\em arXiv preprint arXiv:2104.01120}.

\bibitem[Vershynin, 2010]{vershynin2010introduction}
Vershynin, R. (2010).
\newblock Introduction to the non-asymptotic analysis of random matrices.
\newblock {\em arXiv preprint arXiv:1011.3027}.

\bibitem[Wainwright, 2019]{wainwright2019high}
Wainwright, M.~J. (2019).
\newblock {\em High-dimensional statistics: A non-asymptotic viewpoint}.
\newblock Cambridge University Press.

\bibitem[Wang and Yang, 2020]{wang2020episodic}
Wang, T. and Yang, L.~F. (2020).
\newblock Episodic linear quadratic regulators with low-rank transitions.
\newblock {\em arXiv preprint arXiv:2011.01568}.

\bibitem[Zhou et~al., 1996]{zhou1996robust}
Zhou, K., Doyle, J.~C., Glover, K., et~al. (1996).
\newblock {\em Robust and optimal control}, volume~40.

\end{thebibliography}
 \appendix
\newpage

\section{Summary of Sample Complexity Results}\label{app: summary}

\begin{figure}[t]
\begin{center}
    \begin{minipage}{0.4\textwidth}
\centering
\begin{algorithm}[H]
\caption{
 Oracle via second-moments}
\label{alg: dynamics for diagonal covariance}
\begin{algorithmic}[1]
\State {\bf Require:} $N>0$, $\sigma_0>0$ 
\State Sample $\Dcal = \cbr{(x_{0,n},x_{1,n})}_{n=1}^{N}$
\State Set $\hat{A} = \frac{1}{N\sigma_0^2} \sum_{n=1}^N x_{1,n}x_{0,n}^{\top} $
\State {\bf Output: $\hat{A}$} 
\vspace{0.28cm}
\end{algorithmic}
\end{algorithm}
\end{minipage}
\hspace{0.2cm}
\begin{minipage}{0.52\textwidth}
\centering
\begin{algorithm}[H]
\caption{Oracle via Semiparametric LS}
\label{alg: dynamics for PD covariance}
\begin{algorithmic}[1]
\State {\bf Require:} $N>0$ 
\State Sample $\Dcal =\cbr{(x_{0,n},x_{1,n})}_{n=1}^{2N}$ 
\ForAll{$i,j \in [d]$}
    \State Estimate $\widehat{A}(i,j)$ via semiparametric LS 
    \State for indices $(i,j)$, \pref{alg: semi parametric regression}
\EndFor
\State {\bf Output:}  $\widehat{A}$
\end{algorithmic}
\end{algorithm}
\end{minipage}
\end{center}
\end{figure}

In \pref{sec: learning sparse lqr}, we study the performance of \pref{alg: control pc lqr samples}, which assumes an oracle access to an $(\sqrt{\epsilon/2s(s+d_u))},\delta)$ element-wise estimate of the dynamics $(A,B)$. Given such access, \pref{thm: PC-LQ main result} establishes a near-optimal performance guarantee of \pref{alg: control pc lqr samples}.

Then, in \pref{sec: estimation oracle sample complexity} we study the sample complexity of the assumed $(\epsilon,\delta)$ element-wise estimate for two settings: when $x_0$ has a diagonal covariance (\pref{sec: diagonal matrix assumption}) and when $x_0$ has a PD covariance (\pref{sec: PD covariance assumption}). 

By combining these together, we get, as a corollary, the sample complexity of the two algorithms considered in this work. That is, when \pref{alg: control pc lqr samples} is instantiated with \pref{alg: dynamics for diagonal covariance} or with \pref{alg: dynamics for PD covariance}. We now formally give these corollaries, which \pref{tab:tabbounds} summarizes, for completeness.

\begin{corollary}[Learning PC-LQR with second-moment Estimate]
Let the assumptions of \pref{prop: elementwise convergence of second-moment based estimation} and \pref{thm: PC-LQ main result} hold. Then, given 
\begin{align*}
    N = O\rbr{\frac{\log\rbr{\frac{d}{\delta}}\rbr{s^2+d_u s}\sigma^2_{\mathrm{eff}}\vee 1}{\epsilon}}
\end{align*}
samples, the optimal policy $\bar{K}$ of the returned model of \pref{alg: control pc lqr samples} is at most $\epsilon$ suboptimal, where
\begin{align*}
    J_\star(\bar{K}) \leq J_\star + O\rbr{\norm{P_{\star,1:2}}^8 \epsilon}.
\end{align*}
\end{corollary}
\begin{proof}
By \pref{prop: elementwise convergence of second-moment based estimation}, given such amount of samples, the second-moment estimate is an $\rbr{O(\sqrt{\epsilon/s(s+d_u)}),\delta}$ element-wise estimate of the dynamics matrix. Applying \pref{thm: PC-LQ main result} implies the result.
\end{proof}

\begin{corollary}[Learning PC-LQR with semiparametric Least Square Estimate]
Let the assumptions of \pref{prop: semiparameteric least model sample complexity} and \pref{thm: PC-LQ main result} hold. Then, given 
\begin{align*}
    N =  O\rbr{\frac{\sigma^2(s^2+sd_u)\log\rbr{\frac{d}{\delta}}}{\epsilon\lambda_{\min}(\Sigma)} + \frac{d\rbr{A_{\max}^2 \lambda_{\max}(\Sigma)+\sigma^2}\sqrt{s^2+sd_u}\log\rbr{\frac{d}{\delta}}}{ \sqrt{\epsilon \lambda_{\min}(\Sigma)}}}
\end{align*}
samples, the optimal policy $\bar{K}$ of the returned model of \pref{alg: control pc lqr samples} is at most $\epsilon$ suboptimal, where
\begin{align*}
    J_\star(\bar{K}) \leq J_\star + O\rbr{\norm{P_{\star,1:2}}^8 \epsilon}.
\end{align*}
\end{corollary}
\begin{proof}
By \pref{prop: semiparameteric least model sample complexity}, given such amount of samples, the semiparametric LS estimate is an $\rbr{O(\sqrt{\epsilon/s(s+d_u)}),\delta}$ element-wise estimate of the dynamics matrix. Applying \pref{thm: PC-LQ main result} implies the result.
\end{proof}

\section{Comment on the Structure of Covariance Matrix}\label{app: structure of covariance matrix}
In~\pref{sec: PD covariance assumption} we devised an entrywise estimator given a general covariance matrix. We further elaborate why this is needed for general PC-LQR problems. We consider two cases, (1) that the sampling policy does not depend on the state variables of the third block, and (2) when they may depend on the state variables of the third block.

\paragraph{Case 1: sampling policy does not depend on the state variables of the third block.} In this case, assuming the noise is Gaussian with a diagonal covariance matrix, the covariance matrix has the following block structure
\begin{align*}
    \Sigma = 
    \sbr{
    \begin{matrix}
      \Sigma_{11} & \Sigma_{12} &0\\
      \Sigma_{12}^{\top}& \Sigma_{22} &\Sigma_{23}\\
      0 & \Sigma_{23}^{\top} & \Sigma_{33} 
    \end{matrix}
    }.
\end{align*}
That is, there is a coupling between the state-variables on the second and third block.

\paragraph{Case 2. sampling policy depends on the state variables of the third block.} In this case, the covariance matrix may take an arbitrary shape. That is, if the sampling policy is a function of the state variables of the third block, the covariance matrix might have non-zero off-diagonal in the PC-LQR model. Indeed, in lack of prior information on the identity of the non-controllable and non-relevant state variables, the sampling policy may depend on these state variables.

\newpage
\section{Counterexample with a General Uncontrollable System} \label{app: counterexample for general uncontrollable system}

Consider an LQR model $L_\rho = (A_{\rho}, B, Q)$ where $|\rho|<1$
\begin{align}
    A_{\rho}
    =
    \sbr{
    \begin{matrix}
     1 & 1  \\
     0 & \rho\\
    \end{matrix}
    }, \quad 
    B
    =  
    \sbr{
    \begin{matrix}
     1 \\
     0 
    \end{matrix}
    }, 
    Q
    =
    \sbr{
    \begin{matrix}
     1 & 0 \\
     0 & 0\\
    \end{matrix}
    }. \label{eq: stable system 2d with uncontrollable coordinates}
\end{align}
See that by \pref{thm: invariance of optimal policy} the optimal policy of this LQR and the LQR with a modified cost $Q = I_d$ is invariant. See that only the first coordinate of this system is controllable. For simplicity of analysis, we consider $L_\rho = (A_{\rho}, B, Q)$.

Let $P_{\rho,\star}$ be the solution of the Riccati equation. Then, the optimal policy is then given by
\begin{align*}
    K_{x,\star} = (R+ B^{\top} P_{\rho,\star} B)^{-1}(B^{\top} P_{\rho,\star} A).
\end{align*}
In \pref{app: closed form solution riccati} we solve the Riccati equation, in closed form, and show that
\begin{align*}
    &P_{\rho,1} = \frac{1 + \sqrt{5}}{2}, \textit{ and } P_{\rho,12} =  \frac{P_1}{P_1^2 -\rho}.
\end{align*}
This implies that the optimal policy takes the following form,
\begin{align*}
     K_{\rho,\star} = \frac{1}{1+ P_1}\sbr{\begin{matrix}
       P_1 & P_1 + \rho P_{\rho,12}
     \end{matrix}} = \frac{P_1}{1+ P_1}\sbr{\begin{matrix}
       1 & \frac{P_1^2}{P_1^2-\rho }
     \end{matrix}}.
\end{align*}
Observe that since $P_1>1$ and $|\rho|<1$ this object is well defined. 

The above implies that for $\rho_1\neq \rho_2$ it holds that
\begin{align}
     K_{\rho_1,\star} \neq  K_{\rho_2,\star}. \label{eq: two d system non controllable information is relevant}
\end{align}
Hence, the optimal policy is a function of $\rho$.

\paragraph{Extending the Construction to Arbitrary Dimension}
To extend the argument to arbitrary dimension  consider the $d$ dimensional deterministic LQR problem $L_\rho = (A_{\rho},B,Q)$ 
\begin{align*}
    A_{\rho}
    =
    \sbr{
    \begin{matrix}
     1 & 1  & 1       &   \cdots     &  1   \\
     0 &\rho(1)  & 0         &  \cdots  &    0\\
     0  &   0    & \rho(2)  &    0     &     0  \\
      &     &  \vdots &       &      \vdots    \\
     0 &  0    &\cdots        & 0  &        \rho(d-1)
    \end{matrix}
    }, \quad 
    B
    =  
    \sbr{
    \begin{matrix}
     1 \\
     0 \\
     \vdots \\
     0 \\
     0
    \end{matrix}
    }, Q
    =
    \sbr{
    \begin{matrix}
     1       & 0       & \cdots  \\
     0      & 0         &  \cdots\\
     \vdots  &        &  0  
    \end{matrix}
    }, 
\end{align*}
where $|\rho| < 1$. As before, the optimal policy of $L_\rho$ and the LQR system $L_\rho = (A_{\rho},B,I_d)$ is invariant by \pref{thm: invariance of optimal policy}: only the first coordinate of this system is controllable. For simplicity we analyze $L_\rho$.

Observe that if a state variable is initialized as $x_0(i)=0$ for any $i\in \cbr{2,..,d}$ then it remains zero, no matter which action is applied, since these coordinates are uncontrollable. Furthermore, since $K_{\rho,\star}\in \mathbb{R}^{d_u\times d}$ induces an optimal policy \emph{for any} state variable, it induces an optimal policy for any such initial state. 

Observe that if we initialize the state variable as $x_0(i) = \ind\cbr{i= i_0}$ for some $i_0\in \cbr{2,..,d}$ the system is effectively equivalent to the $2$-dimensional system of \pref{app: closed form solution riccati}. For this two dimensional system, we show the optimal controller is a function of $\rho$, see~\eqref{eq: two d system non controllable information is relevant}. This establishes the fact that for any two different vectors $\rho_1\neq \rho_2$ the optimal policy of $L_{\rho_1} = (A_{\rho_1},B,I_d)$ and $L_{\rho_2} = (A_{\rho_2},B,I_d)$ is different.

\subsection{Solving the Riccati Equation}\label{app: closed form solution riccati}
The Riccati equation, for the above systems, has the following form.
\begin{align*}
    &P_{\rho,\star}= A_{\rho}^{\top} P_{\rho,\star} A_{\rho} + Q - (B^TP_{\rho,\star}A_{\rho})^{\top}(R+B^TP B)^{-1}B^TP_{\rho,\star}A_{\rho}\\
    &=
    \sbr{\begin{matrix}
      P_{\rho,1} &P_{\rho,1} + \rho P_{\rho,12}\\
      P_{\rho,1} + \rho P_{\rho,12} &P_{\rho,1} + 2\rho P_{\rho,12} +\rho^2 P_{\rho,2}
    \end{matrix}} 
    + I_2 - \frac{1}{1+P_{\rho,1}}
    \sbr{\begin{matrix}
     P_{\rho,1}^2 &P_{\rho,1}(P_{\rho,1} + \rho P_{\rho,12})\\
      P_{\rho,1}(P_{\rho,1} + \rho P_{\rho,12})  & (P_{\rho,1} + \rho P_{\rho,12})^2
    \end{matrix}}.
\end{align*}
\paragraph{Solving for $P_{\rho,1}$. } We solve the Riccati equation for its $(1,1)$ entry. For this entry, we get
\begin{align*}
     P_{\rho,1}^2 -  P_{\rho,1}-1=0.
\end{align*}
Solving for $P_{\rho,1}$ we get two solutions, independently of the value of $x$.
\begin{align*}
    P_{\rho,1} \equiv P_1= \frac{1 \pm \sqrt{5}}{2}.
\end{align*}
Eventually, we will show that only a single solution is valid among the two.
\paragraph{Solving for $P_{\rho,12}$. } We solve the Riccati equation for its $(1,2)$ entry (or, equivalently $(2,1)$). For this entry, we get
\begin{align}
    P_{\rho,12} = \frac{P_1}{1+P_1 -\rho} = \frac{P_1}{P_1^2 -\rho}. \label{eq: result on P12}
\end{align}

\paragraph{Solving for $P_{\rho,2}$.} Finally, we solve the Riccati equation for its $(2,2)$ entry. For this entry, we get

\begin{align}
    P_{\rho,2} = \frac{\rbr{P_1^5 -P_1\rho^2 - P_1^4}/(1-\rho^2)}{ (P_1^2 -\rho)^2} \label{eq: relation 1 uncontrollability is necessary}
\end{align}


\paragraph{Picking a solution.} Observe that the eigenvalues of a $2\times 2$ matrix are
\begin{align}
    \lambda_{\pm} = \frac{\tr(A) \pm \sqrt{\tr(A)^2 - 4 \det(A)}}{2}. \label{eq: eigenvalues 2X2}
\end{align}
We now show that $P_1= \frac{1 + \sqrt{5}}{2}$ is a PSD solution whereas $P_1 = \frac{1 - \sqrt{5}}{2}$ induces a non-PSD $P_{\rho,\star}$.
\paragraph{ $P_1= \frac{1 + \sqrt{5}}{2}$ is a PSD solution.} We check that $\det(P_{\rho,\star})\geq 0$ and $\tr(P_{\rho,\star})\geq 0$. This implies that $P_1= \frac{1 + \sqrt{5}}{2}$ is a PSD solution by~\eqref{eq: eigenvalues 2X2}. We show that  $\det(P_{\rho,\star})\geq 0$. Since $P_{\rho,\star}$ is symmetric, this condition is equivalent to
\begin{align*}
    &\rbr{P_1^6 -P_1^2\rho^2 - P_1^5}/(1-\rho^2) \geq P_1^2\\
    \iff &P_1^6 -P_1^2\rho^2 - P_1^5 \geq P_1^2(1-\rho^2)\\
    \iff &P_1^4 -P_1^3\geq1,
\end{align*}
which holds since $P_1^4 -P_1^3\geq 2.6$. We show that $\tr(P_{\rho,\star})\geq 0$. To show that, it suffices to check that
\begin{align*}
    &P_1 + \frac{\rbr{P_1^5 -P_1\rho^2 - P_1^4}/(1-\rho^2)}{ (P_1^2 -\rho)^2}\geq 0.
\end{align*}
Since $P_1^5 -P_1\rho^2 - P_1^4\geq 0$ for $\rho\in (-1,1)$ and $P_1=(1+\sqrt{5})/2$ we get that $\tr(P_{\rho,\star}) \geq 0$. Hence, $P_1= \frac{1 + \sqrt{5} }{2}$ induces a PSD solution.

\paragraph{ $P_1= \frac{1 - \sqrt{5}}{2}$ is not a PSD solution.} We show that for this solution, either $\det(P_{\rho,\star})<0$ or $\tr{P_{\rho,\star}}<0$. This implies, by~\eqref{eq: eigenvalues 2X2} that the matrix has a negative eigenvalues and thus it is not a PSD matrix. This contradicts the fact $P_{\rho,\star}$ is PSD. 

By the above calculation, and since $P_{\rho,\star}$ is symmetric, it holds that
\begin{align*}
    &\det(P_{\rho,\star}) = P_1 P_{\rho,2} - P_{\rho,12}^2.
\end{align*}
To show that $\det(P_{\rho,\star})<0$ it suffices to show 
\begin{align*}
    &\rbr{P_1^6 -P_1^2\rho^2 - P_1^5}/(1-\rho^2) < P_1^2\\
    \iff &P_1^6 -P_1^2\rho^2 - P_1^5 <P_1^2(1-\rho^2)\\
    \iff &P_1^4 -P_1^3<1,
\end{align*}
which always holds since $P_1^4 -P_1^3< 0.4.$ Thus, $\det(P_{\rho,\star})<0$ for $P_1= \frac{1 + \sqrt{5}}{2}$ which implies this solution should be eliminated.

\newpage
\section{Invariance of Optimal Policy of a PC-LQ}\label{app: invariance of optimal policy of a PC-LQR}

\propositionInvarianceOfOptimalPolicyWithDifferentDynamics*
\begin{proof}
{\bf First statement.}  First, we show that for any fixed and stable policy $K$, the difference in values between $L_1$ and $L_2$ does not depend on the policy $K$ when the cost is transformed $I_d \rightarrow I_{1+}$. Fix $K$ which stabilizes $A_1$ and $x\in \mathbb{R}^d$. We calculate the difference $J_{L_2,K}(x) - J_{L_1,K}(x)$ and show it does not depend on $K$. It holds that
\begin{align*}
    &J_{L_2,K}(x) - J_{L_1,K}(x) = \EE\sbr{\sum_{t\geq 1} c_2(x_t) - c_1(x_t)|x_1=x;K}  = \EE\sbr{\sum_{t\geq 1} \sum_{i\in \Ical}\norm{x_{t}(i)}_2^2 \mid x_1=x;K},
\end{align*}
where $\Ical$ is the set of coordinates for which the diagonal of $I_{1+}$ is zero. That is, 
\begin{align*}
    \Ical = \cbr{i \in [d]: I_{1+}(i,i) = 0}.
\end{align*}
Observe that for any coordinate $i\in \Ical$ the state variable $x_{t}(i)$ is in either the second or third blocks of~\eqref{eq: main body central lqr model }, the coordinates that corresponds to uncontrollable state variables. Thus, $x_t(i)$ for any $i\in \Ical$ is not affected by the policy $K$ (see~\pref{lem: marginilizing over endo state and policy dependence}). This implies that for any $x\in \mathbb{R}^d$,
\begin{align*}
    J_{L_2,K}(x) - J_{L_1,K}(x) = \EE\sbr{\sum_{t\geq 1} \sum_{i\in \Ical}\norm{x_{t}(i)}_2^2 \mid x_1=x;K} = C,
\end{align*}
i.e., the difference $J_{L_2,K}(x) - J_{L_1,K}(x)$ is constant. This implies that for any $x$,
\begin{align*}
    \arg\min_{K} J_{L_1,K}(x)= \arg\min_{K} J_{L_2,K}(x).
\end{align*}
Hence, the policy $u(x) = K^*(L_1)x$ which is optimal for $L_1$ is also optimal for $L_2$.

{\bf Second statement.} Combining the first statement together with \pref{lem: invariance under model transformation} we prove the claim. That is, consider two alternative PC-LQR problems, $\tilde{L}_1 = (A,B,I_1), \tilde{L}_2=(\bar{A},B,I_1)$ where $I_1$ is diagonal such that $$I_1(i,i) = \ind\cbr{\textit{i belogns to the first block}}.$$ 
By \pref{lem: invariance under model transformation} it holds that
\begin{align*}
    &K^*(L_1) = K^*(\tilde{L}_1), \textit{and}, K^*(L_2) = K^*(\tilde{L}_2).
\end{align*}
Then, by the first statement it holds that 
\begin{align*}
    K^*(\tilde{L}_1) = K^*(\tilde{L}_2).
\end{align*}
Combining the two relations concludes the proof.
\end{proof}

\begin{lemma}[Invariance of Optimal Policy Under Model Transformation]\label{lem: invariance under model transformation}
Consider the following LQR problems, $L_1=(A,B,I_1),L_2 = (\bar{A},B, I_1)$ where the dynamics are given by
\begin{align*}
    A =
    \sbr{
    \begin{matrix}
     A_1 & A_{12} & 0\\
     0 & A_2 & 0 \\
     0 & A_{32} & A_3
    \end{matrix}},\
    \bar{A} =
    \sbr{
    \begin{matrix}
     A_1 & A_{12} & 0\\
     0 & A_2 & 0 \\
     0 & \bar{A}_{32} & \bar{A}_3
    \end{matrix}},\ B = 
    \sbr{\begin{matrix}
     B_1\\
     0\\
     0
    \end{matrix}
    },
\end{align*}
and,
\begin{align*}
    I_1 = \sbr{
    \begin{matrix}
     I & 0 & 0\\
     0 & 0 & 0 \\
     0 & 0 & 0
    \end{matrix}}.
\end{align*}
Then, the optimal policy of the two models is similar, i.e., $K^*(L_1) = K^*(L_2).$
\end{lemma}
To prove this result, we consider the run of the policy iteration algorithm on both $L_1$ and $L_2$ for a specific initialization. We show, that there is a conserved structure on both $L_1$ and $L_2$ by which we conclude that $P^\star(L_1)=P^\star(L_2)$ by applying \pref{thm: convergence of policy iteration for LQR}. The formal proof is given as follows.
\begin{proof}
{\bf Step 1.} Verifying conditions of \pref{thm: convergence of policy iteration for LQR}. First,  see that $\tilde{P}=0$ satisfies the linear inequality in the requirement of \pref{thm: convergence of policy iteration for LQR}. We now show that exists a stable policy for both $L_1,L_2$. \\
Let $K_0\in \mathbb{R}^{d_u\times s_c}$ be a stable policy for $(A_1,B_1)$. Indeed, since $(A_1,B_1)$ are controllable, such policy exists. We first claim that $A+BK_0^{E}$ where $K_0^{E}  = \sbr{\begin{matrix} K_0 & 0 & 0 \end{matrix}}\in \mathbb{R}^{d_u\times d_x}$ is a stable policy 
To prove this claim, observe that due to the block structure of $A+BK_0^E$ it holds that
\begin{align*}
    \det(A+BK_0^{E}-\lambda I)=  \det(A_{1}+B_1K_0-\lambda I_1)\det(A_{2}-\lambda I_2)\det(A_{3}-\lambda I_3),
\end{align*}
thus, $\lambda$ is an eigenvalue of $A+BK_0^{E}$ if and only if it is an eigenvalue of either $A_{1}+B_1K_0,A_{2}$ or $A_{3}$. Since all of these systems are stable, i.e., every eigenvalue is smaller than one, then $A+BK_0^{E}$ is also stable. Furthermore, since $\bar{A}_3$ is also assumed to be stable, then, by similar reasoning, $\bar{A}+BK_0^{E}$ is stable.

{\bf Step 2.} Applying policy iteration on both $L_1$ and $L_2$ with the initialized $K_0^{(E)}$. We now apply the policy iteration algorithm on both $L_1$ and $L_2$, where we initialize both from $K_0^{(E)}$. Let $K^{(1)}_i, K^{(2)}_i$ be the policies obtained at the $i^{th}$ iteration when running policy iteration on $L_1$ and $L_2$, respectively.\\
The following claim is established via induction: for any iteration $i\geq 0$, it holds that $$K^{(1)}_i =K^{(2)}_i = \sbr{K_{i,1},\ K_{i,2},\ 0},$$ i.e., the policy does not depend on the third block. Due to the convergence of policy iteration to the optimal policy, this result will conclude the proof.\\
\emph{Base case.} Holds due to the initialization $K^{(1)}_0 =K^{(2)}_0 = K_0^E$.\\
\emph{Inductive step.} Assume the claim holds until the $(i-1)^{th}$ iteration. We prove it holds for the $i^{th}$ iteration. Since the policies at the $(i-1)^{th}$ iteration are equal and does not depend on the third block by the induction hypothesis (the third block is zero) it holds that
\begin{align*}
    P^{(1)}_{i-1} =P^{(2)}_{i-1} = P_{i-1}
    =  
    \sbr{
    \begin{matrix}
     P_1 & P_{12} & 0\\
     P_{12} & P_2 & 0\\
     0 & 0 & 0
    \end{matrix}
    },
\end{align*}
by \pref{lem: form of exo independant policy and P} for some $P_{1},P_{12},P_2$. The policy at the $i^{th}$ iteration is given by
\begin{align*}
    &K_{i}^{(1)} = (R + B^{\top} P_{i-1}B)B^TP_{i-1}A\\
    &K_{i}^{(2)} = (R + B^{\top} P_{i-1}B)B^TP_{i-1}\bar{A}
\end{align*}
By a direct calculation due to the form of $P_{i-1}$, it can be observed that $K_{i}^{(1)} = K_{i}^{(2)} = \sbr{\begin{matrix} K_{i,1} & K_{i,2} & 0\end{matrix}}$ for some $K_{i,1},K_{i,2}$, that is, $K_{i}^{(1)}$ and $K_{i}^{(2)}$ are equal and both do not depend on the third block. Hence, the induction step is proven, and the lemma follows.
\end{proof}

\begin{lemma}\label{lem: marginilizing over endo state and policy dependence}
Let $x_{t}(i)$ be a state vector where $i$ belongs either to the second or third blocks of a PC-LQR. That is, state vector of the uncontrollable coordinates. Then, for any policy $K$ it holds that
\begin{align*}
    \EE\sbr{\sum_{t\geq 1} x_{t}(i)^2 \mid x_1=x;K} = C,
\end{align*}
that is, it does not depend on the policy $K$.
\end{lemma}
\begin{proof}
First, observe that any power $n$ of a block matrix is given by
\begin{align}
   A^{n-1} = 
    \sbr{
    \begin{matrix}
     A_1 & X_2\\
     0 & X_3
    \end{matrix}
    }^n = \sbr{
    \begin{matrix}
     A_1^n & Poly(X_2,X_3, A_1)\\
     0 & X_3^n
    \end{matrix}
    },\label{eq: consequences of block structure}
\end{align}
where $Poly(X_2,X_3, A_1)$ is some polynomial of the matrices $X_2,X_3,A_1$. See that the full state vector of any fixed policy $K$ is give by
\begin{align*}
    x_t = (A - BK)^t x_0 + \sum_{\tau=1}^t (A - BK)^{t-\tau}\xi_\tau,
\end{align*}
where $\xi_\tau$ is an independent i.i.d. and zero mean random vector.
Let $e_i$ be a one-hot vector with $e_i(i)=1$ and zero elsewhere. Due to~\eqref{eq: consequences of block structure}, and since the first block is the only controllable block, we get that
\begin{align}
    e_i^{\top}(A - BK)^n  =  e_i^{\top}\sbr{
    \begin{matrix}
     (A_1-B_1K_1)^n & Poly(X_2,X_3, A_1,B_1K_{12})\\
     0 & X_3^n
    \end{matrix}
    } = e_i^{\top}\sbr{
    \begin{matrix}
     0 & 0\\
     0 & X_3^n
    \end{matrix}
    }, \label{eq: full dynamics evolution}
\end{align}
since $$BK = \sbr{
    \begin{matrix}
     B_1 K_1 & B_1 K_{12}\\
     0 &0
    \end{matrix}
    }, $$
and since $e_i(j)=0$ for all coordinates $j$ of the first block. Combining the above we get that
\begin{align*}
    &x_t(i) = e_i^{\top}(A - BK)^tx_0 + \sum_{\tau=1}^t e_i^{\top}(A - BK)^{t-\tau}\xi_\tau\\
    &= e_i^{\top}\sbr{
    \begin{matrix}
     0 & 0\\
     0 & X_3^t
    \end{matrix}
    }x_0 + \sum_{\tau=1}^t T\sbr{
    \begin{matrix}
     0 & 0\\
     0 & X_3^{t-\tau}
    \end{matrix}
    }\xi_{\tau} \\
    &= e_i^{\top} A_0^t x_0 +\sum_{\tau=1}^t A_0^{t-\tau} \xi_\tau,
\end{align*}
where 
$$
A_0 
= 
\sbr{
\begin{matrix}
 0 & 0\\
 0 & X_3
\end{matrix}}
$$ 
does not depend on $K$. Thus,
\begin{align*}
    \EE[x_{t}(i)^2|K] =  \EE\sbr{\rbr{e_i^{\top} A_0^t x_0 +\sum_{\tau=1}^t A_0^{t-\tau} \xi_\tau}^2|K} = C
\end{align*}
does not depend on the policy $K$, since $\xi_\tau$ is i.i.d. and has the same distribution for all $K$.
\end{proof}

\subsection{Useful Results}
\begin{theorem}[Asymptotic Convergence of Policy Iteration for LQR, e.g.,~\cite{lancaster1995algebraic}, Theorem 13.1.1.]\label{thm: convergence of policy iteration for LQR}
Assume that $(A,B)$ are stabilizable, $R$ invertible, and assume that there is an hermitian solution $\tilde{P}$ to the linear matrix inequality
\begin{align*}
    P\leq A^{\top} P A + Q - (B^TPA)^{\top}(R+B^TP B)^{-1}B^TPA,
\end{align*}
for which $R+B^{\top}\tilde{P}B>0$. Then, there exists a unique solution $P^\star$ to the Riccati equation
\begin{align}
    P= A^{\top} P A + Q - (B^TPA)^{\top}(R+B^TP B)^{-1}B^TPA,\label{eq: appendix riccati equation}
\end{align}
such that $P^\star\geq P$ for all the solutions of~\eqref{eq: appendix riccati equation}. Furthermore, the Policy Iteration procedure in which we initialize $(K_0,P_{K_0})$ with some stable policy and update
\begin{align*}
    K_{i} = (R + B^{\top} P_{i-1}B)B^TP_{i-1}A, \quad P_i = \sum_{t\geq 0}((A+BK_{i})^{\top})^t(Q + K_i^{\top} R K_i)(A+BK_{i})^t 
\end{align*}
converges to $P^\star$.
\end{theorem}

\begin{lemma}\label{lem: form of exo independant policy and P}
Lee $L=(A,B,I_1)$ where 
\begin{align*}
    A =
    \sbr{
    \begin{matrix}
     A_1 & A_{12} & 0\\
     0 & A_2 & 0 \\
     0 & A_{32} & A_3
    \end{matrix}},\
    ,\ B = 
    \sbr{\begin{matrix}
     B_1\\
     0\\
     0
    \end{matrix}
    },\ 
    I_1 =
    \sbr{
    \begin{matrix}
     I & 0 & 0\\
     0 & 0 & 0 \\
     0 & 0 & 0
    \end{matrix}}.
\end{align*}
Assume that a policy $K$ is stable and does not depend on the third block. Then,
\begin{align*}
    P_K
    =
    \sbr{
    \begin{matrix}
     P_1 & P_{12} & 0\\
     P_{12} & P_2 & 0\\
     0 & 0 & 0
    \end{matrix}}.
\end{align*}
\end{lemma}
See that if the policy $K$ does have a non-zero component in the third block $P_K$ might have non-zero components at the third row and third column.
\begin{proof}
Observe that by the model assumption
\begin{align*}
    A+BK = \sbr{
    \begin{matrix}
     A_1 + B_1K_1 & A_{12}+B_1 K_2 & 0\\
     0 & A_2 & 0 \\
     0 & A_{32} & A_3
    \end{matrix}}.
\end{align*}
Taking this matrix to some power $t>0$ we get that
\begin{align}
    (A+BK)^t = \sbr{
    \begin{matrix}
     (A_1 + B_1K_1)^t & V_{1,t}(A_{12},B_1,K_2,A_2) & 0\\
     0 & A_2^t & 0 \\
     0 & V_{2,t}(A_2,A_{32},A_3) & A_3^t
    \end{matrix}},\label{eq: P_K if K doesn't depend on exo}
\end{align}
where $V_{1,t},V_{2,t}$ is some polynomial in its arguments.\\ 
We now apply the previous calculation to prove the result. The matrix $P_K$ satisfies the Lyapunov relation
\begin{align*}
    P_K &= \sum_{t\geq 0}\rbr{(A+BK)^t}^{\top}(I_1 + K^{\top} R K)\rbr{(A+BK)^t}\\
    &=\sum_{t\geq 0} \rbr{(I_1 + K^{\top} R K)^{1/2}(A+BK)^t}^{\top}\rbr{(I_1 + K^{\top} R K)^{1/2}(A+BK)^t}.
\end{align*}
By a direct computation and by plugging the form of~\eqref{eq: P_K if K doesn't depend on exo} , we see that the matrix $(I_1 + K^{\top} R K)^{1/2}(A+BK)^t$ have zero elements at the third row and column, that is
\begin{align*}
    (I_1 + K^{\top} R K)^{1/2}(A+BK)^t
    = 
    \sbr{
    \begin{matrix}
     Y_1 & Y_{12} &0\\
     Y_{21} & Y_2 & 0 \\
     0 & 0 & 0
    \end{matrix}
    },
\end{align*}
for some $Y_1,Y_{12},Y_{21},Y_2$. This also implies that $((I_1 + K^{\top} R K)^{1/2}(A+BK)^t)^{\top}(I_1 + K^{\top} R K)^{1/2}(A+BK)^t$ have zero elements at the third row and column, and, hence,
\begin{align*}
    P_K =\sum_{t\geq 0} \rbr{(I_1 + K^{\top} R K)^{1/2}(A+BK)^t}^{\top}\rbr{(I_1 + K^{\top} R K)^{1/2}(A+BK)^t},
\end{align*}
have zero elements at the third row an column as well.
\end{proof}

\newpage

\section{Structural Properties of PC-LQ Problems}\label{app: structural properties of PC-LQ} 

The following lemma is well known, and is used to properly define the notion of controllable subspace, e.g.~\citep{klamka1963controllability,basile1992controlled,zhou1996robust,sontag2013mathematical}.  

\begin{lemma}[E.g.,~\cite{sontag2013mathematical}, Lemma 3.3.3.]\label{lem: controllability and structure of matrices}
Let $A\in \mathbb{R}^{n\times n}, B\in \mathbb{R}^{n\times m}$ and
\begin{align*}
    \Gcal
    =
    \sbr{
    \begin{matrix}
     B & AB & \cdots & A^{n-1} B
    \end{matrix}
    }.
\end{align*}
Then, if $\rnk(\Gcal) = r< n$ then there exists an invertible transformation such that the matrices $\tilde{A} = TAT^{-1},\tilde{B} = TB$ have the block structure
\begin{align}
    \tilde{A} = 
    \sbr{
    \begin{matrix}
     A_1 & A_2\\
     0 & A_3
    \end{matrix}
    },
    \quad
    \tilde{B} = 
    \sbr{
    \begin{matrix}
     B_1 \\
     0 
    \end{matrix}
    },\label{eq: low rank controllability supp}
\end{align}
where $A_1\in \mathbb{R}^{r\times r}, A_3\in \mathbb{R}^{d-r\times d-r}$. Conversely, if $A$ and $B$ are given by~\eqref{eq: low rank controllability supp} then $\rnk(\Gcal)\leq r.$
\end{lemma}



\propositionExistenceOfNonEssentialDisturbances*

\begin{proof}
By \pref{lem: controllability and structure of matrices}, it holds that the controllable subpace is of rank $\leq s_c$ if and only if there exists an invertible transformation $U_1$ such that
\begin{align*}
    T_1AT_1^{-1} = 
    \sbr{
    \begin{matrix}
     A_1 & X_2\\
     0 & X_3
    \end{matrix}
    },
    \quad
    T_1B = 
    \sbr{
    \begin{matrix}
     B_1 \\
     0 
    \end{matrix}
    }. 
\end{align*}
Apply this transformation and consider the relevant disturbances matrix~\eqref{eq: definition relevant disturbances matrix} $$\Rcal\Ucal=  \sbr{\begin{matrix} X^T_{2} & X^T_3 X_{2}^{\top} & \cdots & (X^T_3)^{d-s_c} X_{2}^{\top} \end{matrix}}.$$ By \pref{lem: controllability and structure of matrices}, while plugging $X_2^{\top}=B, X_3^{\top}=A$, it holds that   $\rnk(\Rcal\Ucal)\leq  s_{e}$ if and only if there exists an invertible transformation $T_2\in \mathbb{R}^{d-s_c\times d-s_c}$ such that
\begin{align}
    &\bar{T}_2 X_3^{\top} \bar{T}^{-1}_2 = 
    \sbr{
    \begin{matrix}
     A_2^{\top} & A^T_{32}\\
     0 & A_3^{\top}
    \end{matrix}
    },\quad
    \bar{T}_2X_2^{\top} = \sbr{\begin{matrix} A_{12}^{\top} \\
    0    \end{matrix}} \nonumber \\
    \iff &T^{-1}_2 X_3 T_2 = 
    \sbr{
    \begin{matrix}
     A_2 & 0\\
     A_{32} & A_3
    \end{matrix}
    },\quad
    X_2 T_2 = \sbr{\begin{matrix} A_{12} &   0    \end{matrix}} \label{eq: controllability rd matrix relation 1}
\end{align}
where $T_2 =\bar{T}_2^{\top}.$
Define an invertible transformation extended to $\mathbb{R}^d$,
\begin{align*}
    T_3 
    =
    \sbr{
    \begin{matrix}
     I & 0\\
     0 & T_2^{-1}
    \end{matrix}
    },\quad  T_3^{-1}
    =
    \sbr{
    \begin{matrix}
     I & 0\\
     0 & T_2
    \end{matrix}
    }.
\end{align*}
Then, the concatenation $T_3T_1$ yields the result since,
\begin{align*}
   &T_3 T_1AT_1^{-1}T_2^{-1}= T_3\sbr{
    \begin{matrix}
     A_1 & X_2\\
     0 & X_3
    \end{matrix}
    }T_3^{-1}\\
    &= \sbr{
    \begin{matrix}
     A_1 & X_2T_2\\
     0 & T_{2}^{-1}X_3T_2
    \end{matrix}
    }\\
    &=     \sbr{
    \begin{matrix}
     A_1 & A_{12} & 0\\
     0 & A_2 & 0\\
     0 & A_{32} & A_3
    \end{matrix}
    }, 
\end{align*}
where the last relation holds by~\eqref{eq: controllability rd matrix relation 1}.
\end{proof}

We now prove \pref{prop: lqr rud and minimal invariant subspaces}. This proposition gives an alternative characterization of a PC-LQR relatively to~\pref{prop: necessary and sufficient for existence of non essential disturbances}. Specifically, \pref{prop: lqr rud and minimal invariant subspaces} characterizes a PC-LQR by invariant and minimal invariant subspaces (which we review in \pref{app: invariant subspace and minimal invarinat subspace}) instead of relaying on the notion of the controllability matrix and the relevant disturbances matrix. Before supplying with the proof observe that if $V$ is an invariant subspace of $A$ with $\dim(s)$ then, $A$ can be written as
\begin{align}
    A = 
    \sbr{
    \begin{matrix}
     A_1 & A_{12}\\
     0 & A_2
    \end{matrix}
    }, \label{eq: form of invariant subspace supp}
\end{align}
in the basis were the first $s$ coordinates span $V$.

\PropFromLQRUDtoMIS*
Given the definition of minimal invariant subspace, the proof is straightforward.
\begin{proof}
{\bf $\rightarrow$.} If an LQR is equivalent to a PC-LQR then, there exists some basis such that the dynamics of $L$ is given as
\begin{align*}
    A =
    \sbr{
    \begin{matrix}
     A_1 & A_{12} & 0\\
     0 & A_2 & 0 \\
     0 & A_{32} & A_3
    \end{matrix}},\
    B = 
    \sbr{\begin{matrix}
     B_1\\
     0\\
     0
    \end{matrix}
    }.
\end{align*}
It can be observed that, alternatively, in this basis, we can write
\begin{align*}
    A = P_c A P_c + P_{r}A (P_{r}-P_{c}) + (I-P_{r})A(I-P_c),\quad  B = P_B B,
\end{align*}
where $P_B\subseteq P_c \subseteq P_r$, and $P_B$ is the projection on the coordinates on which $B$ has non-zero rows, $P_c$ is a projection on the coordinate of block $A_1$, and $P_r$ is a projection on the coordinates of the first two blocks. Then, rotating to the original basis does not change this representation.

{\bf $\leftarrow$.} First, rotate $B$ such that $P_B$ is diagonal. In these coordinates,
\begin{align*}
    B = \sbr{\begin{matrix}
     B_1\\
     0\\
     0
    \end{matrix}
    }.
\end{align*}
Since $P_B\subseteq P_c$ it can be jointly diagonalized with $P_B$. Thus, in this basis, since $P_c$ is an invariant subspace, we can write $A$ as
\begin{align*}
        A =
    \sbr{
    \begin{matrix}
     A_1 & X_2\\
     0 & X_3 
    \end{matrix}}, \textit{thus, }        
    (I-P_c)A^{\top} =
    \sbr{
    \begin{matrix}
     0 & 0\\
     X_2^{\top} & X_3^{\top} 
    \end{matrix}}.
\end{align*}
Since $P_c \subseteq P_r$ it can be jointly diagonalized with $P_c$ by a matrix
\begin{align*}
    U=\sbr{
    \begin{matrix}
     I & 0\\
     0 & \tilde{U} 
    \end{matrix}},
\end{align*}
where $\tilde{U} $ is orthogonal matrix. In this basic, by applying the transformation ${(I-P_c)A^{\top}\rightarrow U(I-P_c)A^{\top} U^{\top}}$, it holds that
\begin{align*}
    (I-P_c)A^{\top} = \sbr{
    \begin{matrix}
     0 & 0 \\
     \tilde{U}X_2^{\top} & \tilde{U}X_3^{\top}\tilde{U}^{\top} 
    \end{matrix}}.
\end{align*}
Furthermore, since $P_r$ is an invariant subspace, it must hold that
\begin{align*}
    (I-P_c)A^{\top} =
    \sbr{\begin{matrix}
     0 & 0 & 0 \\
     A_{32}^{\top} & A_2^{\top} &A_{12}^{\top}\\
     0& 0&  A_3^{\top} 
    \end{matrix}},
\end{align*}
since, otherwise, $P_r$ is not an invariant subspace. Lastly, observe that then we can write
\begin{align*}
    (I-P_c)A^{\top} = P_r(I-P_c)A^TP_r + (I-P_r)(I-P_c)A^TP_r +0 =  (P_r-P_c)A^{\top} P_r + (I-P_c) A^{\top} (I-P_r),
\end{align*}
using the fact that $P_c P_r= P_r P_c =P_c$ since $P_c \subseteq P_r$. Combining the above, we get that
\begin{align*}
    A = P_c A P_c +  P_rA (P_r-P_c) + (I-P_r) A (I-P_c)=
     \sbr{
    \begin{matrix}
     A_1 & A_{12} & 0\\
     0 & A_2 & 0 \\
     0 & A_{32} & A_3
    \end{matrix}},
\end{align*}
as we needed to show.

{\bf Minimal representation.} The last part of the proposition is a corollary of \pref{lem: equivalence of minimal invariant subspace and span of Krlyov matrix}. This lemma establishes that the minimal invariant subspace of $A$ w.r.t. $P_B$ and the span of of the Krylov matrix
\begin{align*}
    \Gcal
    =
    \sbr{
    \begin{matrix}
     B & AB & \cdots & A^{n-1} B
    \end{matrix}
    }.
\end{align*}
is equal. 
\end{proof}

\newpage

\section{Invariant Subspace and Minimal Invariant Subspace}\label{app: invariant subspace and minimal invarinat subspace}

An invariant subspace $V$ of a matrix $A\in \mathbb{R}^{n\times n}$ satisfies the following definition.
\begin{definition}[Invariant Subspace, e.g.,~\cite{basile1992controlled}, Section 3.2]\label{def: invariant subspace definition}
Let $A\in \mathbb{R}^{n\times n}$, and $V$ be a subspace of $\mathbb{R}^n$. We say that $V$ is an invariant subspace of $A$ if $AV\subseteq V$. 
\end{definition}

Instead of relaying on the common definition of invariant subspace (see \pref{def: invariant subspace definition}) we give an equivalent and algebraic characterization for this notion. This allows for our proofs to have a more algebraic nature which we found simpler in several proofs along this work.

\begin{restatable}[Equivalent Property of Invariant Subspace]{proposition}{propositionDefinitionsAreEquiv}\label{prop: definitions are equivalent}
Let $A\in \mathbb{R}^{n\times n}$, $V$ be a subspace of $\mathbb{R}^n$ and $P_V\in \RR^{n \times n}$ be the orthogonal projection onto $V$. The subspace $V$ is an invariant subspace w.r.t. $A$ if and only if $AP_V = P_VAP_V$. 
\end{restatable}
\begin{proof}
{\bf \pref{def: invariant subspace definition} $\rightarrow$  \pref{prop: definitions are equivalent}.}  Assume that $V$ satisfies \pref{def: invariant subspace definition}. We show that it also satisfies \pref{prop: definitions are equivalent}. Let $P_V$ be an orthogonal projection on the subspace $V$, which implies that $P_V= UU^{\top}$ for some $U\in \mathbb{R}^{N\times \dim(V)}$ with orthogonal columns. Furthermore, let $\cbr{u_i}_{i=1}^{\dim(V)}$ be the set of orthogonal columns. Since $u_i\in V$ it holds for each $u_i$, since \pref{def: invariant subspace definition} holds, that
\begin{align}
    &Au_i \subseteq V \nonumber \\
    \iff & A u_i = P_V A u_i\nonumber \\
    \rightarrow & A u_i u_i^{\top} = P_V A u_iu_i^{\top}. \label{eq: defn 1 implies 2 rel 1}
\end{align}
Summing on all $\dim(V)$ equations we conclude the proof of this part since,
\begin{align*}
     &A  P_V = \sum_{i=1}^{\dim(V)}A u_i u_i^{\top} \tag{$P_V= \sum_{i=1}^{\dim(V)} u_iu_i^{\top}$}\\
     &=  \sum_{i=1}^{\dim(V)}P_V A u_iu_i^{\top} \tag{Equation~\eqref{eq: defn 1 implies 2 rel 1}}\\
     &=P_V A  \sum_{i=1}^{\dim(V)}u_iu_i^{\top} =P_V AP_V. \tag{$P_V= \sum_{i=1}^{\dim(V)} u_iu_i^{\top}$}
\end{align*}

{\bf \pref{prop: definitions are equivalent} $\rightarrow$ \pref{def: invariant subspace definition}. } Assume that $V$ satisfies \pref{prop: definitions are equivalent}. We show it also satisfies \pref{def: invariant subspace definition}. Observe that by \pref{prop: definitions are equivalent} it holds that $P_VAP_V=AP_V$. Multiplying this relation by any $v\in V$ from both sides we get.
\begin{align*}
    AP_V v = P_VAP_V v.
\end{align*}
Furthermore, by \pref{lem: orthogonal projection and invariance of subspace} for any $v\in V$ it holds that $P_V v= v$. Thus,
\begin{align*}
    A v = P_V Av.
\end{align*}
This also implies that $Av\in V$, since, by \pref{lem: orthogonal projection and invariance of subspace}, any vector that satisfies $P_V u=u$ is contained within $V$, thus, $Av\in V$ since $P_V(Av)=Av$.
\end{proof}



The notion of minimal invariant subspace is given in \pref{def: minimal invariant subspace}. For such a definition to be valid, one needs to show that the minimal subspace is unique. The following result establishes this fact. That is, the minimal invariant subspace is unique, and, thus, it is a well defined notion; there are no two minimal invariant subspaces of $A$ w.r.t. a subspace~$K$.
\begin{restatable}[Minimal Invariant Subspace is Unique]{proposition}{propositionMinimalInvariantSubspaceisUnique}\label{prop: uniqueness of minimal invariant subspace}
Let $K$ be a subspace and $A\in \mathbb{R}^{n\times n}$. If $V_1$ and $V_2$ are both minimal invariant subspaces of $A$ w.r.t. $K$ then $V_1=V_2$.
\end{restatable}


\begin{proof}
Assume that $V_1\neq V_2$ and both are minimal invariant subspaces of $A$ w.r.t. $K$. We show there exists a smaller invariant subspace then both $V_1$ and $V_2$, and, thus, get a contradiction to the assumption $V_1\neq V_2.$

By the requirement (3) of \pref{def: minimal invariant subspace} $\dim(V_1)=\dim(V_2)$. Thus, the subspaces $V_1/V_2$ and $V_1/V_2$ are non-empty. Furthermore, since both are invariant subspaces it holds that
\begin{align}
    &P_{V_1}A P_{V_1} = AP_{V_1}, \textit{and},\ P_{V_2}AP_{V_2} = AP_{V_2}. \label{eq: proof uniqueness 1}
\end{align}
Let $P_{V_1} = P_{V_1\cap V_2} + P_{V_1/ V_2}$ and $P_{V_2} = P_{V_1\cap V_2} + P_{V_2/ V_1}$. First, by multiplying the first and second relations of~\eqref{eq: proof uniqueness 1} by $P_{V_1\cap V_2}$ from the right and using $P_{V_1\cap V_2}P_{V_2/ V_1} = P_{V_1\cap V_2}P_{V_1/ V_2} = 0$, since the subspaces are orthogonal, we get
\begin{align}
    &P_{V_1}A P_{V_1\cap V_2} = AP_{V_1\cap V_2} \label{eq: proof uniqueness 2}\\
    &P_{V_2}AP_{V_1\cap V_2} = AP_{V_1\cap V_2}\nonumber 
\end{align}
which implies that
\begin{align*}
    P_{V_1}A P_{V_1\cap V_2} = P_{V_2}AP_{V_1\cap V_2}.
\end{align*}
Multiplying this relation by $P_{V_1/ V_2}$ from the left and using $P_{V_1/ V_2}P_{V_1} =P_{V_1/ V_2}^2=P_{V_1/ V_2}$ and $P_{V_1/ V_2}P_{V_2}=0$ we get
\begin{align}
    P_{V_1/ V_2}A P_{V_1\cap V_2} = 0. \label{eq: proof uniqueness 3}
\end{align}
This relation, together with~\eqref{eq: proof uniqueness 2} implies the following.
\begin{align*}
    &AP_{V_1\cap V_2} = P_{V_1}A P_{V_1\cap V_2} \tag{Equation~\eqref{eq: proof uniqueness 2}}\\
    &= ( P_{V_1\cap V_2} + P_{V_1/ V_2})A P_{V_1\cap V_2} \\
    &=  P_{V_1\cap V_2}A P_{V_1\cap V_2}. \tag{Linearity and Equation~\eqref{eq: proof uniqueness 3}}
\end{align*}
These relations imply that $$AP_{V_1\cap V_2} = P_{V_1\cap V_2}A P_{V_1\cap V_2},$$ i.e., $V_1\cap V_2$ is an invariant subspace of $A$ w.r.t. to $K$. Observe that  $K \subseteq V_1\cap V_2$: both $V_1$ and $V_2$ includes the subspace $K$, by definition, and, thus, their intersection includes $K$. Hence, we found an invariant subspace of $A$, $V_1\cap V_1$, that includes $K$, and is strictly smaller than $V_1$, since $V_1/V_2$ is non empty. Since $\dim(V_1)=\dim(V_2)$ it also implies that $V_1\cap V_2$ has smaller dimension than $V_2$ as well. This implies a contradiction, since we assumed that $V_1$ and $V_2$ are minimal subspace of $A$ w.r.t.~$K$.
\end{proof}

The next result establishes a relation between the minimal invariant subspace w.r.t. an initial subspace and the span of a Krylov matrix. This allows us to draw a correspondence between the notion of minimal invariant subspace and, e.g., controllable subspace.
\begin{lemma}[Equivalent of the Span of Krlyov Matrices and Minimal Invariant Subspace]\label{lem: equivalence of minimal invariant subspace and span of Krlyov matrix}
Let $A\in \mathbb{R}^{n\times n}, B\in \mathbb{R}^{n\times m}$ and
\begin{align*}
    \Gcal
    =
    \sbr{
    \begin{matrix}
     B & AB & \cdots & A^{n-1} B
    \end{matrix}
    }.
\end{align*}
Then, the span of $\Gcal$ and the minimal invariant subspace of $A$ w.r.t. $B$ are equal.
\end{lemma}
\begin{proof}
Let $B = U\Lambda V^{\top}$ be the SVD decomposition of $B$. Then, let $P_B=UU^{\top}$ and let $V_B$ be the span of $B$.
\begin{align*}
    \Gcal_{P_B}
    =
    \sbr{
    \begin{matrix}
     P_B & A P_B & \cdots & A^{n-1} P_B
    \end{matrix}
    },
\end{align*}
and observe that the span of $\Gcal$ and $\Gcal_{P_B}$ is equal. Let $V_c$ and $V_m$ be the span of $\Gcal_{P_B}$ and the span of the minimal invariant subspace of $A$ w.r.t. $V_B$.\\
{\bf $V_c \subseteq V_m$.} Since the minimal invariant subspace is an invariant subspace w.r.t. $P_B$ (that is $P_mP_B = P_B$) it satisfies that
\begin{align*}
    &A P_B = P_m A P_B, \textit{ and } A P_m = P_m A P_m.
\end{align*}
This implies that for any $n$ $$A^{n-1} P_B = A^{n-1} P_m A P_m P_B = A^{n-2} P_mAP_m AP_B =\cdots (P_mAP_m)^n P_B^n.$$
Hence,
\begin{align*}
    \Gcal_{P_B}
    =
    \sbr{
    \begin{matrix}
     P_B & A P_B & \cdots & A^{n-1} P_B
    \end{matrix}
    } = P_m \Gcal_{P_B},
\end{align*}
which implies that $V_c \subseteq V_m.$

{\bf $V_m\subseteq V_c$.} Since $V_c$ is the span of $\Gcal_{P_B}$ it holds that
\begin{align}
    \Gcal_{P_B} = P_c \Gcal_{P_B} \label{eq: span of controllability matrix}
\end{align}
This relation implies that for all $i\in [n-1] \cup \cbr{0}$
\begin{align}
    A^{i}P_B = P_c  A^i P_B.\label{eq: span of controllability matrix 1}
\end{align}
Observe that by the Cayley Hamilton theorem, the $n^{th}$ power can be written as the following sum, for some set of coefficients
\begin{align}
    A^{n}P_B  &= \sum_{i=0}^{n-1} \alpha_i A^i P_B \nonumber \\
              &=\sum_{i=0}^{n-1} \alpha_i P_c A^i P_B \tag{By~\eqref{eq: span of controllability matrix}}  \nonumber\\
              &= P_c \sum_{i=0}^{n-1} \alpha_i A^i P_B  = P_c A^n P_B \label{eq: nth power of An}
\end{align}
Thus, together with~\eqref{eq: span of controllability matrix 1}, we get that for all $i\in [n] \cup \cbr{0}$
\begin{align}
    A^{i}P_B = P_c  A^i P_B.\label{eq: span of controllability matrix 2}
\end{align}
Observe that
\begin{align*}
    &A \Gcal_{P_B}=A P_c\Gcal_{P_B} \tag{By~\eqref{eq: span of controllability matrix}}\\
    &= \sbr{
    \begin{matrix}
     A P_B & A^2 P_B & \cdots &A^{n} P_B
    \end{matrix}
    } \\
    & = P_c\sbr{
    \begin{matrix}
     A P_B & A^2 P_B & \cdots &A^{n} P_B
    \end{matrix} 
    } \tag{By~\eqref{eq: span of controllability matrix 2}}\\
    &=P_cA\sbr{
    \begin{matrix}
      P_B & A P_B & \cdots &A^{n-1} P_B
    \end{matrix}
    }\\
    &=P_cA\Gcal_{P_B}.
\end{align*}
Since $\Gcal_{P_B}\Gcal_{P_B}^{\dagger} = U^{\top}\Sigma^{1/2}V^TV \Sigma^{-1/2} U^{\top} = UU^{\top}=P_c$ the above relation implies that
\begin{align}
    A  P_c = P_cA P_c,\label{eq: span of controllability matrix 3}
\end{align}
by multiplying by $\Gcal_{P_B}^{\dagger}$ from the RHS. By \pref{prop: definitions are equivalent} this suggests that $P_c$ is an invariant subspace.

From the above, we get that $V_c$ is an invariant subspace. Furthermore, due to the form of $\Gcal_{P_B}$, it must contain the span of $B$, $V_B$. Since the minimal invariant subspace $V_m$ is the smallest subspace that contains $V_B$ and is an invariant subspace w.r.t. $A$, we get that $V_m \subseteq V_c$. This conclude the proof since it holds that $V_m \subseteq V_c$ and $V_c \subseteq V_m$, which implies that $V_m = V_c$.
\end{proof}

\subsection{Linear Algebra Facts}

\begin{lemma}\label{lem: orthogonal projection and invariance of subspace}
Let $P_V$ be an orthogonal projection onto $V$. Then, $v\in V$ if and only if
\begin{align*}
    P_V v= v.
\end{align*}
\end{lemma}
\begin{proof}
{\bf $\rightarrow$.} We prove that $P_V v=v$ implies that $v\in V$. Write $P_V= UU^{\top}$ where $U$ is a matrix with orthonormal columns $\cbr{u_i}_{i=1}^{\dim(V)}$ and $u_i$ span $V$. With this notation, $P_V v=v$ implies that
\begin{align*}
    v = \sum_{i=1}^{\dim(V)} \inner{u_i}{v}u_i,
\end{align*}
hence, $v$ is in the span of $V$ since we can write it as $v = \sum_{i=1}^{\dim(V)} \alpha_i u_i$ and $\cbr{u_i}_{i=1}^{\dim(V)}$ span $V$.

{\bf $\leftarrow$.} We prove that if $v\in V$ then $P_V v=v$. Since $v\in V$ then it can be written as a linear combination of $\cbr{u_i}_{i=1}^{\dim(V)}$,
\begin{align*}
    v = \sum_{i=1}^{\dim(V)} \alpha_i u_i.
\end{align*}
Since $P_V u_i = u_i$ and by the linearity of orthogonal projection we conclude the proof since
\begin{align*}
    P_V v =  \sum_{i=1}^{\dim(V)} \alpha_i P_V u_i = \sum_{i=1}^{\dim(V)} \alpha_i  u_i =v. \tag*\qedhere 
\end{align*}
\end{proof}

\newpage

\section{Learning Sparse LQRs in Partially Controllable Systems}\label{app: from PC-LQ to MIS}

We now establish the correctness of \pref{alg: control pc lqr samples} given an $(\epsilon,\delta)$ element-wise oracle (see \pref{def: dynamical system estimation oracle}). 
\LQRUDviaMIS*
\begin{proof}
{\bf Consequence of thresholded estimation.} Assume that $\widehat{A}$ and $\widehat{B}$ is an  $(\epsilon,\delta)$ entrywise estimator~\pref{def: dynamical system estimation oracle} of $A$ and $B$, and condition on the event it satisfies the entrywise estimation property. Then, the soft thresholded matrices $(\bar{A},\bar{B})$ of $(\widehat{A},\widehat{B})$ satisfy that
\begin{align*}
    \forall i,j\in [d],\ k\in [d_u]: \ A(i,j)=0 \rightarrow \bar{A}(i,j)=0,\ \mathrm{and}\ B(i,k)=0 \rightarrow \bar{B}(i,k)=0.
\end{align*}
By this property, and since the true dynamics is of the form given in \pref{prop: lqr rud and minimal invariant subspaces}, the estimates $\bar{A},\bar{B}$ can be written as follows
\begin{align}
    \bar{A} = P_{\Ical_c} \bar{A} P_{\Ical_c} + P_{\Ical_r}\bar{A} (P_{\Ical_r}-P_{\Ical_c}) + (I-P_{\Ical_r})\bar{A}(I-P_{\Ical_c}) ,\quad \bar{B} = P_{\Ical_B} \bar{B}.\label{eq: approximating PC LQR relation 1}
\end{align}

{\bf Invariance argument for estimated system.} Let $K_\star(\bar{L})$ be the optimal policy of the LQR system $\bar{L} = \rbr{\bar{A}, \bar{B}, I_d}$. This LQR system is also a PC-LQR system by comparing~\eqref{eq: approximating PC LQR relation 1} and the form supplied in \pref{prop: lqr rud and minimal invariant subspaces}. Observe that $\bar{L}$ is a stabilizable PC-LQR. 
\begin{enumerate}
    \item The system that contains the first two blocks  of $\bar{L}$ is stabilizable by utilizing the perturbation result of~\cite{simchowitz2020naive} as we formally establish below in $(*stb)$.
    \item The uncontrolled and non-relevant system $(I-P_{\Ical_r})\bar{A}(I-P_{\Ical_r})$ is stable since 
    $$
    \norm{(I-P_{\Ical_r})\bar{A}(I-P_{\Ical_r})}_\infty \leq \norm{(I-P_{\Ical_r})A(I-P_{\Ical_r})}_\infty \leq 1.
    $$
    The first inequality holds due the soft thresholding which implies that $|\bar{A}(i,j)|\leq |A(i,j)|$ which leads to the inequality. The second inequality holds by~\pref{assum: L1 stability on $A_3$}. Since $\rho((I-P_{\Ical_r})\bar{A}(I-P_{\Ical_r}))\leq \norm{(I-P_{\Ical_r})\bar{A}(I-P_{\Ical_r})}_\infty\leq 1$ we get that uncontrolled and non-relevant is stable.
\end{enumerate}

By the first and second statement of \pref{thm: invariance of optimal policy} the optimal policy is invariant under a change in the dynamics and cost. Let $\bar{L}_{\mathrm{inv}} = \rbr{\bar{A}_{\mathrm{inv}}, \bar{B}, I_{1:2}}$
\begin{align*}
    \bar{A}_{\mathrm{inv}} = P_{\Ical_c} \bar{A} P_{\Ical_c} + P_{\Ical_r}\bar{A} (P_{\Ical_r}-P_{\Ical_c}) ,\quad \bar{B} = P_{\Ical_B} \bar{B},
\end{align*}
that is, when we set $(I-P_{\Ical_r})\bar{A}(I-P_{\Ical_c})=0$, and the cost 
\begin{align*}
    I_{1:2}
    =
    \sbr{
    \begin{matrix}
      I_{s_c} & 0 & 0\\
      0 & I_{s_e} & 0\\
      0 & 0 & 0
    \end{matrix}
    },
\end{align*}
that is, we set the cost of the third block to zero ($I_{1:2}$ is a subset of $I_{1+}$ defined in \pref{thm: invariance of optimal policy}). By \pref{thm: invariance of optimal policy} it holds that
\begin{align}
    K_\star(\bar{L}) = K_\star(\bar{L}_{\mathrm{inv}}). \label{eq: approximating PC LQR relation 3}
\end{align}

{\bf Invariance argument for the true system.} By again applying the first and second statement of \pref{thm: invariance of optimal policy}, we get that the optimal policy of the true system is invariant when transforming it to the LQR system $L_{\mathrm{inv}} = (A_{\mathrm{inv}},B,I_{1:2})$ where
\begin{align*}
    A_{\mathrm{inv}} = P_{\Ical_c} A P_{\Ical_c} + P_{\Ical_r}A (P_{\Ical_r}-P_{\Ical_c}) ,\quad B = P_{\Ical_B} B.
\end{align*}
That is,
\begin{align}
    K_\star(L) = K_\star(L_{\mathrm{inv}}). \label{eq: approximating PC LQR relation 4}
\end{align}

{\bf Perturbation result on invariant systems.} We now apply a perturbation result of~\cite{simchowitz2020naive}, Theorem 5 (which we partially restate in \pref{thm: naive exploration in lqr is optimal} for convenience) on the invariant systems $\bar{L}_{\inv}$ and $L_{\inv}$. First, observe that for both $\bar{L}_{\inv}, L_{\inv}$ the optimal value has the following form
\begin{align*}
    P
    =
    \sbr{
    \begin{matrix}
      P_1 & P_{12} & 0 \\
      P_{12} & P_2 & 0\\
      0 & 0 & 0
    \end{matrix}
    },
\end{align*}
since the cost of the third block is zero $I_{1:2}$, and the dynamics of the third row and column is zero on the invariant systems $L_{\inv}$ and $\bar{L}_{\inv}$. Thus, we can eliminate the third row and third columns of the LQR systems  $\bar{L}_{\inv}$ $ L_{\inv}$ and apply a perturbation bound on the smaller system. Let  $\bar{L}_{\inv,1:2}, L_{\inv,1:2}$ be this restriction.

Observe that the errors of $\bar{L}_{\inv,1:2}$ relatively to $L_{\inv,1:2}$ scales with $\sqrt{s(s+d_u)}\epsilon$, i.e.,
\begin{align*}
    &\norm{\bar{A}_{\mathrm{inv},1:2} - A_{\inv, 1:2}}_F =\norm{\bar{A}_{P_{\Ical_c}} (\bar{A}-A) P_{\Ical_c} + P_{\Ical_r}(\bar{A}-A) (P_{\Ical_r}-P_{\Ical_c})}_F \leq \sqrt{2}s\epsilon\\
    &\norm{\bar{B}_{1:2} -B_{1:2}}_F \leq \sqrt{2sd_u}\epsilon,
\end{align*}
where the $\sqrt{2}$ factor comes from the soft thresholding operations together with the $(\epsilon,\delta)$ element-wise estimation of $(A,B)$. Setting $\epsilon  = \sqrt{\epsilon'/2s(s+d_u)}$ in the element-wise estimation of $(A,B)$, and renaming $\epsilon'$ as $\epsilon$, we get that 
\begin{align*}
    &\max\cbr{\norm{\bar{A}_{\inv,1:2} - A_{\inv,1:2}}_F,\norm{\bar{B} -B}_F}\leq \sqrt{\epsilon}\\
    &\max\cbr{\norm{\bar{A}_{\inv,1:2} - A_{\inv,1:2}}_{\op},\norm{\bar{B} -B}_{\op}}\leq \max\cbr{\norm{\bar{A}_{\inv,1:2} - A_{\inv,1:2}}_F,\norm{\bar{B} -B}_F}\leq  \sqrt{\epsilon},
\end{align*}
since for any matrix $\norm{A}_{\op}\leq \norm{A}_F.$

By Theorem 5 of~\cite{simchowitz2020naive} (see \pref{thm: naive exploration in lqr is optimal}) we get that if $\sqrt{\epsilon} \leq 54 \norm{P_\star(L_{\inv,1:2})}^5_{\op}$ then the optimal policies of $\bar{L}_{\inv,1:2}$ and $L_{\inv,1:2}$ are close and both system are stabilizable (specifically, the first two block of the estimated system $\bar{L}_{\inv,1:2}$ is stable as was needed to show in $(*stb)$). That is,
\begin{align}
    J_\star(K_\star(\bar{L}_{\inv,1:2}); L_{\inv,1:2}) \leq J_\star(K_\star(L_{\inv,1:2}); L_{\inv,1:2})+ 2C_{\mathrm{est}}(A,B) \epsilon, \label{eq: approximating PC LQR relation 5}
\end{align}
where $C_{\mathrm{est}}(A,B) = 142 \norm{P_\star(L_{\inv})}_{\op}^8$. 

Since the optimal policies of the system $L$, $L_{\inv,1:2}$ and $\bar{L}$, $\bar{L}_{\inv,1:2}$ is invariant, the above implies that,
\begin{align*}
    &J_\star(K_\star(\bar{L}); L_{\inv,1:2}) = J_\star(K_\star(\bar{L}_{\inv,1:2}); L_{\inv,1:2}) \tag{By~\eqref{eq: approximating PC LQR relation 3}}\\
    &\leq J_\star(K_\star(L_{\inv,1:2}); L_{\inv,1:2})+ 2C_{\mathrm{est}}(A,B) \epsilon \tag{By~\eqref{eq: approximating PC LQR relation 5}}\\
    & = J_\star(K_\star(L); L_{\inv,1:2})+ 2C_{\mathrm{est}}(A,B)\epsilon \tag{By~\eqref{eq: approximating PC LQR relation 4}}.
\end{align*}

Lastly, since the difference in values between the invariant and original system is a constant, that does not depend on the policy, by the first statement of \pref{thm: invariance of optimal policy}, it holds that
\begin{align*}
    &J_\star(K_\star(\bar{L}); L_{\inv,1:2})  = J_\star(K_\star(\bar{L}); L ) + C = J_\star(K_\star(\bar{L})) + C \\
    &J_\star(K_\star(L); L_{\inv,1:2})  = J_\star(K_\star(L); L ) + C = J_\star + C.
\end{align*}
Combining the above yields that
\begin{align*}
    J_\star(K_\star(\bar{L}))  \leq J_\star +2C_{\mathrm{est}}(A,B)\epsilon.
\end{align*}
\end{proof}

\begin{theorem}[\cite{simchowitz2020naive}, Theorem 5] \label{thm: naive exploration in lqr is optimal} Let $L=(A,B,I_d)$ be a stabilizable system. Given an alternative pair of matrices $\bar{L} = (\bar{A},\bar{B},I_d)$, for each $\circ\in \cbr{\op,F}$ define $\epsilon_\circ = \max\cbr{\norm{A-\widehat{A}}_\circ, \norm{B-\widehat{B}}_\circ}$. Then, if $\epsilon_{\op}\leq 1/54\norm{P_\star}^5_\op$
\begin{align*}
    J_\star(K_\star(\bar{L})) \leq J_\star +  C_{\mathrm{est}} J\epsilon_F^2,
\end{align*}
where $C_{\mathrm{est}} = 142 \norm{P_\star}^8.$
\end{theorem}

\newpage
\section{Learning Element-wise Estimates of a Matrix} \label{supp: results of learning minimal AIP}

\subsection{Diagonal Covariance Matrix}\label{app: disc estimation via thresholding}
In this section, we analyze the sample complexity of of obtaining an $(\epsilon,\delta)$ element-wise good estimate of a matrix assuming that the covariance matrix of $x_0$ is diagonal. 

\theoremDiscViaThreshold*
This result is a direct corollary of \pref{lem: elementwise convergence of second-moment based estimation} as we now show.
\begin{proof}
Observe that we apply random inputs of the form $u_0\sim \Ncal(0,I_d)$, that $x_0 \sim \Ncal (0, \sigma_0 I_d)$ and that $\xi$ is $\sigma$ subgaussian. Thus,
\begin{align}
    x_1 = A x_0 + B u_0+ \xi. \label{eq: model with random inputs}
\end{align}
{\bf Estimation of $A$.} The estimator of $A$ is given by
\begin{align*}
    \widehat{A} = \frac{1}{N\sigma_0^2} \sum_{n} x_{1,n} x^T_{0,n},
\end{align*}
where, by~\eqref{eq: model with random inputs}
\begin{align*}
    x_{1,n} = A x_{0,n} + \xi_{B,n}
\end{align*}
and $\xi_{B,n} = \xi_0 + B u_{0,n}$. Observe that $u_0$ and $x_0$ are i.i.d., and, for any $i$, $\xi_{B,n}(i)$ is a zero mean $\sigma_B$ sub gaussian noise where
\begin{align*}
    \sigma_B = \sqrt{\sigma^2 + \norm{B(i,\cdot)}_2^2}\leq \sqrt{\sigma^2 + B_{\max}^2 d_u}
\end{align*}
Applying \pref{lem: elementwise convergence of second-moment based estimation} directly implies that 
\begin{align*}
    \PP\rbr{\forall i,j\in [d]: \abr{\hat{A}(i,j) - A(i,j)}\geq \frac{5  \gamma_{2,\delta/(6d^2)}(\sigma_{B}+ \sigma_0 \max_{i}\norm{A(i,\cdot)}_2}{\sigma_0\sqrt{N}}}\leq \delta.
\end{align*}

{\bf Estimation of $B$.} The analysis is similar to the first part. The estimator of $B$ is given by
\begin{align*}
    \widehat{B} = \frac{1}{N} \sum_{n} x_{1,n} u^T_{0,n}.
\end{align*}
By~\eqref{eq: model with random inputs}, we see that $x_{1,n}$ can be written as
\begin{align*}
    x_{1,n} = B u_{0,n} + \xi_{A,n},
\end{align*}
and $\xi_{A,n} = \xi_0 + A x_{0,n} $. Since $x_0$ and $u_0$ are i.i.d., and for any $i$  it holds that $\xi_{A,n}(i)$ is zero mean $\sigma_A$ sub gaussian noise where
\begin{align*}
    \sigma_A = \sqrt{\sigma^2 + \norm{A(i,\cdot)}_2^2} \leq \sqrt{\sigma^2 + A^2_{\max}s + 1}, 
\end{align*}
since if $i$ is in the third block it holds that $\sum_{j} (A_3(i,j))^2\leq \rbr{\sum_{j} |A_3(i,j)|}\leq 1$ by~\pref{assum: L1 stability on $A_3$}. Applying \pref{lem: elementwise convergence of second-moment based estimation} directly implies that 
\begin{align*}
    \PP\rbr{\forall i\in [d],j\in [d_u] :\abr{\hat{B}(i,j) - B(i,j)}\geq \frac{5  \gamma_{2,\delta/(6d^2)} (\sigma_{A}+ \max_{i}\norm{B(i,\cdot)}_2}{\sqrt{N}}}\leq \delta.
\end{align*}
Taking a union bound concludes the proof.
\end{proof}

\begin{lemma}[Elementwise Convergence of second-moment Based Estimation] \label{lem: elementwise convergence of second-moment based estimation}
Let $\epsilon,\delta>0$. Let the plug-in estimator of $A$ be given as
\begin{align*}
        &\widehat{A}= \frac{1}{\sigma_0^2 N}\sum_{i=1}^N x_{1,n} x_{0,n}^{\top},
\end{align*}
where $x_1= Ax_0 + \xi$, $x_0\sim \Ncal(0,\sigma_0^2 I)$ and for any $i\in [d]$ it holds that $\xi(i)$ is $\sigma$ sub gaussian. Then, 
\begin{align*}
    \PP\rbr{\forall i,j\in [d]:|\widehat{A}(i,j) - A(i,j)|\geq \frac{5  \gamma_{2,\delta/(6d^2)}(\sigma+\sigma_0\norm{A(i,\cdot)}_2 )}{\sigma_0\sqrt{N}}} \leq \delta.
\end{align*}
\end{lemma}
\begin{proof}
Observe that
\begin{align}
   \widehat{A} = A + \underbrace{\frac{1}{ N} \sum_{n} A\rbr{\frac{x_{0,n} x_{0,n}^{\top}}{\sigma_0^2}- I}}_{(i)} +  \underbrace{\frac{1}{\sigma_0^2 N}\sum_{n}\xi_n x_{0,n}^{\top} }_{(ii)} \label{eq: term 1 bound on forward discrepency relation 1}
\end{align}
We get a point-wise bound for each one of the terms to conclude the proof. \\
{\bf Term $(i)$.} Let $z_n=\frac{1}{\sigma_0}x_{0,n}$ and observe it is $\Ncal(0,I_d)$ gaussian random vector. Fix $i,j\in [d]$. It holds that the $i,j$ entry of term $(i)$ can be written as follows.
\begin{align}
    &\sbr{\frac{1}{N}\sum_{n}A(z_{n} z_{n}^{\top}-I)}_{ij}\\
    &=\underbrace{\frac{1}{N}\sum_{n}\sum_{m\neq j}A(i,m)z_{n}(m) z_{n}(j)}_{(i)} + \underbrace{\frac{1}{N}\sum_{i} A(i,j)(z_{n}(j)^2-1)}_{(ii)} \label{eq: thresholding pointwiese relation 1}.
\end{align}
To bound the first term of~\eqref{eq: thresholding pointwiese relation 1}, we write it as follows
\begin{align*}
    \frac{1}{N}\sum_{n}\sum_{m\neq k}A(i,m)z_{n}(m) z_{n}(j) = \frac{1}{N}\sum_{n}\inner{A(i,[d]/ j)}{z([d]/j)} z_{n}(j).
\end{align*}
Observe that  $\EE[\inner{A(i,[d]/ j)}{z([d]/j)} z_{n}(j)] = \EE[\inner{A(i,[d]/ j)}{z([d]/j)} ] \EE[z_{n}(j)]$ since the first term the vector $z([d]/j)$ does not contain $z_{n}(j)$, and, thus, the two are independent. By
\pref{lem: application of cross correlation} we get that with probability at least $1-\delta$ it holds that
\begin{align*}
    \abr{\frac{1}{N}\sum_{n}\sum_{m\neq j}A(i,m)x_{0,n}(m) x_{0,n}(j)} = \abr{\frac{1}{N}\sum_{n}\inner{A(i,[d]/ j)}{z_{n}([d]/ k)} z_{n}(j)} \leq \frac{3\norm{A(i,[d]/ j)}_2\gamma_{2,\delta/3}}{\sqrt{N}},
\end{align*}
for $N\geq \gamma_{2,\delta/3}.$

We bound the second term of~\eqref{eq: thresholding pointwiese relation 1} by directly applying~\pref{lem: covariance estimation sub gaussian} by which
\begin{align*}
    \abr{\frac{1}{N}\sum_{n} (z_{n}(j)^2-1)}\leq \frac{\gamma_{1,\delta} \abr{ A(i,j)}}{\sqrt{N}},
\end{align*}
with probability greater than $1-\delta$ for $N\geq \gamma_{1,\delta}.$ By taking the union bound on the two events and on all $i,j\in [d]$ we get that for all $i,j\in [d]$
\begin{align*}
    &\abr{A_{kl}\sbr{\frac{1}{N}\sum_{n}A(z_{n} z_{n}^{\top}-I)}_{ij}} \leq \frac{3\gamma_{2,\delta/(6d^2)} (|A(i,j)| + \norm{A(i,[d]/ j)}_2}{\sqrt{N}} \leq \frac{5\gamma_{2,\delta/(6d^2)} \norm{A(i,\cdot)}_2}{\sqrt{N}},
\end{align*}
with probability greater than $1-\delta.$ The last inequality follows from Jensen's inequality,
\begin{align*}
    a_i + \sqrt{\sum_{j\neq i}a_j^2} = \sqrt{a^2_i} + \sqrt{\sum_{j\neq i}a_j^2} \leq \sqrt{2}\sqrt{\sum_{i}a_j^2},
\end{align*}
since $\sqrt{c_1}+\sqrt{c_2} \leq \sqrt{2}\sqrt{c_1+c_2}$.

{\bf Term $(ii).$} See that $|\frac{1}{N}\sum_{n}\xi_{n}(i)x_{0,n}(j)|$ can be bounded by a direct application of \pref{lem: bound on cross correlation} since $\xi,x_{0,n}$ are independent. Specifically, with probability greater than $1-\delta$ it holds that
\begin{align*}
    \abr{\frac{1}{N\sigma_0^2}\sum_{n}\xi_{n}(i)x_{0,n}(j)}\leq \frac{3\sigma  \gamma_{2,\delta/(6d^2)}}{\sigma_0\sqrt{N}}
\end{align*}
for all $i,j\in [d]$ by applying the union bound.

{\bf Combining the two bounds.} By a union bound on the events by which terms $(i)$ and $(ii)$ are bounded, we get that for  $N\geq \gamma_{d}$ it holds that
\begin{align*}
    \PP\rbr{\forall i,j\in [d]:\abr{\widehat{A}(i,j) - A(i,j)}\geq \frac{5  \gamma_{2,\delta/(6d^2)}(\sigma+\sigma_0\norm{A(i,\cdot)}_2 )}{\sigma_0\sqrt{N}}} \leq \delta.
\end{align*}
\end{proof}

\begin{lemma}\label{lem: application of cross correlation}
Let $\cbr{x_{n}}_{n=1}^N$ be an i.i.d. vector such that $x_n\sim \Ncal(0,\sigma_{x} I_d)$, and let $\cbr{y_n}_{n=1}^N$ be i.i.d. $\sigma_y$ subgaussian, zero mean, random variables and assume that $N\geq \gamma_{2,\delta/3}$. Let $a\in \mathbb{R}^d$. Then, 
\begin{align*}
    \PP\rbr{\abr{\frac{1}{N}\sum_{n} \inner{a}{x_n}y_{n}} \geq  \frac{3\sigma_{x}\sigma_y\norm{a}_2\gamma_{2,\delta/3}}{\sqrt{N}}} \leq \delta. 
\end{align*}
with probability greater than $1-\delta$.
\end{lemma}
This result is a direct application of \pref{lem: bound on cross correlation} as we now show.
\begin{proof}
Observe that  $z_{1,n}=\frac{1}{\sigma_{x}\norm{a}_2}\inner{a}{x_n}, z_{2,n} = \frac{1}{\sigma_y} n_i$ are both $1$-sub-gaussian random variable with zero mean. Thus, to prove this result we can bound
\begin{align*}
    &\PP\rbr{\abr{\frac{1}{N}\sum_{n} \inner{a}{x_n}y_{n}} \geq \frac{3\sigma_{x}\sigma_y\norm{a}_2\gamma_{2,\delta/3}}{\sqrt{N}} }  =\PP\rbr{\abr{\frac{1}{N}\sum_{i} z_{1,n}z_{2,n}} \geq  \frac{3\gamma_{2,\delta/3}}{\sqrt{N}}},
\end{align*}
where both $z_{1,n}$ and $z_{2,n}$ are independent. Thus, we can apply \pref{lem: bound on cross correlation} while setting $d_1=d_2=1$ and conclude the proof.
\end{proof}

\newpage
\subsection{Positive Definite Covariance Matrix}\label{app: disc estimation via ls with nuisance parameter}
We now analyze the sample complexity of obtaining an element-wise good estimation of a matrix assuming that the covariance matrix of $x_0$ is PD. This result is a corollary of a careful semiparametric LS analysis we supply in the next section.

\corrDiscViaNuisanceLS*
\begin{proof}
Fix an $i,j\in [d]$. For any such $i,j$ we can estimate $A(i,j)$ via a semiparametric LS where the model is
\begin{align*}
        x_{1}(i) = A(i,j) x_0(j) + \inner{A(i, [d]/ j)}{x_0([d]/j)} +\xi_i.
\end{align*}
Applying \pref{prop: semiparameteric least model sample complexity} and setting $d_w = 1,d_e=d-1\leq d$ and $|w_\star|=|A(i,j)|\leq A_{\max}$ yields the bound for any any fixed $i,j\in [d]$. Applying the union bound on all $i,j\in [d]$ concludes the proof for estimating matrix $A$.
\end{proof}

\subsubsection{Semiparametric Least Squares for Linear Model}\label{app: semi parametric ls linear model}

\begin{algorithm}[t]
\caption{Semiparametric Least Squares}\label{alg: semi parametric regression}
\begin{algorithmic}[1]
\State {\bf Require:} Number of samples $N>0$, row and column indices $i,j\in [d]$
\State Sample $\cbr{(y_n , x_{1,n} , x_{2,n})}_{n=1}^{2N}$
\State Estimate cross correlation $\widehat{L} = \rbr{\sum_{n=1}^{N} x_{1,n}x_{2,n}^{\top}} \rbr{\sum_{n=1}^{N} x_{2,n}x_{2,n}^{\top}}^\dagger $
\State Estimate conditional output
$
    \widehat{c} = \rbr{\sum_{n=1}^{N} x_{2,n}x_{2,n}^{\top}}^\dagger \rbr{\sum_{n=1}^{N} y_n x_{2,n}} 
$
\State Estimate $\widehat{w}$ through plug-in
$$
\widehat{w} =\rbr{\sum_{n=N+1}^{2N} (x_{1,n} -  \widehat{L}x_{2,n})(x_{1,n} -  \widehat{L}x_{2,n})^{\top}}^\dagger \rbr{\sum_{n=N+1}^{2N} \rbr{y_n - \inner{\widehat{c}}{x_{2,n}}} (x_{1,n} -  \widehat{L}x_{2,n})}
$$
\State {\bf Output:} $\widehat{w}$
\end{algorithmic}
\end{algorithm}

Consider the following model
\begin{align}
    y = \inner{w_\star}{x_1} + \inner{e_\star}{x_2} +\epsilon \label{eq: ls with nuisnace parameter for linear model},
\end{align}
where $x_1\in \mathbb{R}^{d_w},x_2\in \mathbb{R}^{d_e}$ and $\epsilon$ is a zero mean $\sigma$ sub-gaussian noise. Furthermore, assume that the covariance matrix of $[x_1 \quad x_2]$ is PD, that is $\Sigma = \EE\sbr{[x_1 \quad x_2]^{\top}[x_1 \quad x_2]}$ is PD. Our goal is to recover $w_\star$ by accessing tuples of $\cbr{y,x_1,x_2}$, and, to achieve improved rates relatively to estimation of the entire vector $[w_\star\quad e_\star]$.

Observe that would $x_1,x_2$ be uncorrelated, LS regression of $y$ given $x_1$ achieves our goal. With this observation, a natural first step would be to orthogonalize the model as we now show. Since $x=(x_1,x_2)$ is normally distributed it holds that 
\begin{align}
    \EE[x_1|x_2] = \Sigma_{12}\Sigma_{2}^{-1} x_2 \equiv L_\star x_2, \label{eq: model relation x1 to x2}
\end{align}
where $L_\star\in \mathbb{R}^{d_w \times d_e}$, from which we get, by linearity of expectation, that
\begin{align}
    \EE[y|x_2] = \inner{w_\star}{L_\star x_2} + \inner{e_\star}{x_2} \equiv \inner{c_\star}{x_2}, \label{eq: nuisance paratmeter conditional mean of y}
\end{align}
where $c_\star = L_\star^Tw_\star + e_\star$. Using this, the model~\eqref{eq: ls with nuisnace parameter for linear model} can be written as follows.
\begin{align}
    y &= \inner{w_\star}{x_1} + \inner{e_\star}{x_2} +\epsilon \nonumber \\
    & = \inner{w_\star}{(x_1 - L_\star x_2)} + \inner{c_\star}{x_2}  +\epsilon. \label{eq: nuisnace parameter general model equivalent model assumption}
\end{align}
Unlike in~\eqref{eq: ls with nuisnace parameter for linear model} where the features are not orthogonal $\EE[x_1x_2^{\top}] = \Sigma_{12}\Sigma_{2}^{-1} \neq 0$, in this new representation, the features are orthogonal since
\begin{align}
    \EE[(x_1 - L_\star x_2)x_2^{\top}] =0, \label{eq: appendix orthogonal features}
\end{align}
by construction. Thus, if we define $z_1 = x_1 - L_\star x_2$ and $z_2 = x_2$ we get that~\eqref{eq: nuisnace parameter general model equivalent model assumption} is given by 
\begin{align*}
    y = \inner{w_\star}{z_1} + \inner{c_\star}{z_2}  +\epsilon,
\end{align*}
where $z_1,z_2$ are orthogonal and their covariance matrix  is given by
\begin{align}
    &\cov_{z_1} = \EE[(x_1 - L_\star x_2)(x_1 - L_\star x_2)^{\top}] = \Sigma_{1} - \Sigma_{12}\Sigma_{2}^{-1}\Sigma_{12}^{\top} \equiv \Sigma /\Sigma_{2},\nonumber \\
    &\cov_{z_2} = \Sigma_2 \label{eq: covariance matrix of orthogonalized model},
\end{align}
where $\Sigma/\Sigma_2$ is the known as the Schur complement.

Importantly, would we be given $L_\star$ and $c_\star$, we can get an unbiased estimate of $w_\star$ using the data set $\cbr{(y_n,x_{1,n},x_{2,n})}_{n=1}^N$ through an ordinary least-squares approach,
\begin{align}
    \widehat{w}=\rbr{\sum_{n=N+1}^{2N} (x_{1,n} -  L_\star x_{2,n})(x_{1,n} -  L_\star x_{2,n})^{\top}}^\dagger \rbr{\sum_{n=N+1}^{2N} \rbr{y_n - \inner{ c_\star }{x_{2,n}}} (x_{1,n} -  L_\star x_{2,n})}
. \label{eq: ols exact case nuisance parameters}
\end{align}
It can be shown that $\EE[\widehat{w}]=w_\star$ when the design matrix $V_N = \sum_{n=1}^N (x_{1,n}-L_\star x_{2,n}) (x_{1,n}-L_\star x_{2,n})^{\top}$ is PD. This fact, motivates us to study the \emph{finite} sample performance of this approach when both $L_\star$ and $c_\star$ are estimated from data (\pref{alg: semi parametric regression}). In the next section, we study this estimator without any assumption besides of positive minimal eigenvalue of the covariance matrix of $x=(x_1,x_2).$ 

\subsubsection{Finite Sample Analysis: Semiparametric LS} \label{app: finite sample analysis, semi parametric LS}

We are now ready to analyze the performance of \pref{alg: semi parametric regression}. Relaying on the OLS~\eqref{eq: ols exact case nuisance parameters}, \pref{alg: semi parametric regression} splits the data in two, with the first dataset it estimates $L_\star$ and $c_\star$. With the second dataset, it solves the OLS~\eqref{eq: ols exact case nuisance parameters} in which the exact $L_\star$ and $c_\star$ are replaced by their estimators.

The following lemma establishes a finite performance guarantee of \pref{alg: semi parametric regression}. Importantly, we see that there's only a lower order dependence in $d_e$ which we suffer due to the need to estimate $L_\star$ and~$c_\star$. 

\theoremDiscViaNuisanceLS*
\paragraph{Overview of the analysis of~\pref{prop: semiparameteric least model sample complexity}}. We decompose the error into three terms in~\eqref{eq: central term central theorem nuisnace full analysis}. The first term is of dimension $s$ (as oppose to $d$) and is bounded via standard concentrations for least-squares~\cite{hsu2012tail}. The second and third terms are errors we suffer due to in-exact estimation of $L_\star$ and $c_\star$.

Importantly, we bound the errors in the estimates of $L_\star$ and $c_\star$ in weighted norms. Specifically, we show we can bound
\begin{align*}
    \norm{\Sigma/ \Sigma_{2}^{-1/2}(\widehat{L} - L_\star)\Sigma^{1/2}_2}_{\op} \textit{ and } \norm{\Sigma_{2}^{1/2}(\widehat{c} - c_\star)}_{2}
\end{align*}
by a term which is independent of minimal eigenvalues of $\Sigma$ or $\Sigma / \Sigma_2$. With this at hand, and by further careful analysis, we show, that the second and third terms in~\eqref{eq: central term central theorem nuisnace full analysis} can be bounded by terms that are independent of minimal eigenvalues. The final result follows by relating the minimal and eigenvalues of $\Sigma/\Sigma_2$ to the of $\Sigma$, supplied in~\cite{smith1992some}.

\begin{proof}
The OLS solution $\widehat{w}$ satisfies the following relation
\begin{align}
    &\sum_{n=N+1}^{2N} \frac{1}{N}\sbr{\rbr{x_{1,n}-\widehat{L}x_{2,n}}\rbr{x_{1,n}-\widehat{L}x_{2,n}} }^{\top}(w_\star - \widehat{w}) \nonumber\\
    &=\sum_{n=N+1}^{2N} \frac{1}{N}\rbr{x_{1,n}-\widehat{L}x_{2,n}}\rbr{y_n - \inner{w_\star}{x_{1,n}-\widehat{L}x_{2,n}} - \inner{\widehat{c}}{x_{2,n}}}. \label{eq: central lemma orthogonalied ls relation 1}
\end{align}
Let $$\widehat{z}_{1,n}(\widehat{L}) = x_{1,n}-\widehat{L}x_{2,n},\ z_{1,n} = x_{1,n}-L_{\star}x_{2,n},\ z_{2,n} = x_{2,n}$$ and define the design matrix as $$V_N = \frac{1}{N}\sum_{n=N+1}^{2N} \widehat{z}_{1,n}(\widehat{L})  \widehat{z}_{1,n}(\widehat{L}) ^{\top}.$$ By multiplying both sides of this relation by $\Sigma/\Sigma_2(\widehat{L})^{-1/2} = \EE[\widehat{z}_{1,n}(\widehat{L}) \widehat{z}_{1,n}(\widehat{L}) ^{\top}]$ and by some additional algebraic manipulations, it can be shown that~\eqref{eq: central lemma orthogonalied ls relation 1} implies that
\begin{align}
    &\Sigma/\Sigma_2(\widehat{L})^{-1/2} V_N (w_\star - \widehat{w}) \nonumber \\
    &= \Sigma/\Sigma_2(\widehat{L})^{-1/2}\sum_{i=1}^N\frac{1}{N} \rbr{x_{1,n}-\widehat{L}x_{2,n}}\rbr{y_i - \inner{w_\star}{x_{1,n}-L_\star x_{2,n}} - \inner{c_\star}{x_{2,n}}} \nonumber\\
    &\quad +\Sigma/\Sigma_2(\widehat{L})^{-1/2}\sum_{i=1}^N \frac{1}{N}\rbr{x_{1,n}-L_\star x_{2,n}}\inner{(c_\star - \widehat{c}) + (\widehat{L}-L_\star)^{\top} w_\star}{x_{2,n}} \nonumber \\
    & \quad + \Sigma/\Sigma_2(\widehat{L})^{-1/2}\sum_{i=1}^N \frac{1}{N}\rbr{(\widehat{L}-L_\star )x_{2,n}}\inner{(c_\star - \widehat{c}) + (\widehat{L}-L_\star)^{\top} w_\star}{x_{2,n}}\nonumber\\
    &= \Sigma/\Sigma_2(\widehat{L})^{-1/2}\sum_{i=1}^N\frac{1}{N} \widehat{z}_{1,n}(\widehat{L})  \epsilon_n \nonumber\\
    &\quad +\Sigma/\Sigma_2(\widehat{L})^{-1/2}\sum_{i=1}^N \frac{1}{N}z_{1,n}\inner{(c_\star - \widehat{c}) + (\widehat{L}-L_\star)^{\top} w_\star}{z_{2,n}} \nonumber \\
    & \quad + \Sigma/\Sigma_2(\widehat{L})^{-1/2}\sum_{i=1}^N \frac{1}{N}\rbr{(\widehat{L}-L_\star )z_{2,n}}\inner{(c_\star - \widehat{c}) + (\widehat{L}-L_\star)^{\top} w_\star}{z_{2,n}}\label{eq: nuisance parameter approach}
\end{align}
where the last relation holds since ${y_i - \inner{w_\star}{x_{1,n}-L_\star x_{2,n}} - \inner{c_\star}{x_{2,n}} = \epsilon_n}$ due to the model assumption~\eqref{eq: nuisnace parameter general model equivalent model assumption}. Observe we obtained a vector equality of the form
\begin{align*}
    a = b_1 + b_2 + b_3
\end{align*}
where $a,b_1,b_2,b_3\in \mathbb{R}^{d_w}$. This equality implies that
\begin{align}
    \norm{a}_2 = \norm{b_1 + b_2 + b_3}_2\leq \norm{b_1}_2 + \norm{b_2}_2 + \norm{b_3}_2 \label{eq: nuisance squared L2 convexity consequnce}
\end{align}
due to the triangle inequality. Hence, the vector equality in~\eqref{eq: nuisance parameter approach} together with \eqref{eq: nuisance squared L2 convexity consequnce} implies that
\begin{align}
    &\norm{\Sigma/\Sigma_2(\widehat{L})^{-1/2} V_N(\widehat{w}-w_\star)}_2 \nonumber \\
    &=\underbrace{\norm{\sum_{i=1}^N\frac{1}{N} \widehat{z}_{1,n}(\widehat{L})  \epsilon_n}_{\Sigma/\Sigma_2(\widehat{L})^{-1}} }_{(i)} \nonumber \\
    &+\underbrace{\norm{\sum_{i=1}^N \frac{1}{N}z_{1,n}\inner{(c_\star - \widehat{c}) + (\widehat{L}-L_\star)^{\top} w_\star}{z_{2,n}}}_{\Sigma/\Sigma_2(\widehat{L})^{-1} } }_{(ii)}\nonumber\\
    &+\underbrace{\norm{\sum_{i=1}^N \frac{1}{N}\rbr{(\widehat{L}-L_\star )z_{2,n}}\inner{(c_\star - \widehat{c}) + (\widehat{L}-L_\star)^{\top} w_\star}{z_{2,n}} }_{\Sigma/\Sigma_2(\widehat{L})^{-1} }}_{(iii)}. \label{eq: central term central theorem nuisnace full analysis}
\end{align}
We bound each one of these terms by \pref{lem: ls nuisance 3 terms, term 1},  \pref{lem: ls nuisance 3 terms, term 2} and  \pref{lem: ls nuisance 3 terms, term 3}. We verify the conditions of these lemmas hold.
\begin{enumerate}
    \item $N\geq 9 \gamma_{d,\delta}^2\geq  \gamma_{d,\delta}$, by assumption and since $\delta\in (0,e^{-1})$ (see that $9 \gamma_{d,\delta}^2=\Theta(d\log(\frac{1}{\delta}))$).
    \item By \pref{lem: sample complexity c star}
        \begin{align*}
            &\norm{\Sigma_{2}^{1/2}((c_\star - \widehat{c}) + (\widehat{L}-L_\star)^{\top} w_\star)} \\
            &\leq \norm{\Sigma_{2}^{1/2}((c_\star - \widehat{c})}_2 \norm{\Sigma_{2}^{1/2}(\widehat{L}-L_\star)^{\top}\Sigma/\Sigma_{2}^{-1/2}}_{\op} \norm{\Sigma_{2}^{1/2}w_\star}_2\\
            &\leq 10\sigma_c \sqrt{\frac{d d_w \log\rbr{\frac{d_w}{\delta}}}{N}} = \Delta
        \end{align*}
        and  by \pref{lem: sample complexity L star} 
        \begin{align*}
            \norm{\rbr{\Sigma / \Sigma_2}^{-1/2}(\widehat{L} - L_\star)\Sigma_2^{1/2} }_{\op}\leq 5\sigma_c \sqrt{\frac{d d_w \log\rbr{\frac{d_w}{\delta}}}{N}} =\Delta_L = \Delta/2.
        \end{align*}
    with probability greater than $1-\delta$. Thus, we approximate $c_\star,L_\star$ in the scaled norms by the covariance matrices $\Sigma_2$ and $\Sigma / \Sigma_2$.
    \item Since $\Delta_L^2= 25\sigma^2_cd d_w \log\rbr{\frac{d_w}{\delta}}/N$ it holds that for $N\geq 50\sigma^2_cd d_w \log\rbr{\frac{d_w}{\delta}} /\lambda_{\min}(\Sigma)$ the covariance matrix $\Sigma/\Sigma_{2}(\widehat{L})$ is PD, and specifically, 
    \begin{align}
        \lambda_{\min}(\Sigma/\Sigma_{2}(\widehat{L})) \geq \lambda_{\min}(\Sigma/\Sigma_{2})/2 >   \lambda_{\min}(\Sigma)/2  >0, \label{eq: approximated covariance is PD}
    \end{align}
    where the first relation holds by \pref{lem: lower bound on the perturbed least square} while setting $\Delta_L^2=\lambda_{\min}(\Sigma/\Sigma_2)/2$, and the second relation by standard fact on the Schur complement of a matrix (see~\cite{smith1992some}, Theorem 5).
    \item For $N\geq 50\sigma^2_cd d_w \log\rbr{\frac{d_w}{\delta}} /\lambda_{\min}(\Sigma/\Sigma_2)$ by the third relation of \pref{lem: lower bound on the perturbed least square} we get that\\ ${\norm{\rbr{\Sigma/\Sigma_{2}(\widehat{L})}^{-1/2} \Sigma/\Sigma_{2}^{1/2}}_{\op}\leq \sqrt{2}.}$
\end{enumerate}
Observe that by taking
\begin{align}
    N \geq  O\rbr{\rbr{\sigma^2_c/\lambda_{\min}(\Sigma) \vee 1}d d_w \log\rbr{\frac{d_w}{\delta}}}, \label{eq: N minimal value semi parametric}
\end{align}
we satisfy all the requirements on the sample size.

Applying the union bound on all the above and scaling $\delta\gets \delta/3$ we get that all the events hold with probability greater than $1-\delta$. We refer to this event as the first good event $\Gcal_1$. We can now apply \pref{lem: ls nuisance 3 terms, term 1},  \pref{lem: ls nuisance 3 terms, term 2} and  \pref{lem: ls nuisance 3 terms, term 3} and bound~\eqref{eq: central term central theorem nuisnace full analysis} conditioning on $\Gcal_1$. By applying these lemmas and using the union bound we get that with probability greater than $1-\delta$
\begin{align}
    &\norm{\Sigma/\Sigma_2(\widehat{L})^{-1/2} V_N(\widehat{w}-w_\star)}_2 \nonumber \\
    &\leq 3\sigma \sqrt{\frac{d_w\log\rbr{\frac{6}{\delta}}}{N}} + \frac{5\Delta \gamma_{d,\delta/9}}{\sqrt{N}} + \frac{5\Delta^2/2 \gamma_{d,\delta/9}}{\sqrt{N}}  + \Delta^2 \nonumber \\
    &\leq 3\sigma \sqrt{\frac{d_w\log\rbr{\frac{6}{\delta}}}{N}} + 50 \gamma_{d,\delta/9}\sigma_c\sqrt{\frac{d_wd\log\rbr{\frac{6d_w}{\delta}}}{N^2}} + \frac{250 \gamma_{d,\delta/9}\sigma_c^2d d_w \log\rbr{\frac{6d_w}{\delta}}}{N^{3/2}}  +\frac{100\sigma^2_c d d_w \log\rbr{\frac{6d_w}{\delta}}}{N}\label{eq: w widehat c widehat L widehat central bound final}
\end{align}
by plugging the form of $\Delta$, $\Delta_L$ and using $d_e+d_w = d$.

Finally, we translate this bound to a bound on $\norm{\widehat{w}-w_\star}_{\Sigma/\Sigma_2(\widehat{L})}$ by applying \pref{lem: translating empirical to expected performance}. We now verify the conditions of this lemma.
\begin{enumerate}
    \item The matrix $\Sigma/\Sigma_2(\widehat{L})$ is PD by~\eqref{eq: approximated covariance is PD}.
    \item The empirical covariance is concentrated around the true one,
    \begin{align}
        &\norm{\rbr{\Sigma/ \Sigma_2(\hat{L})}^{-1/2}V_N \rbr{\Sigma/ \Sigma_2(\hat{L})}^{-1/2}- I }_{\op} \leq \frac{\gamma_{d_e,\delta}}{\sqrt{N}}\leq \frac{1}{3}, \label{eq: central theorem application of covariance matrix concentration}
    \end{align}
    with probability greater than $1-\delta$ by \pref{lem: covariance estimation sub gaussian}, and the second inequality holds since $N\geq 9 \gamma^2_{d,\delta}\geq9 \gamma^2_{d_e,\delta}.$
\end{enumerate}
Applying  \pref{lem: translating empirical to expected performance} with $c=1/3$, while using the bound in~\eqref{eq: w widehat c widehat L widehat central bound final} we get
\begin{align*}
    &\norm{(\widehat{w}-w_\star)}_{\Sigma/\Sigma_2(\widehat{L})} \\
    &\leq 6\sigma \sqrt{\frac{d_w\log\rbr{\frac{6}{\delta}}}{N}} + \frac{80 \gamma_{d,\delta/9}\sigma_c\sqrt{d_wd\log\rbr{\frac{6d_w}{\delta}}}}{N} + \frac{400 \gamma_{d,\delta/9}\sigma_c^2d d_w \log\rbr{\frac{6d_w}{\delta}}}{N^{3/2}}  + \frac{150\sigma^2_cd d_w \log\rbr{\frac{6d_w}{\delta}}}{N}.
\end{align*}
Furthermore, observe that $$\lambda_{\min}(\Sigma/\Sigma_2(\widehat{L})) \geq  \frac{1}{2}\lambda_{\min}(\Sigma/\Sigma_2)\geq \lambda_{\min}(\Sigma)$$ where the first relation holds by, and the second by identifies of the Schur complement of a PD matrix~\eqref{eq: approximated covariance is PD} and the second relation by~\cite{smith1992some}, Theorem 5. Thus,
\begin{align}
    &\norm{(\widehat{w}-w_\star)}_2\leq \frac{\sqrt{2}}{\sqrt{\lambda_{\min}(\Sigma)}}\norm{(\widehat{w}-w_\star)}_{\Sigma/\Sigma_2(\widehat{L})} \nonumber \\
    &\leq \frac{1}{\sqrt{\lambda_{\min}(\Sigma)}}\rbr{9 \sqrt{\frac{\sigma^2 d_w\log\rbr{\frac{1}{\delta}}}{N}} +  \frac{400 \gamma_{d,\delta/9}\sigma_c^2d d_w \log\rbr{\frac{d_w}{\delta}}}{N^{3/2}} +\frac{230\rbr{\sigma^2_c \vee  \sigma_c} d d_w \log\rbr{\frac{9d_w}{\delta}}}{N}} \label{eq: final bound semi parametric LS}.
\end{align}
Lastly, by the choice of $N$ given in~\eqref{eq: N minimal value semi parametric} and the definition of $\gamma_{d,\delta/9} = O(\sqrt{d \log\rbr{\frac{1}{\delta}}})$ (see \pref{lem: covariance estimation sub gaussian}) it holds that, 
\begin{align*}
    \frac{\gamma_{d,\delta}}{\sqrt{N}}\leq O(1).
\end{align*}
Thus, the last two term of~\eqref{eq: final bound semi parametric LS} are related by a multiplicative constant factor. This concludes the proof.
\end{proof}

\subsubsection{Analysis of the First Phase Errors}


\begin{lemma}[Sample Complexity of Learning $c_\star$]\label{lem: sample complexity c star}
Let $\delta\in(0,e^{-1})$ and let $\Sigma_{2} = \EE[x_2x_2^{\top}]$. Assume that $N\geq 9\gamma^2_{d_e,\delta}$. Then, with probability greater than $1-\delta$ it holds that
\begin{align*}
    \norm{\Sigma_{2}^{1/2}(\widehat{c} - c_\star)}_{2} \leq 5\sigma_c\sqrt{\frac{ d_e \log\rbr{\frac{2}{\delta}}}{N}},
\end{align*}
where $\sigma_c^2=  \norm{w_\star}_{\Sigma/\Sigma_{2}}^2 + \sigma^2$.
\end{lemma}
\begin{proof}
This result is a direct application of \pref{prop: regularized least square general result} which establishes performance guarantee on the OLS. We show that $y= \inner{c_\star}{x_2} + \epsilon$ to apply this result. See that $$y_n = \EE[y|x_{2,n}] + \EE[y|x_{2,n}] - y_n= \inner{c_\star}{x_2} +\epsilon_{c,n}$$
where $\epsilon_{c,n} = y_n - \inner{c_\star}{x_{2,n}}$ is a $\sigma_c = \sqrt{\norm{w_\star}_{\Sigma/\Sigma_{2}}^2 + \sigma^2}$ sub gaussian, zero mean random variable. Indeed, 
\begin{align*}
    &\EE[\epsilon_{c,n}] = \EE[y_n - \inner{c_\star}{x_{2,n}}] = \EE[y_n - \EE[y_n|x_{2,n}]] =0.
\end{align*}
To see it is a $\sigma_c$ sub gaussian observe that $$\epsilon_{c,n} = y_n - \inner{c_\star}{x_{2,n}} = \norm{x_{1,n} - L_\star x_{2,n},w_\star} + \epsilon_n.$$
Thus, and due to the independence of $\epsilon$ and $(x_1,x_2)$ we get
\begin{align*}
    &Var(\epsilon_{c,n}) = Var(\inner{x_{1,n} - L_\star x_{2,n}}{w_\star}) +Var(\epsilon_n) \\
    &= Var(\inner{\rbr{\Sigma/\Sigma_{2}}^{-1/2}(x_{1,n} - L_\star x_{2,n})}{\rbr{\Sigma/\Sigma_{2}}^{1/2}w_\star}) + \sigma^2\\
    &=\norm{\Sigma/\Sigma_{2}^{1/2}w_\star}^2 + \sigma^2 = \norm{w_\star}_{\Sigma/\Sigma_{2}}^2 + \sigma^2.
\end{align*}
Thus, the claim follows from \pref{prop: regularized least square general result}.
\end{proof}

\begin{lemma}[Sample Complexity of Learning $L_\star$]\label{lem: sample complexity L star}
Let $\delta\in(0,e^{-1})$ and let $\Sigma_{2} = \EE[x_2x_2^{\top}]$. Assume that $N\geq 9\gamma^2_{d_e,\delta}$. Let the OLS estimate of $L_\star$ be
\begin{align*}
    \widehat{L} = \frac{1}{N_1} \sum_{n=1}^{N_1} x_{1,n}x_{2,n}^{\top} V_{N_1,x_2}^{\dagger}.
\end{align*}
Then, with probability greater than $1-\delta$ it holds that
\begin{align*}
    \norm{\Sigma/ \Sigma_{2}^{-1/2}(\widehat{L} - L_\star)\Sigma^{1/2}_2}_{\op} \leq 5 \sqrt{\frac{ d_w d_e \log\rbr{\frac{2d_w}{\delta}}}{N}}.
\end{align*}
\end{lemma}
\begin{proof}
We apply the concentration result on the OLS estimator, \pref{prop: regularized least square general result}. To see it is applicable, we reduce this problem to a single parameter estimation. First, bound the operator norm by the Frobenius norm. Let $e_i\in \mathbb{R}^{d_w}$ be a one hot vector with one at its $i^{th}$ entry. Then,
\begin{align}
    &\norm{\Sigma/ \Sigma_{2}^{-1/2}(\widehat{L} - L_\star)\Sigma^{1/2}_2}^2_{\op} \leq \norm{\Sigma/ \Sigma_{2}^{-1/2}(\widehat{L} - L_\star)\Sigma^{1/2}_2}^2_{F} = \sum_{i=1}^{d_w} \norm{e_i^{\top}\rbr{\Sigma/ \Sigma_{2}^{-1/2}(\widehat{L} - L_\star)\Sigma^{1/2}_2}}_2^2 \label{eq: L widehat bound first phase relation 1}.
\end{align}
Observe that the following vector equality holds by the model assumption~\eqref{eq: model relation x1 to x2}.
\begin{align}
    x_1 = \EE[x_1|x_2]  + x_1 - \EE[x_1|x_2] = L_\star x_2 + (x_1 - L_\star x_2), \label{eq: Lstar error propogation relation 1 5}
\end{align}
where $\EE[x_1 - L_\star x_2]=0,\ \cov(x_1 - L_\star x_2) = \Sigma / \Sigma_{2}$ (see~\eqref{eq: covariance matrix of orthogonalized model}). Multiplying~\eqref{eq: Lstar error propogation relation 1 5} from the left by $e_i^{\top}\rbr{\Sigma / \Sigma_{2}}^{-1/2}$, we get that for any $i\in [d_w]$
\begin{align*}
    y_{n,i} \equiv e_i^{\top}\rbr{\Sigma / \Sigma_{2}}^{-1/2}x_{1,n} = e_i^{\top}\rbr{\Sigma / \Sigma_{2}}^{-1/2} L_\star x_2 + e_i^{\top}\epsilon_n = \inner{\beta_i}{x_{2,n}} + \epsilon_{n,i},
\end{align*}
where 
\begin{align*}
    &\beta_i =e_i^{\top}\rbr{\Sigma / \Sigma_{2}}^{-1/2} L_\star,\\
    &\epsilon_{n,i} = e_i^{\top}\rbr{\Sigma / \Sigma_{2}}^{-1/2}(x_1 - L_\star x_2),  
\end{align*}
and $\epsilon_{n,i}$ is zero mean with a unit variance, since,
\begin{align}
    &\EE[\epsilon_{n,i}] = e_i^{\top}\rbr{\Sigma / \Sigma_{2}}^{-1/2}\EE[x_1 - L_\star x_2]=0 \nonumber \\
    &Var(\epsilon_n) = e_i^{\top}\rbr{\Sigma / \Sigma_{2}}^{-1/2}\rbr{\Sigma / \Sigma_{2}}\rbr{\Sigma / \Sigma_{2}}^{-1/2}e_i= 1. \label{eq: matrix estiamtion row with unit variance}
\end{align}
Observe that the ordinary least square estimator of $\beta_i^{\top}$ is given by the following equivalent forms
\begin{align*}
    \widehat{\beta}_i^{\top} &= \sum_{n=1}^N  y_{n,i}x_{2,n}^{\top} V_{N,2}^{\dagger}\\
    &=\sum_{n=1}^N e_i^{\top}\rbr{\Sigma / \Sigma_{2}}^{-1/2}x_{1,n}x_{2,n}^{\top} V_{N,2}^{\dagger}\\
    &=e_i^{\top}\rbr{\Sigma / \Sigma_{2}}^{-1/2} \widehat{L}_N.
\end{align*}
By applying  the concentration result for OLS, \pref{prop: regularized least square general result}, and applying the union bound, we get that for all $i\in [d_w]$, assuming $N\geq 9\gamma_{d,\delta}^2$
\begin{align*}
    \norm{e_i^{\top}\rbr{\Sigma / \Sigma_{2}}^{-1/2} (\widehat{L}_N-L_\star)} = \norm{ (\widehat{\beta}_i - \beta_i)^{\top}} \leq 5\sigma_{i}^2\frac{d_e\log\rbr{\frac{d_w}{\delta}}}{N}\leq 5\frac{d_e\log\rbr{\frac{d_w}{\delta}}}{N}.
\end{align*}
Thus,
\begin{align*}
    &\eqref{eq: L widehat bound first phase relation 1} = \sum_{i=1}^{d_w} \norm{e_i^{\top}\rbr{\Sigma/ \Sigma_{2}^{-1/2}(\widehat{L} - L_\star)\Sigma^{1/2}_2}}^2_2 \\
    &\leq  \sum_{i=1}^{d_w} \frac{ 25\sigma_{i,l}d_e \log\rbr{\frac{d_w}{\delta}}}{N}\\
    & =  \frac{ 25d_w d_e \log\rbr{\frac{d_w}{\delta}}}{N}, \tag{By~\eqref{eq: matrix estiamtion row with unit variance} $\sigma_{i,l}=1$ for all $i\in [d_w]$}
\end{align*}
which concludes the proof.
\end{proof}

\subsubsection{Analysis of the Second Phase Errors}

\begin{lemma}[Bound on First Term of~\pref{prop: semiparameteric least model sample complexity}]\label{lem: ls nuisance 3 terms, term 1} Let $\delta\in (0,e^{-1})$. Assume that $\Sigma/\Sigma_2(\widehat{L}) $ is invertible. Then, with probability greater then $1-\delta$ it holds that $$\norm{\frac{1}{N}\sum_{n} \rbr{x_{1,n}-\widehat{L}x_{2,n}}\epsilon_n}_{\Sigma/\Sigma_2(\widehat{L})^{-1} }\leq 3\sigma \sqrt{\frac{d_w\log\rbr{\frac{1}{\delta}}}{N}}.$$
\end{lemma}
\begin{proof}
This term cab be directly bounded by applying \pref{lem: useful concentration for OLS}. Let $z_n =\frac{1}{\sqrt{N}}( x_{1,n}-\widehat{L}x_{2,n})$ and define $Z_N\in \mathbb{R}^{N\times d_w}$ as the matrix with $z_i$ as its rows. Observe that with this notation $V_N = \sum_n z_n z_n^{\top} = Z_N^{\top} Z_N$. Furthermore, define $\xi_N\in\mathbb{R}^N$ as a vector with $\epsilon_n$ in its rows. The following relations hold
\begin{align*}
    \norm{\frac{1}{N} \sum_{n} \rbr{x_{1,n}-\widehat{L}x_{2,n}}\epsilon_n}^2_{\Sigma/\Sigma_2(\widehat{L})^{-1} } = \norm{\Sigma/\Sigma_2(\widehat{L})^{-1/2}Z_N^{\top} \xi_N}_{2}.
\end{align*}
With this form of writing $\norm{\frac{1}{N} \sum_{n} \rbr{x_{1,n}-\widehat{L}x_{2,n}}\epsilon_n}^2_{\Sigma/\Sigma_2(\widehat{L})^{-1} }$, we see it can be bounded by applying \pref{lem: useful concentration for OLS}.
\end{proof}

\begin{lemma}[Bound on Second Term of \pref{prop: semiparameteric least model sample complexity}]\label{lem: ls nuisance 3 terms, term 2}
Let $\delta\in (0,e^{-1})$. Assume the following holds.
\begin{enumerate}
    \item $N\geq \gamma_{d,\delta}$.
    \item$\Sigma/\Sigma_2(\widehat{L}) $ is invertible and $\norm{\Sigma/\Sigma_2(\widehat{L})^{-1/2}\Sigma/\Sigma_2^{1/2}}_{\op}\leq \sqrt{2}.$
    \item $\norm{\Sigma_{2}^{1/2}((c_\star - \widehat{c}) + (\widehat{L}-L_\star)^{\top} w_\star)}\leq \Delta$.
\end{enumerate}
Then, with probability greater then $1-\delta$ it holds that
\begin{align*}
    &\norm{\frac{1}{N}\sum_{i=1}^N \rbr{x_{1,n}-L_\star x_{2,n}}\inner{(c_\star - \widehat{c}) + (\widehat{L}-L_\star)^{\top} w_\star}{x_{2,n}}}_{\Sigma/\Sigma_2(\widehat{L})^{-1} } \leq  \frac{5\Delta \gamma_{d,\delta/3}}{\sqrt{N}}.
\end{align*}
\end{lemma}
\begin{proof}
First, observe that the following relation hold
\begin{align}
    &\norm{\frac{1}{N}\sum_{i=1}^N \rbr{x_{1,n}-L_\star x_{2,n}}\inner{(c_\star - \widehat{c}) + (\widehat{L}-L_\star)^{\top} w_\star}{x_{2,n}}}_{\Sigma/\Sigma_2(\widehat{L})^{-1} } \nonumber \\
    &\leq \norm{\frac{1}{N}\sum_{i=1}^N \rbr{x_{1,n}-L_\star x_{2,n}}\inner{(c_\star - \widehat{c}) + (\widehat{L}-L_\star)^{\top} w_\star}{x_{2,n}}}_{\Sigma/\Sigma_2^{-1}}\norm{\Sigma/\Sigma_2(\widehat{L})^{-1/2}\Sigma/\Sigma_2^{1/2}}\nonumber \\
    &\leq \sqrt{2}\norm{\frac{1}{N}\sum_{i=1}^N \rbr{x_{1,n}-L_\star x_{2,n}}\inner{(c_\star - \widehat{c}) + (\widehat{L}-L_\star)^{\top} w_\star}{x_{2,n}}}_{\Sigma/\Sigma_2^{-1}}\tag{By assumption}\\
    &\leq \frac{\sqrt{2}}{N}\norm{\sum_{i=1}^N \rbr{x_{1,n}-L_\star x_{2,n}}\inner{(c_\star - \widehat{c}) + (\widehat{L}-L_\star)^{\top} w_\star}{x_{2,n}}}_{\Sigma/\Sigma_2^{-1}}\nonumber\\
    &=\frac{\sqrt{2}\Delta}{N}\norm{\sum_{i=1}^N \Sigma/\Sigma_2^{-1/2}\rbr{x_{1,n}-L_\star x_{2,n}}(\Sigma_{2}^{-1/2}x_{2,n})^{\top}}_{\op},\label{eq: second term nuisnace parameter rel 1}
\end{align}
where the last relation holds since $\norm{Ab}_2^2\leq \norm{b}_2^2\norm{A}_{\op}^2$, and since $\norm{\Sigma^{1/2}_2((c_\star - \widehat{c}) + (\widehat{L}-L_\star)^{\top} w_\star)}_2\leq \Delta$ by assumption. We now bound~\eqref{eq: second term nuisnace parameter rel 1}. 

Let $z_1 = \Sigma/\Sigma_2(x_{1,n}-L_\star x_{2,n})$ and $z_2=\Sigma_2^{-1/2} x_{2,n}$ and observe that $z_1\sim \Ncal(0,I_{d_w}),z_2\sim\Ncal(0,I_{d_e})$ and are independent random vectors since $\EE[z_1z_2^{\top}]=0$ (see~\eqref{eq: appendix orthogonal features}). Thus, by \pref{lem: bound on cross correlation}, for $N\geq \gamma_{d,\delta/3}$, the following bound holds with probability greater than $1-\delta$
\begin{align*}
   \eqref{eq: second term nuisnace parameter rel 1} =\frac{\sqrt{2}\Delta}{N}\norm{\sum_{i=1}^N z_{1,n}z_{2,n}^{\top}}_{\op} \leq \frac{5\Delta\gamma_{d,\delta/3}}{\sqrt{N}}.
\end{align*}
\end{proof}

\begin{lemma}[Bound on Third Term of \pref{prop: semiparameteric least model sample complexity}]\label{lem: ls nuisance 3 terms, term 3}
Let $\delta\in(0,e^{-1})$. Assume the following holds.
\begin{enumerate}
    \item $N\geq \gamma_{d,\delta/3}$
    \item$\Sigma/\Sigma_2(\widehat{L}) $ is invertible and $\norm{\Sigma/\Sigma_2(\widehat{L})^{-1/2}\Sigma/\Sigma_2^{1/2}}_{\op}\leq \sqrt{2}.$
    \item $\norm{\Sigma_{2}^{1/2}((c_\star - \widehat{c}) + (\widehat{L}-L_\star)^{\top} w_\star)}\leq \Delta$, and $\norm{\rbr{\Sigma / \Sigma_2}^{-1/2}(\widehat{L} - L_\star)\Sigma_2^{1/2} }_{\op}\leq \Delta_L$.
\end{enumerate}
Then, with probability greater than $1-\delta$, it holds that
\begin{align*}
    &\norm{\frac{1}{N}\sum_{i=1}^N \rbr{(\widehat{L}-L_\star )x_{2,n}}\inner{(c_\star - \widehat{c}) + (\widehat{L}-L_\star)^{\top} w_\star}{x_{2,n}} }_{\Sigma/\Sigma_2(\widehat{L})^{-1} } \leq \frac{5\Delta_L\Delta \gamma_{d_e,\delta/3}}{\sqrt{N}}  + \sqrt{2}\Delta_L\Delta.
\end{align*}
\end{lemma}
\begin{proof}
The following relations hold.
\begin{align*}
    &\norm{\frac{1}{N}\sum_{i=1}^N \rbr{(\widehat{L}-L_\star )x_{2,n}}\inner{(c_\star - \widehat{c}) + (\widehat{L}-L_\star)^{\top} w_\star}{x_{2,n}} }_{\Sigma/\Sigma_2(\widehat{L})^{-1} }\\
    &\leq \norm{\frac{1}{N}\sum_{i=1}^N \rbr{(\widehat{L}-L_\star )x_{2,n}}\inner{(c_\star - \widehat{c}) + (\widehat{L}-L_\star)^{\top} w_\star}{x_{2,n}} }_{\Sigma/\Sigma_2^{-1}}\norm{\Sigma/\Sigma_2(\widehat{L})^{-1/2}\Sigma/\Sigma_2^{1/2}}_{\op}\\
    &\leq \sqrt{2}\norm{\frac{1}{N}\sum_{i=1}^N \rbr{(\widehat{L}-L_\star )x_{2,n}}\inner{(c_\star - \widehat{c}) + (\widehat{L}-L_\star)^{\top} w_\star}{x_{2,n}} }_{\Sigma/\Sigma_2^{-1}}\tag{By assumption}\\
    &= \sqrt{2}\norm{\rbr{\Sigma/\Sigma_2^{-1/2}(\widehat{L}-L_\star )\Sigma_2^{1/2}}\frac{1}{N}\sum_{i=1}^N \rbr{\Sigma_2^{-1/2}x_{2,n}}\rbr{\Sigma_2^{-1/2}x_{2,n}}^{\top} \rbr{\Sigma_2^{1/2}(c_\star - \widehat{c}) + (\widehat{L}-L_\star)^{\top} w_\star}}_2\\
    &\leq  \sqrt{2}\Delta_L \Delta\norm{\sum_{i=1}^N \rbr{\Sigma_2^{-1/2}x_{2,n}}\rbr{\Sigma_2^{-1/2}x_{2,n}}^{\top}}_{\op}\\
    &\leq  \sqrt{2}\Delta_L \Delta\norm{\sum_{i=1}^N \rbr{\Sigma_2^{-1/2}x_{2,n}}\rbr{\Sigma_2^{-1/2}x_{2,n}}^{\top} - I_{d_e}}_{\op} +  \sqrt{2}\Delta_L \Delta\\
    &\leq \frac{ 5\Delta_L \Delta \gamma_{d_e,\delta/3}}{\sqrt{N}} +  \sqrt{2}\Delta_L \Delta.
\end{align*}
where the last relation holds with probability greater than $1-\delta$ by the concentration of the empirical covariance matrix (see \pref{lem: covariance estimation sub gaussian}) for $N\geq \gamma_{d,\delta}\geq \gamma_{d_e,\delta}$, while observing that $\Sigma^{-1/2}_2x_{2,n}\sim \Ncal(0,I_{d_e})$. 
\end{proof}

\begin{lemma}[Lower Order Bound on the Design Matrix of the Perturbed Least Square]\label{lem: lower bound on the perturbed least square}
Let $\norm{(L_\star - \widehat{L}) \Sigma_2^{1/2}}_{\op}\leq \Delta_L$. Furthermore,
\begin{align*}
    &\Sigma/\Sigma_{2} = \EE\sbr{(x_1 - L_\star x_2)(x_1 - L_\star x_2)^{\top}}\\
    &\Sigma/\Sigma_{2}(\widehat{L}) = \EE\sbr{(x_1 - \widehat{L} x_2)(x_1 - \widehat{L} x_2)^{\top}}.
\end{align*}
Then, the following relations hold.
\begin{enumerate}
    \item $\norm{\Sigma/\Sigma_{2} - \Sigma/\Sigma_{2}(\widehat{L})}_{\op} \leq \Delta^2_L$.
    \item $\lambda_{\min}(\Sigma/\Sigma_{2}(\widehat{L})) \geq \lambda_{\min}(\Sigma/\Sigma_{2}) - \Delta_L^2.$
    \item  Assuming $\rbr{\Delta_L^2/\lambda_{\min}(\Sigma/\Sigma_{2})}<1$ then $$\norm{\rbr{\Sigma/\Sigma_{2}(\widehat{L})}^{-1/2} \Sigma/\Sigma_{2}^{1/2}}_{\op} \leq \sqrt{(1- \rbr{\Delta_L^2/\lambda_{\min}(\Sigma/\Sigma_{2})})^{-1}}.$$
\end{enumerate}
\end{lemma}
\begin{proof}
{\bf First relation.} By definition,
\begin{align}
    &\Sigma/\Sigma_{2} - \Sigma/\Sigma_{2}(\widehat{L}) = \EE[(x_1 - L_\star x_2)(x_1 - L_\star x_2)^{\top}] - \EE[(x_1 - \widehat{L} x_2)(x_1 - \widehat{L} x_2)^{\top}] \nonumber \\
    &=\EE[(x_1 - L_\star x_2)(x_1 - L_\star x_2)^{\top}] - \EE[(x_1 - L_\star  x_2 - (\widehat{L}-L_\star) x_2)(x_1 -  L_\star  x_2 - (\widehat{L}-L_\star) x_2)^{\top}]\label{eq: lower bound perturbed design matrix rel 1}
\end{align}
Since $(a+b)(a+b)^{\top} = aa^{\top} + bb^{\top} + ab^{\top} + ba^{\top}$, defining $a=x_1 - L_\star  x_2 , b= - (\widehat{L}-L_\star) x_2$, we get that
\begin{align*}
    &\eqref{eq: lower bound perturbed design matrix rel 1} = -\EE[(x_1 - L_\star  x_2)((\widehat{L}-L_\star) x_2)^{\top}] -\EE[(\widehat{L}-L_\star) x_2)(x_1 - L_\star  x_2)^{\top}] + (\widehat{L}-L_\star)\EE[ x_2 x_2^{\top}](\widehat{L}-L_\star)^{\top}\\
    &=(\widehat{L}-L_\star)\Sigma_2(\widehat{L}-L_\star)^{\top},
\end{align*}
since $$\EE[(x_1 - L_\star  x_2)((\widehat{L}-L_\star) x_2)^{\top}]=\EE[\EE[(x_1 - L_\star  x_2)|x_2](\widehat{L}-L_\star) x_2)^{\top}]=\EE[(L_\star   - \widehat{L})  x_2((L_\star-L_\star) x_2)^{\top}]=0,$$
and, similarly, $\EE[((\widehat{L}-L_\star) x_2)(x_1 - L_\star  x_2)^{\top}   ]=0$. We get that
\begin{align*}
    &\norm{\Sigma/\Sigma_{2} - \Sigma/\Sigma_{2}(\widehat{L})}_{\op} =  \norm{(\widehat{L}-L_\star)\Sigma_{2}(\widehat{L}-L_\star)}_{\op}\\
    & =  \norm{((\widehat{L}-L_\star)\Sigma^{1/2}_2)((\widehat{L}-L_\star)\Sigma^{1/2}_2)^{\top}}_{\op}\\
    &\leq \norm{((\widehat{L}-L_\star)\Sigma^{1/2}_2)}_{\op}^2\leq \Delta^2_L.
\end{align*}

{\bf Second relation.} Direct application of Weyl's inequality.

{\bf Third relation.} Let $\Delta = \Sigma/\Sigma_{2}^{-1/2}\rbr{\Sigma/\Sigma_{2}(\widehat{L})-\Sigma/\Sigma_{2}}\Sigma/\Sigma_{2}^{-1/2}.$ Observe that
\begin{align*}
    &\norm{\Delta}_{\op} \leq \frac{1}{\lambda_{\min}(\Sigma/\Sigma_{2})} \norm{\Sigma/\Sigma_{2}(\widehat{L}) - \Sigma/\Sigma_{2}}_{\op} \tag{$\norm{\cdot}_{\op}$ is submultiplicative}\\
    &\leq  \frac{\Delta_L^2}{\lambda_{\min}(\Sigma/\Sigma_{2})} \tag{First relation of the lemma}.
\end{align*}
Thus, since $\norm{\Delta}_{\op} \leq \frac{\Delta_L^2}{\lambda_{\min}(\Sigma/\Sigma_{2})}<1$ by assumption, we can apply \pref{lem: hsu relative spectral norm error} and conclude the proof.
\end{proof}

\subsubsection{Least Square General Results and Tools}\label{app: least square general results}

\begin{theorem}[\cite{hsu2012tail}, Theorem 1]\label{thm: hsu subgaussia random design concentration}
Let $A\in \mathbb{R}^{m\times N}$ be a matrix, and let $K = A^{\top} A$. Suppose tat $\xi = (\epsilon_1,..,\epsilon_N)$ is a zero-mean and $\sigma$ sub-gaussian vector such that for some $\sigma\geq 0$ it holds that $\EE[\exp(v^{\top} \xi)]\leq \exp\rbr{\norm{v}^2_2\sigma^2/2}$ for all $v\in \mathbb{R}^N.$ Then, for any $t>0$,
\begin{align*}
    \Pr\rbr{\norm{A\xi}_2^2 \geq \sigma^2\rbr{\tr\rbr{K} +2\sqrt{\tr\rbr{K^2}t} +2\norm{K}t}} \leq \exp\rbr{-t}.
\end{align*}
\end{theorem}
The following lemma will be useful for our analysis. 
\begin{lemma}[Noise Concentration for the OLS]\label{lem: useful concentration for OLS}
Let $X_N\in \mathbb{R}^{N\times d}$ be a matrix with $\cbr{\frac{1}{\sqrt{N}} x_n}_{n=1}^N$ in its rows, $V_{N} = \frac{1}{N}\sum_{n=1}^N x_n x_n^{\top}= X_N^TX_N$ and $\xi_N\in \mathbb{R}^N$ be a matrix with $\cbr{\epsilon_n}_{n=1}^N$ in its rows. Furthermore, assume that $(i)$ $\epsilon_n$ and $x_n$ are independent, $(ii)$ $\norm{\Sigma^{-1/2} V_N^{1/2}}_{\op}^2\leq 3/2$, and $(iii)$ $\Sigma = \EE[x_nx_n^{\top}]$ is PD. Then, with probability greater than $1-\delta$ it holds that
\begin{align*}
    \frac{1}{\sqrt{N}}\norm{\Sigma^{-1/2} X_N^{\top}\xi_N}_2 \leq 3\sigma\sqrt{\frac{d\log\rbr{\frac{1}{\delta}}}{N}}.
\end{align*}
\end{lemma}
\begin{proof}
Let $P_{X_N}\in \mathbb{R}^{d\times N}$ be the projection on the column space of $X_N^{\top}$. Observe that, by definition, $X_N^{\top} = P_{X_N}X_N^{\top}$ and $ P_{X_N} = V_N^{1/2} (V_N^\dagger)^{1/2}$, and thus
\begin{align}
    &\eqref{eq: OLS analysis relation 1}=\frac{1}{\sqrt{N}} \norm{\Sigma^{-1/2}X_N^{\top}\xi_N}_2 \nonumber\\
    &= \frac{1}{\sqrt{N}} \norm{\Sigma^{-1/2} V_N^{1/2} (V_N^\dagger)^{1/2} X_N^{\top}\xi_N}_2 \nonumber \\
    &\leq \frac{1}{\sqrt{N}} \norm{\Sigma^{-1/2} V_N^{1/2}}_{\op} \norm{X_N^{\top}\xi_N}_{V_N^\dagger}.\label{eq: OLS analysis relation 2}
\end{align}
The first term is bounded by $3/2$ by assumption. We now bound the second term in~\eqref{eq: OLS analysis relation 2}. It holds that
\begin{align}
    &\norm{X_N^{\top}\xi_N }_{V_{N}^\dagger}^2  = \frac{1}{N} \xi_N^{\top} X_N(X_N^TX_N)^{\dagger}X_N^{\top}\xi_N. \label{eq: OLS analysis relation 2.5}
\end{align}
Furthermore, it can be verified that  $\tr\rbr{X_N(X_N^TX_N)^{\dagger}X_N^{\top}}\leq d$ and $\norm{X_N(X_N^TX_N)^{\dagger}X_N^{\top}}_\op\leq 1$. Hence, \pref{thm: hsu subgaussia random design concentration} of~\cite{hsu2012tail}, is applicable. Applying this result and assuming $\delta\in (0,e^{-1})$, we get  
\begin{align}
    \norm{X_N^{\top}\xi_N }_{V_{N}^\dagger} \leq 2\sigma\sqrt{\frac{d\log\rbr{\frac{1}{\delta}}}{N}}.\label{eq: general least square upper bound relation 2 -1}
\end{align}
with probability greater than $1-\delta$. This concludes the proof of the lemma.
\end{proof}

The following lemma allows us to translate performance of the OLS under the empirical design matrix to the performance under the expected empirical design matrix, i.e., the covariance matrix.
\begin{lemma}[Translating Empirical to Expected Performance]\label{lem: translating empirical to expected performance}
Let $w\in \mathbb{R}^d$ be a vector, $V_N = \frac{1}{N} \sum_{n=1}^N x_nx_n^{\top}$ and $\Sigma = \EE[xx^{\top}]$ be a PD matrix. Let $\Delta = \Sigma^{-1/2}V_N\Sigma^{-1/2} -I$ and assume that 
\begin{enumerate}
    \item $\norm{\Sigma^{-1/2}V_N w} \leq \epsilon$. 
    \item $\norm{\Sigma^{-1/2}V_N\Sigma^{-1/2} -I}_{\op} \leq c<1$.
\end{enumerate}
Then, it holds that $\norm{w}_{\Sigma} \leq \frac{\epsilon}{1-c}.$
\end{lemma}
\begin{proof}
We prove this result by standard analysis and by applying the assumptions.  The following relations hold.
\begin{align*}
    &\norm{w}_{\Sigma} = \norm{\Sigma^{1/2}w}_{2}\\
    &\leq \norm{\Delta\Sigma^{1/2}w}_{2} + \norm{\Sigma^{-1/2}V_N\Sigma^{-1/2}\Sigma^{1/2}w}_{2} \tag{Triangle inequality}\\
    &=\norm{\Delta\Sigma^{1/2}w}_{2} + \norm{\Sigma^{-1/2}V_N w}_{2}\\
    &\leq \norm{\Delta\Sigma^{1/2}w}_{2} + \epsilon \tag{By assumption}\\
    &\leq \norm{\Delta}_{\op}\norm{\Sigma^{1/2}w}_{2} + \epsilon \tag{Submultiplicative property of norm}\\
    &=c\norm{w}_{\Sigma} + \epsilon \tag{By assumption}.
\end{align*}
Rearranging yields the result.
\end{proof}

\begin{lemma}[Relative Spectral Norm Error,~\cite{hsu2012random}, Lemma 3]\label{lem: hsu relative spectral norm error}
Let $\Sigma_1,\Sigma_2$ be a PD matrices. Let $\Delta = \Sigma_1^{-1/2} \rbr{\Sigma_2 - \Sigma_1}\Sigma_1^{-1/2}.$ If $\norm{\Delta}<1$ then
\begin{align*}
    \norm{\Sigma_1^{1/2}\Sigma_2^{-1/2}}^2_{\op} = \norm{\Sigma_2^{-1/2}\Sigma_1^{1/2}}^2_{\op} = \norm{\Sigma_1^{1/2} \Sigma_2^{-1} \Sigma^{1/2}_1}_{\op} \leq \frac{1}{1-\norm{\Delta}_{\op}}.
\end{align*}
\end{lemma}
\begin{proof}
The first equality follows from the fact that $\norm{A}_{\op}=\lambda_{\max}(A^TA) = \lambda_{\max}(A^TA)=\norm{A^{\top}}_{\op}$. The second equality follows from the fact that $\norm{A^TA} = \lambda_{\max}((A^TA)^2)=\lambda_{\max}(A^TA)^2= \norm{A}^2$. The third inequality is proved in~\cite{hsu2012random}, Lemma 3.
\end{proof}

\begin{proposition}[Ordinary Least Squares: In-Distribution Error] \label{prop: regularized least square general result}
Let $\delta\in (0,e^{-1})$. Let $\cbr{x_n,y_n}_{n=1}^N$ be a data set where $x_n\in \mathbb{R}^d, y_n\in \mathbb{R}$, and assume that $y_n = \inner{x_n}{\beta_\star} +\epsilon_n$ where $\epsilon_n$ is zero mean and $\sigma$ sub-gaussian. Let $\Sigma= \EE[xx^{\top}]$. Let $\widehat{\beta}$ be the solution of the ordinary least square objective
\begin{align*}
    &\widehat{\beta}= V_N^{\dagger} X_N^{\top} Y_N
\end{align*}
Then, assuming that $N\geq 9\gamma_{d,\delta}^2$ where $\gamma_{d,\delta}$ is defined in \pref{lem: covariance estimation sub gaussian}, with probability greater than $1-\delta$
\begin{align*}
    \norm{\beta_\star - \widehat{\beta} }_{\Sigma} \leq 5\sigma\sqrt{\frac{d\log\rbr{\frac{2}{\delta}}}{N}}.
\end{align*}
and $\norm{\Sigma^{-1/2}V_N\Sigma^{-1/2} - I}_{\op} \leq \frac{1}{3}.$
\end{proposition}
\begin{proof}
Let $X_N\in \mathbb{R}^{N\times d}$ be a matrix with $\cbr{\frac{1}{\sqrt{N}} x_n}_{n=1}^N$ in its rows, $V_{N} = \frac{1}{N}\sum_{n=1}^N x_n x_n^{\top}= X_N^TX_N$ and $\xi_N\in \mathbb{R}^N$ be a matrix with $\cbr{\epsilon_n}_{n=1}^N$ in its rows.
Let $\widehat{\beta}$ be the OLS, i.e., it is the minimal norm solution that satisfies
\begin{align*}
    V_N (\beta_\star -\widehat{\beta}) =  \frac{1}{\sqrt{N}} X_N^{\top} \xi_N.
\end{align*}
Multiply both sides of the above equation by $\Sigma^{-1/2}$ (where $\Sigma$ is PD by assumption). We get that
\begin{align*}
    \Sigma^{-1/2} V_{N}  (\beta_\star -\widehat{\beta}) =\Sigma^{-1/2}X_N^{\top}\xi_N.
\end{align*}
Taking the norm of both sides and by the triangle inequality we get
\begin{align}
    \norm{ \Sigma^{-1/2} V_{N}(\beta_\star -\widehat{\beta})}_{2} = \frac{1}{\sqrt{N}} \norm{\Sigma^{-1/2}X_N^{\top}\xi_N}. \label{eq: OLS analysis relation 1}
\end{align}
We bound this term by applying \pref{lem: lower bound on the perturbed least square}. To apply this result we bound, with high probability, $\norm{\Sigma^{-1/2}V_N^{1/2}}^2_{\op}\leq \frac{3}{2}$ by the concentration of the empirical covariance matrix. It holds that
\begin{align}
    &\norm{\Sigma^{-1/2}V_N^{1/2}}^2_{\op} = \norm{\Sigma^{-1/2}V_N\Sigma^{-1/2}}_{\op} \nonumber\\
    &\leq 1+ \norm{\Sigma^{-1/2}V_N\Sigma^{-1/2} - I}_{\op} \nonumber\\
    & = 1 + \norm{\frac{1}{N}\sum_{n=1}^N (\Sigma^{-1/2}x_n)(\Sigma^{-1/2}x_n)^{\top} - I}_{\op} \nonumber \\
    &\leq 1+ \frac{\gamma_{d,\delta}}{\sqrt{N}},  \label{eq: first good event ols general concentration  1}
\end{align}
where the last relation holds with probability greater than $1-\delta$ by \pref{lem: covariance estimation sub gaussian} since $\Sigma^{-1/2}x_n\sim \Ncal(0,I_d)$. Thus, if $N\geq 9 \gamma_{d,\delta}^2$ it holds that $\norm{\Sigma^{-1/2}V_N^{1/2}}^2_{\op}\leq \frac{3}{2}$. Hence, conditioning on this event, \pref{lem: lower bound on the perturbed least square} is applicable. Thus, with probability greater than $1-2\delta$
\begin{align}
    \norm{\Sigma^{-1/2}V_N^{1/2}}^2_{\op} = \eqref{eq: OLS analysis relation 1} \leq   3\sigma\sqrt{\frac{d\log\rbr{\frac{1}{\delta}}}{N}}.\label{eq: first good event ols general concentration  2}
\end{align}
We now translate this bound to a bound w.r.t. $\norm{\beta_\star -\widehat{\beta}}_{\Sigma}.$ To do so, we apply \pref{lem: translating empirical to expected performance}. Observe that $(i)$ $\Sigma$ is PD by assumption, $(ii)$ conditioning on the good event $\norm{ \Sigma^{-1/2} V_{N}(\beta_\star -\widehat{\beta})}_{2} $ is bounded in~\eqref{eq: first good event ols general concentration  2}, and, $(iii)$ conditioning on the good event $\norm{\Sigma^{-1/2}V_N \Sigma^{-1/2}- I} \leq \frac{1}{3}$~\eqref{eq: first good event ols general concentration  1} and by the choice of $N$. Thus, by \pref{lem: translating empirical to expected performance} and setting $c=1/3$, we get that conditioning on the good event that holds with probability greater than $1-2\delta$, for $N\geq 9\gamma_{d,\delta}^2$,
\begin{align*}
    \norm{\beta_\star -\widehat{\beta}}_{\Sigma}\leq 5\sigma\sqrt{\frac{d\log\rbr{\frac{1}{\delta}}}{N}}.
\end{align*}
\end{proof}

\newpage

\section{Matrix Concentration Results}\label{supp: matrix concentration results survey}

\begin{lemma}[Covariance Estimation for Sub-Gaussian Distributions, Corollary 5.50,~\cite{vershynin2010introduction} and Remark 5.51] \label{lem: covariance estimation sub gaussian}
Let $\delta\in (0,e^{-1})$. Consider a sub-gaussian distribution in $\mathbb{R}^d$ with covariance  $\EE[xx^{\top}]=I$. Let $\Sigma_N = \frac{1}{N}\sum_{n=1}^N x_nx_n^{\top}$ be the empirical covariance matrix. Then, with probability greater then $1-\delta$ it holds that 
\begin{align*}
    \norm{\Sigma_N- I}_{\op} \leq \frac{\gamma_{d,\delta}}{\sqrt{N}}.
\end{align*}
for $$N \geq \gamma_{d,\delta} =  \sqrt{C d\log\rbr{\frac{1}{\delta}}}$$ where $C=C_K$ depends only on the sub-gaussian norm $K = \norm{x_i}_{\psi_2}$ and $C$ is an absolute constant if $x\sim \Ncal(0,I_D).$
\end{lemma}

\begin{lemma}\label{lem: bound on cross correlation}
Let $z_1\sim \Ncal(0,I_{d_1}),z_2 \sim \Ncal(0,I_{d_2})$ be independent random variables and $d=d_1+d_2$. Assume that $N\geq \gamma_{d,\delta/3}$ where $\gamma_{d,\delta}$ is defined in \pref{lem: covariance estimation sub gaussian}. Then, 
\begin{align*}
 \PP\rbr{ \norm{\frac{1}{N}\sum_{n=1}^N z_{1,n}z_{2,n}^{\top}}_{\op} \leq \frac{3\gamma_{d,\delta/3}}{\sqrt{N}}}\geq 1-\delta.
\end{align*}
\end{lemma}
\begin{proof}
Let $E_{n}=z_{1,n}z_{2,n}^{\top}$. Define the hermitian dilation of $E$ to be
\begin{align*}
    H_n 
    =
    \sbr{
    \begin{matrix}
     0 & E_n\\
     E_n^{\top}  &0
    \end{matrix}
    },
\end{align*}
and see that (e.g., \cite{tropp2012user}, section 2.6)
\begin{align}
    \norm{\frac{1}{N}\sum_n E_n}_{\op} = \lambda_{\max}\rbr{\frac{1}{N}\sum_{n} H_n}. \label{eq: dilation consequnce to bound sum of matrices}
\end{align}
Hence, instead of bounding the first we can bound the latter.

Let $H_{1,n}, H_{2,n}, H_{f,n}\in \mathbb{R}^{d \times d}$ be defined as follows,
\begin{align*}
    H_{1,n} = 
    \sbr{
    \begin{matrix}
     z_{1,n}z_{1,n}^{\top} & 0\\
    0  &0
    \end{matrix}
    }, 
    \quad
    H_{2,n} = 
    \sbr{
    \begin{matrix}
     0 & 0\\
    0  & z_{2,n}z_{2,n}^{\top}
    \end{matrix}
    },
    \quad 
    H_{f,n} = \sbr{
    \begin{matrix}
      z_{1,n}z_{1,n}^{\top} & z_{1,n}z_{2,n}^{\top}\\
    z_{2,n}z_{1,n}^{\top}  & z_{2,n}z_{2,n}^{\top}
    \end{matrix}
    }.
\end{align*}
With these definitions  we get that $$H_n = H_{f,n} -  H_{1,n} - H_{2,n}.$$ Furthermore, since $\EE[H_n]=0$ due to the independence of $z_1$ and $z_2$, and since they are assumed to be zero mean, we get that
\begin{align*}
   &\frac{1}{N}\sum_{n} H_n = \frac{1}{N}\sum_{n} H_{f,n} -  H_{1,n} - H_{2,n}\\
   &=\frac{1}{N}\sum_{n} H_{f,n} - \EE[H_{f,n}] -  \frac{1}{N}\sum_{n}H_{1,n}-\EE[H_{1,n}] - \frac{1}{N}\sum_n H_{2,n}-\EE[H_{1,n}]. \tag{Since $\EE[H_n]=\EE[ H_{f,n} -  H_{1,n} - H_{2,n}]=0$}
\end{align*}
Hence, we can bound $\lambda_{\max}\rbr{\frac{1}{N}\sum_n H_n}$ by the following sum
\begin{align}
     &\lambda_{\max}\rbr{\frac{1}{N}\sum_{n} H_n} \\
     &\leq  \lambda_{\max}\rbr{\frac{1}{N}\sum_{n} H_{f,n} - \EE[H_{f,n}]} +  \lambda_{\max}\rbr{\frac{1}{N}\sum_{n}H_{1,n}-\EE[H_{1,n}]} + \lambda_{\max}\rbr{\frac{1}{N}\sum_{n}H_{2,n}-\EE[H_{1,n}]} \label{eq: cross correlation relation 2}
\end{align}
since $\lambda_{\max}(A+B)\leq \lambda_{\max}(A) +  \lambda_{\max}(B)$\footnote{E.g., by using the variational form of maximal eigenvalue and since $\max_{x: \norm{x}_2=1} (x^TA x + x^TBx)\leq \max_{x: \norm{x}_2=1} x^TAx + \max_{x: \norm{x}_2=1} x^TBx$}.

Observe that each one of the three summands is the deviation of the empirical covariance from its average. Applying \pref{lem: covariance estimation sub gaussian} and by applying the union bound we get that with probability greater than $1-3\delta$, and assuming that $N\geq \frac{\gamma_{d,\delta}}{\sqrt{N}}$ 
\begin{align*}
    &\eqref{eq: cross correlation relation 2}\leq \rbr{\norm{\Sigma}\gamma_{d,\delta} +\norm{\Sigma_1}\gamma_{d_1,\delta} + \norm{\Sigma_2}\gamma_{d_2,\delta}}/\sqrt{N} \leq 3\norm{\Sigma}\gamma_{d,\delta}/\sqrt{N} = 3\gamma_{d,\delta}/\sqrt{N} ,
\end{align*}
where the last relation holds since $\Sigma_1,\Sigma_2\leq \Sigma=I$ and since $\gamma_{d,\delta}$ is increasing in $d.$ Finally, setting $\delta\gets \delta/3$ yields the result.
\end{proof}

\newpage
\section{Experiment Details}
\label{app: appendix-experiments}

In this section, we complete the details for the experimental setup outlined in \pref{sec: experiments}.

\paragraph{Synthetic PC-LQRs.} We constructed a family of PC-LQR problems, parameterized by $(s_c, s_e, d_u, d)$. The diagonal blocks $A_1 \in \RR^{s_c \times s_c}$, $A_2 \in \RR^{s_e \times s_e}$, $A_3 \in \RR^{d-s_c-s_e, d-s_c-s_e}$ were generated by sampling each entry from $\Ncal(0, 1)$, dividing by the spectral radius (i.e. the largest modulus of complex eigenvalues), then multiplying by the desired spectral radius. We set $\rho(A_1) = 1$, to make the controllable part of the system marginally stable, and set $\rho(A_2) = \rho(A_3) = 0.9$. The matrices $A_{12}, A_{32}$, and $B_1$ were obtained by sampling each entry from $\Ncal(0, 1)$. Finally, for the LQR cost matrices, we selected $Q = I_{1+}$ and $R = I_{d_u}$.

\paragraph{System identification.} Two system identification methods for estimating $A$ were compared: ordinary least squares regression from $x_1$ onto $x_0$ (with the least-Frobenius norm solution), and the soft-thresholded semiparametric least squares estimator from \pref{alg: control pc lqr samples}, with a choice of $\epsilon = 0.1$.\footnote{With these synthetic systems, \pref{alg: control pc lqr samples} performed similarly with the thresholded OLS estimator.} Fixing a sample size $N$, we sampled all $x_0 \sim \Ncal(0, I)$ i.i.d., and $x_1 = Ax_0 + \eta_0$, where $\eta_0 \sim \Ncal(0, I)$.

\paragraph{Certainty-equivalent control.} We plugged these $(\hat A, B)$ into SciPy's discrete algebraic Riccati equation solver, which outputs the fixed-solution solution $P_\star$ under the nominal dynamics; then, the LQR cost of the derived controller on the true system was measured; if this was finite and within a factor of $1.1$ of the optimal cost on the true dynamics, we called this trial (indexed by an independent sample) a success: the learned controller stabilized this marginally stable system.

We varied the sample size $N$ between $100$ and $1000$ in increments of 20, and varied $d \in \{20, 50, 100, 150\}$, fixing $s_c = s_e = 5, d_u = 1$. \pref{fig:synthetic-plot} shows the fraction of successful trials over 100 repetitions; error bars show normal approximation-derived standard deviations. All experiments took around 2 hours on a single 2.3 GHz Intel i7 CPU machine.


\end{document}